\documentclass{amsart}
\usepackage[utf8]{inputenc}

\usepackage{amsmath,amsthm,verbatim,amssymb,amsfonts,amscd,graphicx,calc}
\usepackage{xcolor}
\usepackage{graphics}
\usepackage{euscript, enumerate,hhline,array}
\usepackage{url,hyperref}

\newcommand{\p}{\partial}

\newcommand{\la}{\langle}
\newcommand{\ra}{\rangle}

\newcommand{\eae}{\zeta}

\newcommand{\eps}{\varepsilon}

\newcommand{\be}{\begin{equation}}
\newcommand{\ba}{\begin{aligned}}
\newcommand{\bee}{\begin{equation*}}
\newcommand{\ee}{\end{equation}}
\newcommand{\ea}{\end{aligned}}
\newcommand{\eee}{\end{equation*}}
\newcommand{\bea}{\begin{equation} \begin{aligned} }
\newcommand{\eea}{\end{aligned}\end{equation} }

\newcommand{\sgamma}{\bar \gamma}
\theoremstyle{plain}
\newtheorem{theorem}{Theorem}[section]

\newtheorem{corollary}[theorem]{Corollary}

\newtheorem{lemma}[theorem]{Lemma}
\newtheorem{proposition}[theorem]{Proposition}
\newtheorem{prop}[theorem]{Proposition}

\newtheorem{claim}{Claim}[section]

\newtheorem{remark}[theorem]{Remark}

\theoremstyle{definition}
\newtheorem{definition}[theorem]{Definition}

\numberwithin{equation}{section}

\title{continuous family of surfaces translating by powers of Gauss curvature}

\author{Beomjun Choi}

\address{BC: Department of Mathematics, POSTECH, 77 Cheongam-ro, Nam-gu, Pohang, Gyeongbuk 37673, Republic of Korea}
\email{bchoi@postech.ac.kr}
\author{Kyeongsu Choi}
\address{KC: School of Mathematics, Korea Institute for Advanced Study, 85 Hoegiro, Dongdaemun-gu, Seoul 02455, Republic of Korea}
\email{choiks@kias.re.kr}

\author{Soojung Kim}
\address{SK: Department of Mathematics, Soongsil University,  Seoul 06978,    Republic of Korea}
\email{soojungkim@ssu.ac.kr}

\date{}

\begin{document}

\begin{abstract}
This paper shows the existence of convex translating surfaces under the flow by the $\alpha$-th power of Gauss curvature for the sub-affine-critical regime $ 0 < \alpha < 1/4$. The key aspect of our study is that our ansatz at infinity is the graph of homogeneous functions whose level sets are closed curves shrinking under the flow by the $\frac{\alpha}{1-\alpha}$-th power of curvature. For each ansatz, we construct a family of translating surfaces generated by the Jacobi fields with effective growth rates. Moreover, the construction shows quantitative estimate on the rate of convergence between different translators to each other, which is required to show the continuity of the family. As a result, the family is regarded as a topological manifold. The construction in this paper will become the ground of forthcoming research, where we aim to prove that every translating surface must correspond to one of the solutions  obtained herein,   classifying translating surfaces and identifying the topology of the moduli space.
\end{abstract}
 
\maketitle 
\tableofcontents

\section{Introduction}

The $\alpha$-Gauss curvature flow ($\alpha$-GCF) is the evolution of complete convex hypersurfaces under the speed of $\alpha$-th power of Gaussian curvature in the inner-pointing\footnote{The convex hull of the image surface $\Sigma_t$ is the inside region when $\Sigma_t$ is non-compact.} normal direction.
 A convex hypersurface $\Sigma$ is a translating soliton (or simply a translator) under the $\alpha$-GCF if $\Sigma_t= \Sigma+vt $ is an $\alpha$-GCF for some $v\in \mathbb{R}^3$. The translator $\Sigma$ satisfies the equation \be \label{eq-translator1}  K ^\alpha = \la -\nu, v \ra ,\ee 
 at every point $X$ on the hypersurface $\Sigma$.  Here, $K$ and $\nu$ denote the Gauss curvature and the unit outward-pointing  normal vector  at  $X$, respectively. 
 Throughout this paper, let us fix $v=\mathbf{e}_3$. If a translator is an entire smooth   convex graph $x_3=u(x_1,x_2)$, the translator equation for $u(x)$ becomes  \be \label{eq-translatorgraph}\det D^2 u = (1+|Du|^2)^{2 -\frac{1}{2\alpha}}.\ee

The main result of this paper addresses the existence of smooth entire convex graphical translators for sub-affine-critical powers $\alpha <1/4$. Note that not all translators in all ranges of $\alpha>0$ are entire strictly convex graphs \cite{Urbas1988, U98GCFsoliton}.  In forthcoming paper(s), we will show that every complete convex hypersurface that satisfies the translator equation \eqref{eq-translator1} in a generalized sense, for some $\alpha\in(0,1/4)$, has to be the graph of entire smooth strictly convex function. 

The equation for translating solitons under the $\alpha$-GCF can be written for higher dimensional space. In $\mathbb{R}^{n+1}$, a graphical equation for  $x_{n+1}=u(x)$ is \begin{equation}\label{eq:u_in_R^n}
\det D^2 u=(1+|Du|^2)^{\frac{n+2}{2}-\frac{1}{2\alpha}}.
\end{equation} 
The existence and classification for solutions to Monge-Ampere equations have been studied for decades and the most promising result is established for the affine-critical case $\alpha=\frac{1}{n+2}$, i.e., the equation with a  homogeneous constant right-hand side $\det D^2 u = 1$  on $\mathbb{R}^n$ for all dimensions $n\ge1$. By the celebrated result of J\"orgens, Calabi, and Porgorelov, the entire convex solutions are quadratic polynomials \cite{nitsche1956elementary, jorgens1954losungen, calabi1958improper, cheng1986complete, caffarelli1995topics}. 
  There are extensions of this result by Caffarelli--Y. Li \cite{caffarelli2003extension} to the case    when the right-hand side is constant outside of a compact set, and by Y. Li--S. Lu  \cite{caffarelli2004liouville,li2019monge} to the case  with the periodic    right-hand side.

Urbas \cite{Urbas1988, U98GCFsoliton} showed that if $\alpha>\frac{1}{2}$, each translator has to be a graph over a bounded convex domain $\Omega \subset \mathbb{R}^{n}$. Conversely, given a convex open bounded   domain $\Omega \subset \mathbb{R}^n$ and $\alpha>\frac{1}{2}$, there exists a unique translator $\Sigma$ under the $\alpha$-GCF which is asymptotic to $\partial\Omega \times \mathbb{R}$. Moreover, he proved that $\Sigma$ is strictly convex and smooth in $\Omega\times \mathbb{R}$. In general, the translators attract nearby non-compact flows; see \cite{choi2018convergence, choi2019convergence}. Moreover, the translator is the only non-compact ancient solution in each $\Omega\times \mathbb{R}$; see \cite{choi2020uniqueness}.

For $\alpha\leq \frac{1}{2}$ and $\alpha \neq \frac{1}{n+2}$, to our knowledge, no complete classification is known. In \cite{JW}, H. Jian--X.J. Wang    constructed infinitely many solutions for $\alpha\leq \frac{1}{2}$, and also showed that  $\Omega$ should be the entire space $\mathbb{R}^n$ for $\alpha <\frac{1}{n+1}$. To be specific, they showed that given an ellipsoid $E\subset\mathbb{R}^n$,  there exists a solution $u$  satisfying $u(0)=0$ and $Du(0)=0$, such that $E$ is homothetic to the John ellipsoid of the level set $\{x\in \mathbb{R}^n:u^*(x)=1\}$, where $u^*$ is the Legendre transformation of $u$; c.f. K.S. Chou--X.J. Wang \cite{chou1996entire}.

\bigskip

The asymptotic behavior of  translators (near infinity) is closely tied with the blow-down equation 
\begin{equation}\label{eq-asymptranslatorgraph}
\det D^2 \bar u = |D\bar u|^{4-\frac1{\alpha}}.
\end{equation}
More precisely, we construct translators that are asymptotic to self-similar  solutions to the blow-down equation having a form $\bar u(rx)= r^d \bar u(x)$ for some $d>1$.  By a direct computation, the degree $d$ must be $(1-2\alpha)^{-1}$,  i.e., a self-similar solution can be written as $\bar u (x)= |x|^{\frac{1}{1-2\alpha}} g(\theta)$ for some $g(\theta)$, where $\theta$ denotes the angle between the positive $x_1$-axis and $x=(x_1,x_2)$. Let us also call self-similar solutions by homogeneous solutions. This prompts us to define the following notion.
\begin{definition} [Blow-down] \label{def-blowdown} We say that the blow-down of translator $x_3=u(x)$ is $\bar u(x)$ if $u_\lambda (x):= \lambda ^{-\frac{1}{1-2\alpha}} u(\lambda x)$ converges locally uniformly to $\bar u(x)$ as $\lambda \to \infty$. If $\bar u $ is the rotationally symmetric homogeneous solution to \eqref{eq-asymptranslatorgraph}, then the blow-down is called rotationally symmetric. If $\bar u$ is a (non-radially symmetric) $k$-fold symmetric homogeneous solution to \eqref{eq-asymptranslatorgraph}, then the blow-down is called $k$-fold symmetric.   See Definition \ref{def:k-fold}  for the definition of $k$-fold symmetry.
\end{definition}
 Let $\bar u$ be the blow-down of $u$. Then $u$ has the asymptotics   $|u(x)- \bar u(x)|=o ( |x|^{\frac{1}{1-2\alpha}})$ as $x\to \infty$. By a formal computation, one easily checks that the blow-down $\bar u$ of the translator $x_3=u(x)$ has to be a solution to the blow-down equation if $u_\lambda$ converges smoothly to $\bar u$.

  The crucial observation is that the level curve of homogeneous blow-down $\bar u$  must be a shrinking soliton (or simply shrinker) under the curve shortening flow of speed $\kappa^\frac{\alpha}{1-\alpha}$; see Remark \ref{remark-shrinker}.  Different ranges of $\alpha$ can accommodate different types of shrinkers. There has been the classification of shrinkers by B. Andrews \cite{andrews2003classification}, which is re-stated in our Theorem \ref{thm-andrews2003}.  See also \cite[Section 5]{DSavin} by Daskalopoulos-Savin, where homogeneous solutions to the dual equation $\det D^2 u = |x|^{\frac1\alpha-4}$ have been studied for $0<\alpha<1/4$ and $1/4<\alpha  <1/2$. 
Moreover, we should mention that the dimension of effective Jacobi fields to the blow-down equation, generating a family of translators asymptotic to a given blow-down, varies  with respect to $\alpha$ and the limit shrinker. Now we state the main existence theorem. See also Theorem \ref{thm-existence} for the precise version, stated after required preliminaries are introduced.
\begin{theorem}[Existence]\label{thm:classification-1} For $\alpha\in(0, 1/4)$, let  $m:=\lceil \alpha^{-1/2}\rceil -1 $.\footnote{Recalling the definition $\lceil x\rceil :=\min \{n\in \mathbb{Z}\,:\, x\le n \}$, $m$ is the unique integer such that $\frac{1}{(m+1)^2}\le \alpha < \frac{1}{m^2}$.  
} There exists a  $(2m+1)$-parameter family of smooth translators under the $\alpha$-Gauss curvature flow  whose blow-downs are the rotationally symmetric homogeneous solution to \eqref{eq-asymptranslatorgraph}. In addition, for each integer $k\in [3,m]$ (which necessarily implies $\alpha<1/9$), there exists a ${(2k-1)}$-parameter family of smooth translators whose blow-downs are a given $k$-fold symmetric homogeneous solution to \eqref{eq-asymptranslatorgraph}. {Moreover, the parametrization is continuous in the topology of locally uniform convergence of surfaces.}
\end{theorem}

\begin{table}[ht] 
  \centering{\renewcommand{\arraystretch}{2.5}
\begin{tabular}{ |c|c|c|c|c| c} 
 \hline 
  & \includegraphics[scale=.7]{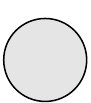}& \includegraphics[scale=.67]{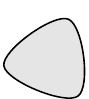} & \includegraphics[scale=.7]{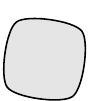} & \includegraphics[scale=.7]{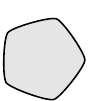}  & 
$\ba \cdots \\ \\\ea $ \\ 
\hhline{|=|=|=|=|=|=} 
 $\frac 19 \le \alpha <\frac 14$ & 5 & X & X  & X & $\cdots$  \\ 
\hline
 $\frac 1{16} \le \alpha <\frac 19$ & 7 & 5 & X & X & $\cdots$ \\ 
 \hline
  $\frac 1{25} \le \alpha <\frac 1{16}$ & 9 & 5 &7 & X& $\cdots$  \\ 
 \hline
  $\frac 1{36} \le \alpha <\frac 1{25}$ & 11 & 5 & 7 & 9 & $\cdots$ \\ 
 \hline
 $\vdots$ & $\vdots$ & $\vdots$ & $\vdots$ & $\vdots$ & $\ddots $ \\ 
\end{tabular}
}
\vspace{.4cm}

\caption{the number of parameters for fixed blow-down homogeneous solutions}
\label{table1}
\end{table}

\begin{remark} [Rigid motion]
 
\begin{enumerate}[(a)]
\item  For each family in Theorem \ref{thm:classification-1}, three parameters always correspond to   the translation in $\mathbb{R}^3$: see (iii) in Theorem \ref{thm-existence}. 

\item  In the case of non-radial blow-downs, the family is presented for a fixed blow-down and hence the rotation about the  $x_3$-axis can not be counted: if one rotates a translator about the $x_3$-axis, then the blow-down also changes. 
  In the case of the radial blow-down, a rotation about the $x_3$-axis provides another solution having the same blow-down. Our construction does not explicitly count the rotation as one of parameters. Nevertheless, in the construction, the deviation of translators from the radial blow-down does not prefer a particular direction, and this suggests   rotated copies of a given solution are also contained in the family.  This point will become clear once we show the family contains all possible solutions in forthcoming research\footnote{The first preprint of current research \cite{choi2021translating}, posted in 2021, includes an existence result weaker than Theorem \ref{thm:classification-1} and a classification theorem. The preprint, however, has an error in the classification part. In this paper we establish the improved existence Theorem \ref{thm:classification-1} which entails  the quantitative estimate in Corollary \ref{cor-wexpression}. They will play a crucial role in the corrected classification proof that will be updated in the new version of \cite{choi2021translating}.}, and hence the family itself is invariant under rotations.  
\end{enumerate}

\end{remark}

A Jacobi field of a given blow-down refers a solution to the linearized equation about the blow-down. This represents a candidate for an infinitesimal deformation in the space of translating solitons.  If we write translators as graphs (a.k.a. fixing a gauge), then 
a Jacobi field that can be written as a separation of variables is effective  when  the growth rate $\beta$ in $|x|^\beta$ respectively belongs to $ (0, \frac{1}{1-2\alpha}] $  and  $ (0, \frac{1}{1-2\alpha}) $   in the case of a non-radial blow-down and  in the radial blow-down.  
Note that $\frac{1}{1-2\alpha}$ is the growth rate of given translator. Heuristically,  the lower bound of the growth rate $\beta$ is needed as any two graphical translator that converges to each other near infinity have to be identical by the comparison principle, and the upper bound condition is imposed as a Jacobi field which grows strictly faster than the solution itself will not generate a solution asymptotic to the given one. It is a delicate question if a Jacobi field that has the same growth rate of the given blow-down can generate a solution. Indeed, non-rotational blow-downs always have one-dimensional such Jacobi fields generated by rotations with respect to $x_3$-axis. On the other hand, the radial blow-down have such Jacobi fields for certain $\alpha=\frac{1}{m^2}$, $m=3,4,\ldots$, and in the forthcoming research we show that there is no translator generated by those Jacobi fields. Since the existence of translators in this paper is shown for each fixed blow-down, non-radial blow-downs are not allowed to rotate and hence we will only count Jacobi fields whose growth rates are strictly less than $\frac{1}{1-2\alpha}$ both radial and non-radial cases and prove the result.

\bigskip

Merely constructing a family of surfaces may provide no information regarding how the surfaces change as the parameter of the family changes: the surfaces can jump discontinuously with respect to the variation of the parameter. Indeed, with a proper choice of topology in the space of surfaces, our construction shows the continuity of the parametrization, and hence the family forms a topological manifold. Although there is no distinction on the regularities of convergence due to elliptic regularity theory, one has to choose the topology of locally uniform convergence rather than the one induced by the global convergence. Indeed, the (Hausdorff) distance between two translating surfaces is infinite unless they are identical upto a translation in space. This is foreseen by the fact that an effective Jacobi field grows near infinity unless it is generated by a translation.

To achieve this objective,  significant effort has been devoted to prove quantitative estimates on solutions near infinity, which are necessary for the delicate continuity statement in Theorem \ref{thm:classification-1} and for the future research. More precisely, in this work,  we express a surface by the support function of its level set in the plane $x_3=\ell$, denoted by $S(\ell,\theta)$. In this expression, the blow-down homogeneous function has a form $\sigma^{-1} \ell ^{\sigma} h(\theta)$, where $\sigma=1-2\alpha$ and $h(\theta)$ is the support function of a shrinking curve. We construct a family of translators $\{S_{\mathbf{a}}(\ell,\theta)\}_{\mathbf{a}\in \mathbb{R}^K}$, which have the asymptotic expansion $S_{\mathbf{a}}(\ell,\theta) = \sigma^{-1} \ell^\sigma h(\theta)+ o(\ell^{\sigma})$. Here, $K$ denotes the dimension of effective Jacobi fields, i.e., the number of parameters. Natural questions are whether (i) one has better estimates on $o(\ell ^\sigma)$. e.g. an improvement on the rate $\sigma$, (ii) there is a quantitative estimate revealing how large $\ell$ should be to have a desired smallness of this error, and (iii) such previously mentioned estimates can be made uniform in the parameter $\mathbf{a}\in \mathbb{R}^K$. Note it is unlikely that one has an affirmative answer to (iii).  For instance, let $\Sigma \subset \mathbb{R}^3$ be a given translating surface; when its translated solutions $\Sigma+s\vec v$ are considered for a non-zero $\vec v \in \mathbb{R}^3$ and large $s<\infty$, no uniform estimate is expected to hold for such a family. We may only expect estimates that are locally uniform in parameters.

Corollary \ref{cor-wexpression} answers  (i) and (ii) in the way that is locally uniform on the parameter. Roughly, we show   the rate is $\sigma-\eps$ with some $\eps = \eps(h,\alpha)>0$ and 
\be \label{eq-ratetoblowdown} |S_{\mathbf{a}}(\ell,\theta)-\sigma^{-1} \ell ^\sigma h(\theta)| \le  C \ell ^{\sigma -\eps} 
\ee 
for $\ell \ge M (1+ |\mathbf{b}(\mathbf{a})|)$.   Here, $|\mathbf{b}(\mathbf{a})|$ is the effective modulus of parameter in Definition \ref{def-b(a)},  $C=C(h,\alpha, |\mathbf{a}|)$, and $M=M(h,\alpha)$. This would be an optimal as it is the estimate anticipated by the behavior of Jacobi fields. Next, we also obtain an estimate on the difference between two translators. i.e., we estimate the rate of convergence of one translator to another translator in terms of the difference in their parameters. In Proposition \ref{thm-modulicont-barr} and in the later part of Section \ref{sec-thm-existence}, again roughly speaking, we show that 
\be \label{eq-ratediffer}|S_{\mathbf{a}_1}(\ell,\theta) -S_{\mathbf{a}_2}(\ell,\theta)| \le C |\mathbf{a}_1-\mathbf{a}_2|\ell ^{\sigma-\eps} + o(\ell ^{-2\alpha}) 
\ee 
for $\ell \ge C(1+\max_{i=1,2} |\mathbf{b}(\mathbf{a}_i)|)$  with some $C=C(\alpha,h, \max_{i=1,2} |\mathbf{a}_i|)<\infty.$ In Proposition \ref{thm-modulicont-barr}, it was shown for each family of global barriers $\{S_{\mathbf{a},+}\}$, and  $\{S_{\mathbf{a},-}\}$, respectively,  the estimate has no extra $o(\ell^{-2\alpha})$ error. The error appears since $|S_{\mathbf{a},+}-S_{\mathbf{a},-}|=o(\ell ^{-2\alpha})$ and the solution $S_{\mathbf{a}}$ sits between two global barriers. Indeed, one may consider the error $o(\ell ^{-2\alpha})$ as negligible: global barriers uniquely determine a translating surface in sitting between them. A difference of order $o(\ell ^{-2\alpha})$ makes no difference among translators as  by applying the comparison principle directly, they must be identical (c.f. Lemma \ref{lem-uniqueness}).

The estimates \eqref{eq-ratetoblowdown} and \eqref{eq-ratediffer} are pivotal for the topological assertions.  The estimate \eqref{eq-ratediffer} is a key component in the continuity statement of our main theorem. In the future research, we aim to classify translating surfaces by establishing that each translator possesses a unique blow-down and it is one of the solutions constructed in this work. Once the classification is resolved, the previous continuity implies the continuity of the bijection from $\mathbb{R}^K$ to the moduli space. Corollary \ref{cor-wexpression} will provides a delicate compactness theorem for translating surfaces with diverging parameters (once they are appropriately rescaled), which will be necessary for ensuring the continuity of the bijection in reverse and thus verifies the topology of moduli space.

%{\color{red}NEED TO WORK BELOW} 
%
%\begin{theorem}[Rigidity]\label{thm:rigidity}
%Suppose that $\bar u\in C^\infty(\mathbb{R}^2\setminus \{0\})\cap C^1(\mathbb{R}^2)$ is a convex solution to \eqref{eq-asymptranslatorgraph} satisfying $\bar u(0)=0$ and $\nabla \bar u(0)=0$. Then, $\bar u$ is asymptotically round at infinity unless $\bar u$ is a $k$-fold symmetric homogeneous function.
%\end{theorem}

%%%%%%%%%%%%%%%%%%%%%%%%%%%%%%%%%%%%%%%%%%%%%%%%%%%%%%%%%%%%%%%%%%%%%%%%%%%%%%%%%%%%%%%%%%%%%%%%%%%%
\section{Preliminaries and main result}\label{sec-prelim}
%%%%%%%%%%%%%%%%%%%%%%%%%%%%%%%%%%%%%%%%%%%%%%%%%%%%%%%%%%%%%%%%%%%%%%%%%%%%%%%%%%%%%%%%%%%%%%%%%%%%

 %%%%%%%%%%%%%%%%%%%%%%%%%%%%%%%%%%%%%%%%%%%%%%%%

 \subsection{Change of coordinates} \label{section-2.2}
  %%%%%%%%%%%%%%%%%%%%%%%%%%%%%%%%%%%%%%%%%%%%%%%%
 %%%%%%%%%%%%%%%%%%%%%%%%%%%%%%%%%%%%%%%%%%%%%%%%

   Given a translator $\Sigma$ which is a graph of strictly convex entire function $u(x)$ on $\mathbb{R}^2$, suppose $\inf_{\mathbb{R}^2} u>-\infty$ and hence $\{x\,:\, u(x)=l\}$ is a convex closed curve for all $l >\inf_{\mathbb{R}^2} u $. We   define the support function of the level convex curve $\{u(x)=l\}\subset \mathbb{R}^2$  {at height $l$}  as follows: for $(l,\theta)\in  \{l\in \mathbb{R}: l >\inf_{\mathbb{R}^2} u\}\times \mathbb{S}^1$, 
  \bea \label{eq-supportfunction} 
  S(l,\theta) := \sup_{\{u(x)=l\}} \la x, (\cos \theta, \sin \theta) \ra  .
  \eea  
When no confusion is possible, we simply call $S(l,\theta)$ the support function of a given convex hypersurface in this paper. 
  Then the translator equation  \eqref{eq-translatorgraph} translates into the following equation in terms of $S(l,\theta)$.
  
  \begin{prop}\label{prop-supportfunction} Let  $u$ be a  smooth strictly convex function on $\mathbb{R}^2$ with $\inf_{\mathbb{R}^2} u >-\infty$. 
  \begin{enumerate}[(a)]
 \item If $u$ solves the translator equation \eqref{eq-translatorgraph}, then 
  $S=S(l,\theta)$  solves  the equation
\begin{equation}\label{eq-translatorsupport}
 S_{\theta\theta}+S+(1+S_l^{-2})^{\frac1{2\alpha}-2}S_l^{-4}S_{ll} =0 ,
\end{equation}  
   on  $\{l\in \mathbb{R}: l >\inf_{\mathbb{R}^2} u\}\times \mathbb{S}^1$. 
   
\item     If $  u$ solves the blow-down equation \eqref{eq-asymptranslatorgraph}, then the corresponding  $S=S(l,\theta)$   solves  \bea \label{eq-asymptranslatorsupport}
 S_{\theta\theta}+S +S_{ll}S_l^{-\frac1\alpha} =0,
 \eea  on the same region.
\end{enumerate}
\begin{proof}See Appendix \ref{section-appendix-1}.\end{proof}
  \end{prop} It is convenient to consider the equations of translators \eqref{eq-translatorsupport} and blow-downs  \eqref{eq-asymptranslatorsupport} together as many results can be dealt with for both at once. So, let us introduce a new parameter $\eta\in \{0,1\}$ and write both equations as 
\begin{equation}\label{Se} 
S+S_{\theta\theta}+(\eta   +S_l^{-2})^{\frac1{2\alpha}-2}S_l^{-4}S_{ll} =0.
\end{equation}
 
 \begin{remark}\label{remark-shrinker}
Let $\bar u$ be    a  $ {\sigma}^{-1}$-homogeneous  solution  to the blow-down equation \eqref{eq-asymptranslatorgraph}, i.e., $\bar u(\lambda x)=  \lambda ^{\frac{1}\sigma} \bar u(x)$ with  $ \sigma:=1-2\alpha$.
  Then the support function $S(l,\theta)$ is   $\sigma$-homogeneous, namely,  
 $ 
 S(l,\theta)=   {\sigma}^{-1}l^{\sigma} h(\theta)$ for some function $h=h(\theta)$ on $\mathbb{S}^1$. 
 From \eqref{eq-asymptranslatorsupport}, we infer  that $h :  \mathbb{S}^1\to \mathbb{R}$ solves \bea\label{eq-shrinker}  h_{\theta\theta}+h=  2\alpha \sigma h^{\frac{\alpha-1}{\alpha}}.\eea  
Indeed, $h(\theta)$ represents the support function for a shrinking soliton under the $\frac{\alpha}{1-\alpha}$-curve shortening flow in $\mathbb{R}^2$. We say a convex closed embedded curve $\Gamma\subset \mathbb{R}^2$  is  a shrinking soliton under the $\beta$-curve shortening flow ($\beta$-CSF) if 
 \bea
  \kappa ^{\beta}=\lambda \langle x, \nu  \rangle 
 \eea   
  for some  $\lambda>0$. Since the curvature of a convex curve is $(h+h_{\theta\theta})^{-1}$  in terms of   its support function  $h(\theta)$, this readily shows that the level curves of the homogeneous solution $\bar u$ are shrinking solitons for the $\frac{\alpha}{1-\alpha}$-CSF which are unique upto rescalings.   
  \end{remark}
 
\medskip
 
 Remark \ref{remark-shrinker} suggests an important connection between the  the asymptotic behavior of  
 translating solitons for  the $\alpha$-GCF at infinity and the shrinking solitons for  the $\frac{\alpha}{1-\alpha}$-CSF.   In  \cite{andrews2003classification},   
Andrews classified all possible shrinking solitons for  the  CSF,  which will be used to classify possible blow-downs of  the translating solitons for  the $\alpha$-GCF on surfaces. 
In the following, we state the classification result of   \cite{andrews2003classification}.

\begin{definition}[$k$-fold symmetry]\label{def:k-fold} 
Let $\Gamma\subset \mathbb{R}^2$ be    a closed embedded curve  that contains the origin in its inside. 
We say $\Gamma$ is $k$-fold symmetric if   there is  the largest positive integer   $k$  such that $\Gamma$ is invariant under the rotation by angle $\frac{2\pi}{k}$. 
% Here,    such an integer $k$ is unique. 

%We say $\Gamma$ is $k$-fold symmetric if  $k$ is the largest integer such that $\Gamma$ is invariant after the rotation by angle $\frac{2\pi}{k}$. 
% Unless $\Gamma$ is a circle, there is a  unique integer $k\in \{1,2,\ldots \}$.  

Similarly,    an $l$-homogeneous function $u=r^l g(\theta)$ on $\mathbb{R}^2$ is  said to be   $k$-fold symmetric if its level curves are $k$-fold symmetric, i.e., $g(\theta+ \frac{2\pi}{k}) = g(\theta)$ for all $\theta\in \mathbb{S}^1$ and there is no larger integer with the same property.  
 %Unless $g$ is rotationally symmetric, there is unique $k=\{1,2,\ldots \}$.  
	
\end{definition}

\begin{remark}
According to the definition, circles  are not $k$-fold symmetric for any integer $k$.  
    Unless  the curve  $\Gamma$ in Definition \ref{def:k-fold}
  is a circle, there is a  unique positive integer $k$.     
\end{remark}

\begin{theorem} [Shrinkers, \cite{andrews2003classification}] 
 \label{thm-andrews2003} 
{Smooth strictly convex     closed embedded} shrinking solitons to the $\frac{\alpha}{1-\alpha}$-CSF 
 \bea  \kappa^{\frac{\alpha}{1-\alpha}}=  \langle x, \nu\rangle\eea 
 for $\alpha \in (0, 1)$   are classified as follows:
\begin{enumerate}[(a)]
\item If $\alpha=\tfrac14$,  all ellipses    centered at the origin   with  the area $\pi$ are only    shrinking solitons. 
\item If $\alpha\in \big[\tfrac{1}{9},\tfrac{1}{4}\big)\cup \big(\tfrac{1}{4},1\big)$, the unit circle is the only shrinking soliton. 
\item If $\alpha\in \big[\tfrac{1}{(m+1)^2},\tfrac{1}{m^2}\big)$ for some integer $m\ge 3$,  there exist exactly $(m-1)$ shrinking solitons modulo rotations. These are the unit circle  and the  $k$-fold  symmetric curves for   $3\leq k\leq m$.	 

\end{enumerate}
\end{theorem}
 
From Theorem \ref{thm-andrews2003}, 
   we have the following corollary for   the homogeneous solutions to  \eqref{eq-asymptranslatorgraph}. 
 
 \begin{corollary}[Classification of homogeneous solutions]\label{cor-anderews2003} 
 	Given $\alpha \in (0,1/4)$,    let $m$ be the largest integer satisfying $m^2 < \tfrac 1\alpha$. 
  Then modulo rotations in $x_1x_2$-plane, there are exactly $(m-1)$ homogeneous solutions to \eqref{eq-asymptranslatorgraph}, which are the rotationally symmetric solution and  (if $m\geq3$,) $k$-fold symmetric solutions for  $3\leq k\leq m$.

 \end{corollary}

 \medskip
 
 In order to   study translators for the     $\alpha$-GCF,   we   consider the asymptotic expansion  of the support function  near infinity. Let   $h$ be  a solution to  the equation \eqref{eq-shrinker} which is satisfied by   the support function  of    the homogeneous solutions  to  \eqref{eq-asymptranslatorgraph}. 
Suppose that $S(l,\theta)$ represents a translator and it is asymptotic to the homogeneous function $  {\sigma}^{-1} l^{\sigma} h(\theta)$, which is a solution to \eqref{eq-asymptranslatorsupport}, near $l=\infty$. 
  %Here,       $h$ is  a solution to \eqref{eq-shrinker}. 
%With an idea that $ (1-2\alpha)^{-1} l^{1-2\alpha}h(\theta)$ provides an ansatz of solution near infinity, 
To investigate the higher order asymptotics, we express the solution $S(l,\theta)$  as follows:
 \bea \label{eq-w}
 S(l,\theta)=  {\sigma}^{-1}   e^{\sigma s} h(\theta) + w(s,\theta) \quad \text{ for }  s= \ln l
 \eea 
with some function $w=w(s,\theta)$. If $S(l,\theta)$ solves \eqref{Se} (translator or blow-down equation),  $w(s,\theta)$ solves an equation of form
\be \label{eq-w29} w_{ss}+w_s+{h}^{\frac1\alpha}(w_{\theta\theta}+w)  = -E_1(w)-\eta E_2(w).\ee 
Here, the precise expressions of the error terms   $E_1(w)$ and $E_2(w)$ can be found in Lemma \ref{lem-w-eq}, but we will rarely make use of them {once the structural estimates in Lemma \ref{lem-error-estimate} are shown.}

Regarding the error terms above,   observe that a given solution to the blow-down equation has a rescaling denoted by $S_{\lambda}(l,\theta ) = \lambda^{-1} S(\lambda ^{\frac{1}{\sigma}} l ,\theta)$. Consequently, $w(s,\theta)$ is transformed into $w_{\lambda}(s,\theta) = \lambda^{-1} w(s+ \ln \lambda ^{\frac{1}{\alpha}},\theta) $. 
From this perspective, the function $e^{-\sigma s} w(s,\theta)$ behaves as  scaling invariant one. 
For this reason, $e^{-\sigma s} E_1(w)$ is a collection of terms that are at least quadratic in $e^{-\sigma s}w$ (and its derivatives), and $e^{-\sigma s}E_2(w)$ is a collection of       rapidly decaying terms in $s$, which are due to the difference between the translator and the blow-down equation.  The properties mentioned above are summarized to the extent necessary for this paper in Lemma \ref{lem-error-estimate}.

%%%%%%%%%%%%%%%%%%%%%%%%%%%%%%%%%%%%%%%%%%%%%%%%%%
 \subsection{Function space} \label{subsection-spectral} 
 %%%%%%%%%%%%%%%%%%%%%%%%%%%%%%%%%%%%%%%%%%%%%%%%%% 
For a given shrinker $h=h(\theta)$, let us   consider the operator $L$    defined  by 
 \bea \label{eq-linearopL}
 L\varphi=:h^{\frac1{\alpha}}( \varphi_{\theta\theta}+\varphi) .\eea
 %for  any smooth function $\varphi $ on $\mathbb{S}^1$.
Let 
$L^2_h(\mathbb{S}^1)$ denote the Hilbert space of   functions in $L^2(\mathbb{S}^1)$  equipped with the inner product %\marginpar{\color{red}SK: Replace $\langle,\rangle_h$ by $(,)_h$. }
 \bea \label{eq-innerh} 
 ( f,g)_h := \int_{\mathbb{S}^1} fg\, h^{-\frac1\alpha}d\theta , \eea
which is  adapted to  the operator $L$. The operator 
 $L$ is self-adjoint with respect to    $ ( \cdot , \cdot )_h$ and the resolvent   $(-L+\mu I)^{-1}$ (for large $\mu$) is compact and  self-adjoint with respect to   $ ( \cdot , \cdot )_h$.  
 Thus it follows that  the   eigenvalues of $L$  are real  and its eigenfunctions provide
an orthonormal basis of $L^2_h(\mathbb{S}^1)$. 
\medskip

Given the shrinker $h$, let us denote the spectrum of $L$  by  
\bea  \label{eq-lambda} \lambda_0 \ge \lambda_1\ge \lambda_2 \ge \ldots, \eea
where  the    eigenvalues  $\lambda_j$ are 
of finite multiplicity and $\lambda_j\to-\infty$.    An  orthonormal basis of $L^2_h(\mathbb{S}^1)$  is denoted by $\{\varphi_j \}_{j=0}^\infty$  which  consists 
of eigenfunctions satisfying 
    \bea \label{eq-lambdavarphi} L \varphi_j = \lambda_j \varphi_j.\eea 
In light of  \cite[Lemma 5]{andrews2002non}, the first three eigenvalues $\lambda_0$, $\lambda_1$ and $\lambda_2$  are known as \bea   \lambda_0=2\alpha\sigma > 0=\lambda_1=\lambda_2 >\lambda_3  \geq \cdots \eea and their    associated  eigenfunctions are 
\be 
h,\quad \cos\theta,\quad    \sin\theta ,
\ee 
respectively. Here, we  recall $\sigma=1-2\alpha$. In view of this observation, from now on, let us  fix 
\begin{equation} \label{eq-c0c1c2}
\varphi_0:=m_0 h, \quad  \varphi_1:= m_1 \cos \theta ,\quad \varphi_2:=m_2 \sin\theta,
\end{equation} 
where $m_i$'s are positive normalizing constants. 

\medskip 

We introduce   $L^2$   and  H\"older norms  which will be used  to study the asymptotic behavior at infinity. 
\begin{definition}[Norms on annuli and exterior domains]\label{def:norm.abbr}
Let $w=w(s,\theta)$ be a function  over    $I \times  \mathbb{S}^1$ for an interval $I\subset \mathbb{R}$.  Let   $k$ be  a non-negative integer,   let $0<\beta<1$ be a    constant and  let   $t$ and  $a$  be    constants such that $[t-a,t+a]\subset I$. 
Let us define 
\bea \Vert w\Vert_{L^2_{t,a}} := \Vert w \Vert_{L^2([t-a,t+a]\times \mathbb{S}^1)}, \eea 
\bea  
\Vert w \Vert _{C^{k}_{t,a}}:= \Vert w  \Vert_{C^{k}([t-a,t+a]\times \mathbb{S}^1)} ,\quad \hbox{and}\quad  \Vert w \Vert _{C^{k,\beta}_{t,a}}:= \Vert w  \Vert_{C^{k,\beta}([t-a,t+a]\times \mathbb{S}^1)}  . 
 \eea

For    a given  function    $w $      over       $[R,\infty)\times  \mathbb{S}^1$ with some  $R\in\mathbb{R}$ and   a constant     $\gamma\in\mathbb{R}$,
we      define a $\gamma$-weighted $C^{k,\beta}$-norm on   exterior domain       $[R,\infty)\times  \mathbb{S}^1$ as follows: 
\bea \label{eq-1566}\Vert w \Vert_{C^{k,\beta,\gamma}_R} :=  \sup_{t\ge R+1}} e^{-\gamma t}\Vert w \Vert_{C^{k,\beta}_{t,1}  .  \eea 
We denote  by $C^{k,\beta,\gamma}_R$  the   Banach  space   equipped with the  norm    $\Vert \cdot \Vert_{C^{k,\beta,\gamma}_R} $.
\end{definition}

\subsection{Main result}

With previously introduced preliminaries, let us introduce the precise statement of our main result. Assume that $\alpha\in  (0,1/4)$ and $h$, a smooth solution  to the shrinker equation \eqref{eq-shrinker} are fixed. In terms of the support function $S(l,\theta)$, we aim to construct solutions that are asymptotic to $\sigma^{-1} l^{\sigma} h(\theta)$. 
The difference $w=w(s,\theta)$ defined in \eqref{eq-w} would have asymptotics
\be \label{def-w}
w(s,\theta)=S(e^s,\theta)-  {\sigma}^{-1}   e^{\sigma s} h(\theta) = o(e^{\sigma s}) \quad \text{ as } s \to \infty.
\ee 
Here, we recall the constant $\sigma:=1-2\sigma$. 
In view of \eqref{eq-w29} and the explanation followed by, the equation
\be  \label{eq-Jacobi}
w_{ss}+w_s+{h}^{\frac1\alpha}(w_{\theta\theta}+w)   =0 
\ee
amounts to the Jacobi equation (the linearized equation)  around ${\sigma}^{-1}   e^{\sigma s} h(\theta) $ near infinity. 
\smallskip

The Jacobi fields, solutions to the Jacobi equation \eqref{eq-Jacobi}, serve as ansatzs for $w(s,\theta)$ given in \eqref{def-w} near infinity. In terms of the spectral decomposition of  $L$ on $L^2_h(\mathbb{S}^1)$ in Section \ref{subsection-spectral}, we may find Jacobi fields of the form $e^{\beta s} \varphi_j(\theta)$  via a separation of variables.  For each $j=0,1,2,\ldots$, there are two solutions $\beta= \beta^+_j$ and $\beta^-_j$ which are two distinct real roots of $ \beta^2+ \beta + \lambda_j =0 $, namely, 
\bea\label{eq-def-beta_j-pm}
 \beta^\pm_j := -  \frac 12 \pm  \frac{\sqrt {1- 4 \lambda_j  }}{2}.
\eea 
Note that   $\beta_0^+= -2\alpha$, $\beta_0^-= -\sigma$, $\beta_1^+=0=\beta_2^+$ and $\beta_1^-=-1=\beta_2^-$, and the  exponents $\beta^+_j$ and $\beta^-_j$ have the following order:
 \bea\ldots  < \beta^-_2=-1= \beta^-_1 <\beta^-_0 =-\sigma< \beta^+_0=-2\alpha <\beta^+_1=0=\beta^+_2< \ldots.\eea 
   % {\color{red}we already explained many times}Here  $\{\varphi_j \}_{j=0}^\infty $  is an orthonormal basis of $L^2_h(\mathbb{S}^1)$  consisting of  eigenfunctions of $L$ solving  \eqref{eq-lambdavarphi}.   
  We say the decaying rate 
   $ \beta^+_j $ is  `slow' if $\beta^+_j< \sigma=1-2\alpha$.   
  The dimension of `slowly decaying' Jacobi fields  is defined by 
  \bea \label{eq-K}
  K:= \# \left\{j\in \mathbb{N}\cup\{0\}:  \beta^+_j < \sigma \right\}.
\eea  
Here,    the term  `slowly decaying'    was previously used in \cite{Chan1997CompleteMH} for minimal hypersurfaces   asymptotic to minimal cones near infinity.   
Let $\mathcal{J}$ denote  the  space of  the   slowly decaying Jacobi fields: 
\be
 \mathcal{J}:=   \Big\{\,\sum_{j=0}^{K-1} y_j \,  l^{\beta^+_j} \varphi_j(\theta)  \Big\}_{ (y_0,\ldots,y_{K-1})\in \mathbb{R}^{K} } . 
 \ee
We will construct the translators which roughly correspond to the space of slowly decaying Jacobi fields $\mathcal{J}$. 
 
  \begin{remark}\label{remark-K}
 The dimension $K$ can be expressed in terms of $\alpha$ and $h$ as follows.
 \begin{enumerate}[(a)]
     \item If $h$ is radial, then $K = 2m+1 $ with $m=\lceil \alpha^{-1/2}\rceil -1$.  
     \item If  $h$ is $k$-fold symmetric, then $K  = 2k-1.$
 \end{enumerate}
 In fact, 
if $h$ is radial, i.e., $h\equiv (2\alpha\sigma)^{\alpha}$, then   $\lambda_j$ and $\beta^\pm_j$ are explicitly computed,  %as in  \eqref{eq-lambdavarphi} and \eqref{eq-def-beta_j-pm}, respectively. 
and hence  we obtain that  
$ K = 2m+1 $ in the case.  Next, if  $h$ is $k$-fold symmetric,   utilizing the result  of   \cite[Theorem 1.2]{ChoiSun} implies that 
$ K  = 2k-1.$
 \end{remark}
 
The construction is based on a fixed point argument, i.e., we perturb the blow-down by a Jacobi field and design an iteration that gradually (geometrically) reduces the error   originating from the nonlinearity of the equation. In forthcoming papers, we will conversely prove that a slowly decaying Jacobi field appears as the leading asymptotics of the difference between two translators. Then   the translators will be classified  by showing a one-to-one correspondence between the translators and the  constructed translators in this paper. To make this plan work, we construct translators hierarchically starting from the last parameter  (corresponding to the coefficient of the slowest decaying mode)  subject to a reference translator. In this way, the construction assigns each translator a coordinate (parameter) $\mathbf{y}=(y_0, \ldots, y_{K-1})\in \mathbb{R}^K$. Moreover, we will also be interested in the topology of  moduli space of translators. For this reason, here we construct the translators so that they change continuously with respect to the parameter $\mathbf{y}$. Let us introduce  the following  convention: 
\be
  (\mathbf{0}_j , (a_j,\ldots,a_{K-1}))= (\mathbf{0}_j,a_j,a_{j+1},\ldots,a_{K-1}) = (0,0,\ldots, 0,a_j,\ldots, a_{K-1}),   \ee
for $ (a_0,a_1,\cdots, a_{K-1}) \in \mathbb{R}^{K}$   abusing the notation,  where $\mathbf{0}_j $ denotes   the zero vector in $\mathbb{R}^j$. 
\begin{theorem}[Existence]\label{thm-existence}
For given $\alpha \in (0,\frac{1}{4})$ and shrinker $h$, there exists $K$-parameter (see Remark \ref{remark-K}
 for $K$) family of translators, namely $\{\Sigma_{\mathbf{y}}\}_{\mathbf{y} \in \mathbb{R}^{K}}$,   whose level curves are asymptotic to the shrinker after rescalings.
 More precisely, the surfaces $\Sigma_{\mathbf{y}}$ and their support functions $S_{\mathbf{y}}$ defined in \eqref{eq-supportfunction} satisfy the followings:  for some  constant $\eps>0$ depending only on $\alpha$   and $h$, 
\begin{enumerate}[(i)]
\item  $\Sigma_{\mathbf{y}}$ is the graph of a smooth strictly convex entire function  and $S_{\mathbf{y}}$ has the asymptotics  
\bee  S_{\mathbf{y}}(l,\theta) = \sigma^{-1} { l^{\sigma} } h(\theta) + o (l^{\sigma-\eps })\quad \text{ as } l\to \infty.
\eee  
%\noindent (ii) For any $(y_j,\ldots, y_{K-1})\in\mathbb{R}^{K-j} $ with some $0\leq j \leq K-1$ and $a\in \mathbb{R}$,  
%\be \label{eq-897}
%S_{(\mathbf{0}_j,a+y_{j},y_{j+1},\ldots,y_{K-1})}(l,\theta) -S_{(\mathbf{0}_j,y_j,y_{j+1},\ldots,y_{K-1})}(l,\theta) 
%=  a \,   l^{\beta^+_j} \varphi_j(\theta) + o(l^{\beta^+_j-\eps})\quad \hbox{as $l\to \infty$.}
%\ee 
  \item For distinct vectors $\mathbf{\hat y}=(\hat y_0, \ldots,\hat y_{K-1} )$, $\mathbf{\bar y }=(\bar  y_0, \ldots,\bar y_{K-1} )$ in $\mathbb{R}^{K}$, let     $j=\max\{i\,: \, \hat y_i \neq \bar  y_i\}$ be the last parameter that does not coincide. 
Then there holds 
 \bee 
 S_{\mathbf{\hat y}}(l,\theta) -S_{\mathbf{\bar  y}}(l,\theta)= (\hat y_j-\bar y_j)l^{\beta^+_j} \varphi_{j}(\theta)+ o(l^{\beta^+_j-\eps}).
 \eee  

\item The first three parameters correspond to translations in $\mathbb{R}^3$:  for $\mathbf{y} \in\mathbb{R}^{K} $  and $a_j\in  \mathbb{R}$  ($j=0,1,2$),  
\bee  
\Sigma_{\mathbf{y}+(a_0,a_1,a_2,\mathbf{0}_{K-3} )}=  \Sigma_{\mathbf{y}} -a_0m_0 \mathbf{e}_3  +a_1m_1\mathbf{e}_1+a_2m_2\mathbf{e}_2,
 \eee 
 where $m_j$'s are the positive  normalizing constants given in \eqref{eq-c0c1c2}. 

\item The parametrization $\mathbf{y}\mapsto \Sigma_{\mathbf{y}}$ is continuous in the $C^0_{\text{loc}}$-convergence topology of surfaces.
\end{enumerate}
%Here,  $\eps>0$  is a constant depending only on $\alpha$.
\end{theorem}

 \medskip

\begin{remark}\label{remark-existence}

\begin{enumerate}[(a)]

\item  It is a direct consequence of $(ii)$ that  $\Sigma_{\mathbf{\hat y}}\neq \Sigma_{\mathbf{\bar y }}$ if $\mathbf{\hat y} \neq \mathbf{\bar y}$. 
\item The construction is designed so that the proof shows more quantitative estimates than  $(i)$ and $(ii)$. See Corollary \ref{cor-wexpression}. 
\item By the regularity theory for the Monge--Amp\`ere  equation, we also get the continuity of parametrization in the $C^\infty_{loc}$-convergence topology in (iv).  
\item The proof equally works for constructing a family of entire solutions to the blow-down equation $\det D^2 u = |Du|^{4-\frac1\alpha}$ and we get the same result except that now each solution $u(x)$ is smooth away from the point at which the gradient $Du$ is $0$.  See Remark \ref{remarkno1} at the end of Section  \ref{sec-thm-existence}.

	\end{enumerate}

\end{remark}

  %%%%%%%%%%%%%%%%%%%%%%%%%%%%%%%%%%%%%%%%%%%%%%%%%%

\section{Existence on exterior domains} \label{sec-existence-exterior}

%%%%%%%%%%%%%%%%%%%%%%%%%%%%%%%%%%%%%%%%%%%%%%%%%%

 We  first    establish  the existence of  translators on exterior domains. To begin with, we study an invertibility of the following (asymptotic) Jacobi operator   
  \be\label{eq-L-hat}
  \hat Lw:=  w_{ss} +w_s+{h}^{\frac1\alpha}( w_{\theta\theta}+w) 
  \ee    on exterior regions.     %{\color{red}BC: do we need and use $R>1$ in the lemma below??}
    
\begin{lemma}[Invertibility of Jacobi operator]\label{lem-ex-linear}
 Let  $\beta$ and $\gamma$ be  constants such that   $0<\beta<1$ and   $ \beta^+_{m-1}< \gamma  <\beta^+_m$  for   some $m\in \mathbb{Z}_{\geq0 }$ (here we use the notational convention $\beta^+_{-1}:= \beta^-_0$).    
  For a given function  $g$ in $C^{0,\beta,\gamma}_R$ for some $R\in\mathbb{R}$, there exists a solution $w\in C^{2,\beta,\gamma}_R $  to the Jacobi equation 
\bea \label{eq-w-linear}
\hat Lw=  w_{ss} +w_s+{h}^{\frac1\alpha}( w_{\theta\theta}+w)    =g\quad\hbox{ on  \,\,} (R,\infty)\times \mathbb{S}^1  \eea with the zero boundary condition: $ w \equiv 0$ on $\{s=R\}\times \mathbb{S}^1$. 
More precisely,  $w$ has an explicit representation   
\be\label{eq-w-decomp}
w(s,\theta)=\sum_{j=0}^\infty  w_j(s) \varphi_j(\theta),\ee
 where 
\begin{align} 
  w_j (s)&= -e^{\beta^-_j s} \int_R^s e^{(\beta^+_j-\beta^-_j)r }\int_r^\infty e^{-\beta^+_j t}g_j(t) dt dr \quad \text{ for  }j \ge m,  \label{eq-wj}\\
  w_j (s)&= e^{\beta^-_j s} \int_R^s e^{(\beta^+_j-\beta^-_j)r }\int_R^r e^{-\beta^+_j t}g_j(t) dt dr \quad  \text{ for }0\le j <m \label{eq-wj2}
\end{align}
with  $g_j(s):= ( g(s,\cdot),\varphi_j )_h$. %for $j\geq 0$. Here,  $\beta_j^\pm$ are given in \eqref{eq-def-beta_j-pm} and  $\{\varphi_j \}_{j=0}^\infty $  is an orthonormal basis of $L^2_h(\mathbb{S}^1)$  consisting of  eigenfunctions of $L$ solving  \eqref{eq-lambdavarphi}.    
Moreover, there holds an estimate
\bea \label{eq-w-est}
 \Vert w \Vert_{C_R^{2,\beta,\gamma}} \le C \Vert g \Vert_{C_R^{0,\beta,\gamma} }
 \eea  
 for some constant $C>0$ independent of $R\in\mathbb{R}$.  Here, $C$ may depend on $\alpha, h, $ $m,$ $ \gamma,$ and $ \beta$.
 
  \end{lemma}

 \begin{proof}
  If we express our solution $w= \sum_{j=0}^\infty w_j(s) \varphi_j(\theta)$,  then each function $w_j(s)$ should solve  
  \be \label{eq-w_j}
  (w_j)_{ss}+(w_j)_{s}+\lambda_jw_j=g_j\quad \text{ with }\,\, w_j(R)=0.
  \ee
We note that  this is a second-order ODE in the variable $s$, and it is easy to check   that     $w_j$ given in   \eqref{eq-wj}-\eqref{eq-wj2}  solves \eqref{eq-w_j} with the zero  boundary data.  Here, we notice that  $w_j$  is  well-defined since   $g_j(s)= ( g(s,\cdot),\varphi_j )_h$ belongs to  $C^{0,\beta,\gamma}_R$.  
\smallskip

    To prove  that  $w$   satisfies \eqref{eq-w-linear} and \eqref{eq-w-est},  
 the following weighted $L^2_h$-estimate is crucial. 
  
  \begin{claim}[Weighted $  L^2_h$-estimate] \label{cla-L-infty-ext}
There is a  constant $C>0$ (independent  of  $R$) such that  
 \bea\label{eq-L-2-linear-sol}
  e^{-\gamma s}\Vert w(s,\cdot)\Vert_{L^2_h(\mathbb{S}^1)} =  e^{-\gamma s}  \Big(\sum_{j=0}^\infty {w_j^2(s)} \Big)^{1/2}   \le C \Vert g \Vert_{C_R^{0,0,\gamma}}\quad\hbox{for all $s\geq R$}.
  \eea   
  Here,  a constant $C>0$  may depend on $\alpha$, $ h$,   $m$,  and $\gamma$. 
 \end{claim}

 Once   Claim \ref{cla-L-infty-ext} is shown,  we obtain   a weighted  $C^0$-estimate
 \bea\label{eq-L-infty-linear-sol}
 \Vert w \Vert_{C_R^{0,0,\gamma}} \le C\Vert g \Vert_{C_R^{0,0,\gamma}} 
 \eea
  for a  constant $C>0$ (independent in $R$).  
 Indeed,   note first that     $\sum_{i=0}^j  w_i(s) \varphi_i(\theta)$   solves   \eqref{eq-w-linear}  with $\sum_{i=0}^j  g_i(s) \varphi_i(\theta)$ as the right-hand side.      Since   the equation \eqref{eq-w-linear}   is uniformly elliptic in terms of $(s,\theta)$, we infer \eqref{eq-L-infty-linear-sol} by applying (interior and boundary) $W^{2,2}$-estimates and the Sobolev embedding to $\sum_{i=0}^j  g_i(s) \varphi_i(\theta)$ and then taking $j\to \infty$. Therefore,   by the  Schauder theory,     $w$ satisfies \eqref{eq-w-linear} and \eqref{eq-w-est}, which completes the proof.   
  \medskip
  
  Lastly, we will give the proof  of Claim \ref{cla-L-infty-ext}.

\begin{proof}[Proof of   Claim \ref{cla-L-infty-ext}]  
Let    $\delta>0$ be  a  small constant  such that 
\be
0<\delta< \min(\beta_m^+-\gamma, \gamma-\beta_m^-). 
\ee
 By  H\"older's inequality applied to \eqref{eq-wj},   it follows   that  for $j\ge m$ and  $s\ge R$, 
\bea w_j^2 (s)  &\le e^{2\beta^-_js}\left[\int_R^s e^{2(\gamma- \beta^-_j -\delta) r} dr \right]\left[\int_R^s e^{2(\beta^+_j +\delta-\gamma )r}\Big( \int _r^\infty  e^{-\beta^+_j t}g_j(t) dt \Big)^2 dr \right]. 
\eea  
 We see that for $s\ge R$, 
\be\int_R^s e^{2(\gamma- \beta^-_j -\delta)  r}dr \le \frac{e^{2(\gamma- \beta^-_j -\delta) s}-e^{2(\gamma- \beta^-_j -\delta)  R}}{2(\gamma- \beta^-_j -\delta) }  \le \frac{e^{2(\gamma- \beta^-_j -\delta)  s}}{2(\gamma- \beta^-_j -\delta) }  
\ee
since  $\gamma -\beta^-_j\geq \gamma -\beta^-_m >\delta>0$ for $j\geq m$. Using    H\"older's inequality again implies that 
\be
\ba
 \Big( \int _r^\infty  e^{-\beta^+_j t}g_j(t) dt \Big)^2 &\le  \int_r^\infty e^{2(\gamma+\delta-\beta^+_j)t}dt \int_r^\infty e^{-2(\gamma+\delta) t}g^2_j(t) dt\\
&= \frac{e^{2(\gamma+\delta-\beta^+_j)r}}{2(  \beta^+_j-\gamma-\delta)}\int_r^\infty e^{-2(\gamma+\delta) t}g^2_j(t) dt  
\ea
\ee
since  $\beta^+_j-\gamma \geq \beta^+_m-\gamma  > \delta >0 $ for $j\geq m$. 
Combining the estimates above together, we get  that for $j\geq m$ and $s\geq R$, 
\be
\ba
 e^{-2\gamma s} {w_j^2(s)}  &\le \frac{ e^{-2\delta s }}{4 (\gamma- \beta^-_j -\delta) (\beta^+_j-\gamma -\delta)}\int_R^s e^{4\delta r} \int_r^\infty e^{-2(\gamma+\delta) t}g^2_j(t) dtdr \\
& \le \frac{ e^{-2\delta s }}{4 (\gamma- \beta^-_m -\delta) (\beta^+_m-\gamma -\delta)}\int_R^s e^{4\delta r} \int_r^\infty e^{-2(\gamma+\delta) t}g^2_j(t) dtdr . 
\ea\ee
 
 Since
\be  \sum_{j=1}^\infty  g_j^2 (t) = \Vert g(t,\cdot)\Vert_{L^2_h(\mathbb{S}^1)}^2 \le Ce^{2\gamma t}\Vert g\Vert_{C_R^{0,0,\gamma}}^2 \quad\hbox{for $t\geq R$},\ee
  we obtain that for $s\geq R$,
\be\label{eq-w_j-sum-big}
\ba
 e^{-2\gamma s} \sum_{j\ge m}  {w_j^2(s)} & \le  C{ e^{-2\delta s }}\int_R^s e^{4\delta r} \int_r^\infty e^{-2(\gamma+\delta) t}\sum_{j\ge m}g^2_j(t) dtdr\\
 &\le C{ e^{-2\delta s }}\int_R^s e^{4\delta r} \int_r^\infty e^{-2\delta t}\Vert g\Vert_{C_R^{0,0,\gamma}}^2dtdr  \le  C\Vert g\Vert_{C_R^{0,0,\gamma}}^2.
\ea\ee
Here,  a constant $C>0$  may depend on $\alpha$, $ h$,    $ \beta^+_m-\gamma $,  $ \gamma-\beta^-_m  $ and  $\delta$, and may vary from line to line.

A similar argument applied to \eqref{eq-wj2} yields that   
\be \label{eq-w_j-sum-sm}
e^{-2\gamma s}  \sum_{0\le j\le m-1} {w_j^2(s)} \le C\Vert g\Vert_{C_R^{0,0,\gamma}}^2,\ee 
where we used that 
$  \gamma-\beta_j^+\ge    \gamma-  \beta_{m-1}^+>0 $ %and $  \gamma-\beta_j^-\ge    \gamma-  \beta_{0}^->0 $
for $0\leq j \leq m-1$.
Thus     by   \eqref{eq-w_j-sum-big}  and \eqref{eq-w_j-sum-sm}, 
we obtain the weighted $L^2_h$-estimate in  \eqref{eq-L-2-linear-sol},  which    finishes the proof of Claim  \ref{cla-L-infty-ext}. 
\end{proof} 
The proof of Lemma  \ref{lem-ex-linear} is complete.
 \end{proof}

  In  light of  Lemma  \ref{lem-ex-linear},   define  an inverse operator    of $\hat L$:  
\be H_{m,R} : C^{0,\beta,\gamma}_R \rightarrow C_R^{2,\beta,\gamma}  
\ee
       by 
\be\label{eq-def-H0}
H_{m,R}(g):=   \sum_{j=0}^\infty  w_j(s) \varphi_j(\theta),
\ee  
Here, for a given $g\in C^{0,\beta,\gamma}_R$,  $w_j$'s  are defined as in \eqref{eq-wj}-\eqref{eq-wj2}. 
A convenient version of  Lemma \ref{lem-ex-linear} which we will use throughout this section is the following. 
We first define a   constant $\delta_0 = \delta_0(\alpha,h)>0$ by
%\[\delta_0 := \min \{\sigma-\beta^+_{K-1}\} \cup \{\beta^+_{0}-\beta^-_{0}\}\cup \{\beta^+_{m}-\beta^+_{m-1}\,:\, m=1,\dots, K-1 \text{ and }\beta^+_{m}-\beta^+_{m-1}\neq0 \}.\] 
 \be\label{def-del0}
\delta_0 := \min \{\sigma-\beta^+_{K-1}\} \cup \{\beta^+_{m}-\beta^+_{m-1}\,:\,  1 \leq m\leq  K-1 ,\,\,\beta^+_{m}-\beta^+_{m-1}>0 \}  \cup \{\beta^+_0 - \beta^-_0 \}.
\ee 
 Here, the nonzero condition $\beta^+_{m}-\beta^+_{m-1}\neq 0$ is due to possible multiplicity of $\beta^+_m$. 
 \begin{corollary}  \label{cor-ex-linear}   
For any $\delta\in (0,\delta_0)$, there is  a constant  $C=C(\alpha,h,\delta,\beta)>0$  such that  
\bea \label{eq-w-est'}
 \Vert H_{m,R}(g)\Vert_{C_R^{2,\beta,\beta^+_m -\delta }} \le C \Vert g \Vert_{C_R^{0,\beta,\beta^+_m -\delta } } 
 \eea   and 
 \bea 
 \Vert H_{K,R}(g)\Vert_{C_R^{2,\beta,\sigma  -\delta }} \le C \Vert g \Vert_{C_R^{0,\beta,\sigma  -\delta } } 
 \eea  for all $m \in \{0,\ldots, K-1\}$ and      $R\in\mathbb{R}$. 

 \end{corollary}
 % To prove  the existence of    the  barriers,    we  will use a fixed point argument.
 
 \medskip
 The term $E(w)$ in \eqref{eq-w29}  is the remainder when the translator or blow-down equation is linearized around the homogeneous profile.  The nonlinear term $E(w)$ basically consists of quadratic terms and exponentially decaying terms. The following lemma makes this assertion precise.

\begin{lemma} [Error estimate] \label{lem-error-estimate} Let $E=-E_1-\eta E_2$ be the remainder terms   in    \eqref{eq-w29} with $\eta\in \{0,1\}$ (see also Lemma \ref{lem-w-eq}). 
Then there is  a small  constant  $c_0=c_0(\alpha,h,\beta)>0$   so that 
  if functions $w(s,\theta)$,   $u(s,\theta)$ and $w_i(s,\theta)$ ($i=1,2$) satisfy 
\be \label{eq-lem-error-esti1}
  \Vert w\Vert_{C^{2,\beta}_{\tau ,1}},\,\,  \Vert u\Vert_{C^{2,\beta}_{\tau ,1}}, \,\,  \Vert w_i\Vert_{C^{2,\beta}_{\tau,1}}\, \le\, c_0 e^{\sigma \tau}   
  \ee
 for some $0<\beta<1$ and $\tau\ge0$, then 
\bea \label{eq-est-error-diff1} 
e^{-\sigma \tau} \Vert E(w+u)-E(w)\Vert_{C^{0,\beta}_{\tau,1}} \le C  \left(\eta e^{-4\alpha \tau }+e^{-\sigma \tau }\Vert w\Vert_{C^{2,\beta}_{\tau,1}} +e^{-\sigma \tau }\Vert u\Vert_{C^{2,\beta}_{\tau,1}} \right)e^{-\sigma \tau}\Vert u\Vert_{C^{2,\beta}_{\tau,1}} ,
\eea 
\bea \label{eq-est-error-quadratic}
 e^{-\sigma \tau}\Vert E(w)\Vert_{C^{0,\beta}_{\tau,1}} \le C  \left[\eta e^{-4\alpha \tau }+ \left( e^{-\sigma \tau }\Vert w\Vert_{C^{2,\beta}_{\tau,1}}\right) ^2   \right] ,
 \eea
%{\color{red}BC: Above is also true, but shouldn't this be 
%\bea 
% e^{-\sigma \tau}\Vert E(w)\Vert_{C^{0,\beta}_{\tau,1}} \le C  \left[\eta e^{-4\alpha \tau }+  e^{-\sigma \tau }\Vert w\Vert_{C^{2,\beta}_{\tau,1}}   \right] e^{-\sigma \tau }\Vert w\Vert_{C^{2,\beta}_{\tau,1}}  ,
% \eea? This is stronger and it seems what we need. }
and 
\bea \label{eq-est-error-diff2}
e^{-\sigma \tau}\Vert (E(w_1+u ) - E(w_1)) - (E(w_2+u) -E(w_2))\Vert_{C^{0,\beta}_{\tau,1}} \le C e^{-2\sigma \tau } \Vert w_1-w_2\Vert _{C^{2,\beta}_{\tau,1}} \Vert u\Vert_{C^{2,\beta}_{\tau,1}}, 
 \eea 
for some  $C=C( \alpha,h, \beta)<\infty $.
\begin{proof} 
To denote the full dependency of $E(w)$ on all derivatives of $w$, we write $E(v)$ as $E(v)(s,\theta  )=E(D^2 v(s,\theta ),\nabla v(s,\theta ), v(s,\theta), s,\theta )$ in the proof. 
At $(s,\theta )\in \mathbb{R} \times \mathbb{S}^1  $,  
\bea  \label{eq-integralrep1}
E(w+u)-E(w) = \int _0^1 \left\langle (DE)_{w+tu},  (D^2 u,\nabla u , u, 0,0)\right\rangle\,dt     .
\eea 
Here, $(DE)_{w+tu}$ denotes the gradient of $E$ (in five arguments) evaluated at $(  D^2 (w+tu) ,\nabla  (w+tu), w+tu,s,\theta  )$ and the derivatives of $u$ and $w+tu$ are evaluated at $(s,\theta)$. Similarly, it holds that  at $(s,\theta )\in \mathbb{R} \times \mathbb{S}^1  $,
\bea \label{eq-integralrep2} 
&\left[E(w_1+u ) - E(w_1)\right] - \left[E(w_2+u) -E(w_2)\right]  \\
& = \int_0^1 \int_0^1 ( D^2 E)_{w_2+tu+r(w_1-w_2)}[(D^2 (w_1-w_2),\nabla (w_1-w_2) , w_1-w_2, 0,0),(D^2 u,\nabla u , u, 0,0)]\, dt dr.
 \eea   
 \smallskip

We observe that  $e^{-\sigma s}E_1(w)$ is at least quadratic in the derivatives of $e^{-\sigma s}w$. If $\Vert w\Vert_{C^{2,\beta}_{\tau ,1}} $, $\Vert w_i\Vert _{C^{2,\beta}_{\tau ,1}}$, $\Vert u\Vert _{C^{2,\beta}_{\tau ,1}} \le c_0  e^{\tau s}$ 
for some sufficiently small $c_0>0$, then there exists $C=C(\alpha,h,\beta)$ (independent of $\tau\geq 0$ and $c_0>0$) such that for all $t \in [0,1]$,
\bea 
e^{-\sigma \tau }\Vert (DE_1)_{w+tu}\Vert_{C^{0,\beta}_{\tau ,1} }\le  C e^{-\sigma \tau } \left(e^{-\sigma \tau } \Vert w\Vert _{C^{2,\beta}_{\tau ,1}}+e^{-\sigma \tau } \Vert u\Vert _{C^{2,\beta}_{\tau ,1}}\right)  ,
\eea
 and for all $t , s\in[0,1]$, 
\bea e^{-\sigma \tau }  \Vert (D^2 E_1)_{w_2+tu+s(w_1-w_2)}\Vert_{C^{0,\beta}_{\tau ,1} }\le Ce^{-2\sigma \tau } .
\eea
  Similarly, we get that for all $t \in [0,1]$,
\be
  e^{-\sigma \tau }\Vert (DE_2)_{w+tu}\Vert_{C^{0,\beta}_{\tau ,1} }\le  C  e^{-(\sigma +4\alpha)\tau },\qquad   e^{-\sigma \tau }  \Vert (D^2 E_2)_{w_2+tu+s(w_1-w_2)}\Vert_{C^{0,\beta}_{\tau ,1} }\le C   e^{-{ (2\sigma+ 4\alpha)}\tau } .
\ee
Here,  a  sufficiently small  $c_0>0$ is chosen  to control nonlinear terms   involving $ (h+{ e^{-\sigma s}}{w_s})$ in  $E $   for the estimates above. 
The first and third estimates are direct consequences of these bounds and the integral representations \eqref{eq-integralrep1} and \eqref{eq-integralrep2}. The second estimate follows from the first one since  \be
e^{-\sigma \tau } \Vert E(0)\Vert _{C^{0,\beta}_{\tau,1}} \le C\eta e^{-4\alpha \tau} .
\ee
This finishes the proof.
\end{proof}
\end{lemma}

Now we introduce  effective coordinates  which are designed to    obtain  the quantitative estimate of translators. %  in Corollary \ref{cor-wexpression}.

 \begin{definition}[Effective coordinates] \label{def-b(a)}
We define a transformation    $\mathbf{b} : \mathbb{R}^K \rightarrow \mathbb{R}^K$  which maps $\mathbf{a}=(a_0,\ldots, a_{K-1})$ to $\mathbf{b}(\mathbf{a})=(b_{0},\ldots, b_{K-1})$ with  $b_i=\textrm{sgn}(a_i)|a_i|^{1/({\sigma-\beta^+_i})}$. 
\end{definition}
Later,  the 
translators $\Sigma_{\mathbf{a}}$  will be constructed so that they satisfy uniform estimates on a region $s\ge  R_{\rho,{\mathbf{a}} }$, where  we denote  
\be 
R_{\rho,{\mathbf{a}} }:=\rho + \ln (|\mathbf{b}(\mathbf{a})|+1)
\ee
for    $0<\rho<\infty$ and ${\mathbf{a}}\in \mathbb{R}^K$; see Corollary \ref{cor-wexpression} for instance. 
\medskip

 Next,  let us fix the constants %$\sgamma $ so that $\sgamma  =  \sigma -\delta $, where 
%\be \label{eq-deltagap}
 %\qquad \sgamma  :=  \sigma -\delta \qquad \hbox{with\,\, $\delta :  =  \frac12 { \min  (\sigma- \beta^+_{K-1} ,  4\alpha , \delta_0 )} >0$},
 %\ee 
 \be \label{eq-deltagap}
 \qquad \sgamma  :=  \sigma -\delta \qquad \hbox{with\,\, $\delta :  =  \frac12 { \min   (  \delta_0 ,4\alpha  )} >0$},
 \ee 
 where $\delta_0>0$ is  the constant defined by  \eqref{def-del0}.  

\begin{theorem} [Existence on exterior domains]\label{thm-existencerevised11} There exist a constant $0<\rho_1<\infty$ and  a $K$-parameter family of translators defined outside of compact sets, which are represented by  functions {$\tilde  w _{\mathbf{a}} $ in $ C^{2,{ \frac12},\sgamma}_{ R_{\rho_1,\mathbf{a}}}$} 
%$\tilde w_{\mathbf{a}}:[  R_{\rho_1,\mathbf{a}},\infty)\times \mathbb{S}^1 \to \mathbb{R} $ 
for $\mathbf{a}\in \mathbb{R}^K$.   
  Here, $ R_{\rho_1,\mathbf{a}} := \rho_1+ \ln (\vert \mathbf{b}(\mathbf{a})\vert+1) . $  
\smallskip

More precisely,   the constructed solutions  $\tilde w_{\mathbf{a}}$ satisfy the following  estimates with  some constant $0<c_1<\infty $  and a family of   functions   {$\tilde  g _{\mathbf{a}} $ in $ C^{2,{ \frac12},\sgamma}_{ R_{\rho_1,\mathbf{a}}}$}  
%$g_{\mathbf{a}}: [ R_{\rho_1,\mathbf{a}},\infty)\times \mathbb{S}^1 \to \mathbb{R} $
for $\mathbf{a} \in \mathbb{R}^K$:  
\begin{enumerate}[(a)]
\item  For each $\mathbf{a}\in \mathbb{R}^K$,
\be \label{eq-448bound}
\Vert \tilde w_{\mathbf{a}} \Vert_{C^{2,{\beta},  \sgamma  }_{ R_{\rho_1,\mathbf{a}}} } \le  2c_1K e^{(\sigma -  \sgamma  ) R_{\rho_1,\mathbf{a}}}    <   e^{(\sigma -  \sgamma  ) R_{\rho_1,\mathbf{a}}}\qquad  \left(\beta=  \tfrac{1}{3}, \tfrac{1}{2}\right), 
\ee
with the constant   $\sgamma $  in \eqref{eq-deltagap}. 
\item  For any two vectors $\mathbf{a}=(\mathbf{0}_{j+1}, \mathbf{a}_{K-j-1})$ and $\mathbf{a}'=(\mathbf{0}_{j}, a_j, \mathbf{a}_{K-j-1}) $ with   $ \mathbf{a}_{K-j-1}\in \mathbb{R}^{K-j-1}$ and $a_j\not=0$,
%\be \label{eq-wdiff} \tilde w_{\mathbf{ a}'}-\tilde w_{\mathbf{a}}= a_j   ( e^{\beta^+_{j}s} \varphi_{j}(\theta) + {\color{red} g}(s,\theta)) \quad \hbox{for some    }g \hbox{ with }\Vert {\color{red} g} \Vert_{C^{2,\beta,\beta^+_{j}-\delta  }_{\tilde R_{\mathbf{  a'}}} } \le   e^{\delta  {\tilde R_{\mathbf{  a'}}}}   ,\ee
  \be \label{eq-wdiff}
\tilde w_{\mathbf{ a}'}-\tilde w_{\mathbf{a}}= a_j   \left( e^{\beta^+_{j}s} \varphi_{j}(\theta) + {  g_\mathbf{  a'}}(s,\theta)\right) \quad \hbox{ with }\,\,\Vert {  g_\mathbf{  a'}} \Vert_{C^{2,{ \beta},\beta^+_{j}-\delta  }_{ R_{\rho_1,\mathbf{  a'}}} } \le   e^{\delta  { R_{\rho_1,\mathbf{  a'}}}}    \qquad\left(\beta=  \tfrac{1}{3}, \tfrac{1}{2}\right),
\ee
  with    the constant $\delta>0$ in \eqref{eq-deltagap}. 
\end{enumerate}
Here, positive   constants $\rho_1$ and $c_1$  depend only on $\alpha$ and  $h$. % and {$  \beta$}. 
\end{theorem}

\begin{remark}\label{rem-two-exp}
Theorem \ref{thm-existencerevised11} implies that  $\Vert \tilde w_{\mathbf{a}} \Vert _{C^{2,{ \beta },\sigma}_{ R_{\rho_1,\mathbf{a}}} }\le 2c_1K<1 $ and  $\Vert {g}_{\mathbf{a}'} \Vert_{C^{2,{ \beta},\beta^+_{j}  }_{ R_{\rho_1,\mathbf{  a}'}} } \le   1 $. 
Next,  we should mention  that   the H\"older exponents are fixed  at $\beta=1/2$ or $1/3$   for the sake of notational simplicity.
%Next, it is for the sake of notational simplicity to fix the H\"older exponents at $\beta=1/2$ or $1/3$.
The reason why we introduce two H\"older exponents and ensure the estimates   hold uniformly for both exponents 
will become apparent in Proposition \ref{thm-modulicont-barr}. In the proposition, we will show the continuity of barriers in a weighted $C^{2,\frac{1}{3}}$-norm on exterior regions. When proving it,  the  following  interpolation inequality is used: 
\begin{equation*}
\Vert w \Vert_{C^{2,{\frac{1}{3}}}_{s,1}} \le C \Vert w \Vert_{L^2_{s,1}}^{\vartheta} \Vert w \Vert^{1-\vartheta}_{C^{2,{\frac{1}{2}}}_{s,1}}
\end{equation*}
in order to bypass a %specific
technical difficulty. Apart from this, it can be regarded that we  construct a solution in the H\"older space with the higher exponent $1/2$.   
\end{remark}

\medskip

\begin{proof} [Proof of Theorem \ref{thm-existencerevised11}]

We will construct $\tilde w_{\mathbf{a}}$ {in $  C^{2,{ \frac12},\sgamma}_{ R_{\rho_1,\mathbf{a}}}$}  for    $\mathbf{a}\in \mathbb{R}^K$  of the  form  $\mathbf{a}=(\mathbf{0}_{j+1}, \mathbf{a}_{K-j-1})\in\mathbb{R}^K$ with  a fixed $j$ which runs  backward  from $j=K-1$ to $j=0$, so that the resulting solutions satisfy    a slightly better bound than \eqref{eq-448bound}:
\be \label{eq-448boundimp}
\Vert \tilde w_{\mathbf{a}} \Vert_{C^{2,{ \beta},  \sgamma  }_{ R_{\rho_1,\mathbf{a}}} } \le  2c_1(K-j) e^{(\sigma -  \sgamma  ) R_{\rho_1,\mathbf{a}}} .  
\ee

 \smallskip
 %Let  us fix       
% \be
% \sgamma   := \max (\beta ^+_{K-1} , \sigma -4\alpha ) +\delta 
% \ee 
% for  some fixed  constant  $0<\delta<\delta_0$   (depending on 
% $\alpha$ and  $h$) such that  $\sgamma +\delta<\sigma\leq \beta_K^+$. 
%    Here, $\delta_0>0$ is  a constant given in Corollary \ref{cor-ex-linear}. 
 
    In order to    choose appropriate constants $\rho_1$ and $c_1$,  {let us select a small   constant  $c_0>0$  from \eqref{eq-lem-error-esti1} in Lemma \ref{lem-error-estimate}  which  is     satisfied  with   both $\beta=\frac{1}{3},\frac12$.}  
    Let  $C_0$  be the constant from  the invertibility estimates of Jacobi operator in Corollary \ref{cor-ex-linear},
 $ C_1$ be 
the constant from the error estimates in Lemma \ref{lem-error-estimate},  and $C_2$ be a constant such that
\bea \label{eq-c2}
\Vert   \varphi_j e^{\beta^+_j s} \Vert_{ C^{2, { \beta},\beta^+_j}_0} \le C_2\qquad \hbox{for all  $  j\in \{1,\ldots,   K-1 \} $  }.\eea 
  Here and below, the constants appearing in uniform  estimates with the H\"older norm    are chosen to  satisfy those estimates  for    both the exponents $\beta=\frac{1}{3},\frac12$. Moreover, 
  we  shall often   abbreviate   estimates with   ``$\|\cdot\|_{k,\beta, \gamma}$ satisfied for both $\beta=\frac{1}{3},\frac12$ " to  those only with ``$\|\cdot\|_{k,\beta, \gamma}$", omitting which $\beta$'s are considered. 
We also notice that $\Vert E(0) \Vert_{C_{0}^{0,{ \beta}, \sgamma }}<\infty$  from Lemma \ref{lem-error-estimate}. Throughout the proof,  we fix   a small  constant $c_1$ such that
\be\label{eq-c-prefix}
 c_1  C_0  C_1(2K+ 5) (1+C_2)\le 1/2 , \quad\hbox{and}\quad 2c_1K\le c_0,
 \ee
  and then fix  a large positive constant $\rho_1$ satisfying 
   \be \label{eq-R0-prefix}
   e^{(\sgamma-\sigma)\rho_1}  C_0 \Vert E(0) \Vert_{C^{0,{ \beta},\sgamma }_{0} } \le c_1 , \quad e^{-4\alpha \rho_1}\le c_1 , \quad\hbox{and}\quad  e^{(\beta^+_{K-1} -\sigma)\rho_1}C_2 \le c_1. 
\ee 
for    $\beta=\frac{1}{3},\frac12$.Later, $c_1$ may be replaced by smaller ones as shown in \eqref{eq-conditioneee} and the explanation following \eqref{eq-barww0}, and $\rho_1$ may be replaced by a larger one as per \eqref{eq-conditioneee}. Here, we do not impose such conditions to avoid extra complications.
\medskip

%{\color{blue} Throughout the proof, the H\"older exponents  $\beta_i$     ($i=1,2$) will be denoted by $\beta$ for simplicity, and uniform constants in the estimates below depend only on $\alpha$, $h$,  $\beta_1, $  and $\beta_2$. }

(i)   First, we construct a reference translator   $\tilde w_{\mathbf{0}}$ by  a fixed point argument with the operator $H_{K,\rho_1}$  and the zero function for  the initial iteration.  
 We inductively define 
\bea \label{eq-u1-00}
u_0:=H_{K,{ \rho_1}}\Big(E(0)\Big) ,\qquad u_1:=H_{K,{ \rho_1}}\Big(E( u_0)- E(0)\Big),         
\eea 
and  
\bea \label{eq-uk0}
u_{k+1}:= H_{K,{ \rho_1}} \left(E\Big( \sum_{i=0}^{k} u_i \Big)-E\Big ( \sum_{i=0}^{k-1} u_i \Big ) \right)\quad \hbox{for $k\ge1$ }.
\eea 
 (To define $u_k$ as above,  we need to check that $E\Big( \sum_{i=0}^{k} u_i \Big)-E\Big ( \sum_{i=0}^{k-1} u_i \Big )$ belongs to $C^{0,\frac12,\sgamma}_{\rho_1}$ and this   follows from      \eqref{eq-ind11} below.)
 Note also that the solution given    by  Lemma \ref{lem-ex-linear} has the explicit expression and the solutions are identical for different $\beta$.
Then we will  show that  $\sum_{i=1}^{\infty} u_i$ converges in $C^{2,\beta,\sgamma}_{\rho_1}$ for    $\beta=\frac{1}{3},\frac12$, and $\tilde w_{\mathbf{0}}:=\sum_{i=0}^{\infty} u_i$ solves $\hat L \tilde w_{\mathbf{0}}=E(\tilde w_{\mathbf{0}})$ on $ (\rho_1  ,\infty)\times \mathbb{S}^1 $.   
\smallskip

By Corollary \ref{cor-ex-linear}  and the choice of $\rho_1$, \be
\Vert  u_0 \Vert_{C_{\rho_1}^{2,\beta,\sigma}}\leq e^{(\sgamma-\sigma)\rho_1} \Vert  u_0 \Vert_{C_{\rho_1}^{2,\beta,\sgamma}} \le  e^{(\sgamma-\sigma)\rho_1} C_0  \Vert E(0)\Vert_{C_{\rho_1}^{0,\beta,\sgamma}} \le   c_1 \le c_0  
\ee
 for     $\beta=\frac{1}{3},\frac12$.
Then by Corollary \ref{cor-ex-linear}  and  Lemma \ref{lem-error-estimate}, we have
\bea\label{eq-est-u1-121}
\Vert  u_1 \Vert_{C_{\rho_1}^{2,\beta,\sgamma}} \le C_0 \Vert E(u_0)-E(0) \Vert_{C_{\rho_1}^{0,\beta,\sgamma} } 
 &\leq    C_0   C_1  \left\{  e^{ (\sgamma-\sigma) {\rho_1}} \Vert u_0\Vert_{C_{\rho_1}^{2,\beta,\sgamma}}  +e^{-4\alpha {\rho_1}} \right\} \Vert u_0\Vert_{C_{\rho_1}^{2,\beta,\sgamma}}\\
&\leq 2 C_0 C_1 c_1 \Vert u_0\Vert_{C_{\rho_1}^{2,\beta,\sgamma}}   <\Vert u_0\Vert_{C_{\rho_1}^{2,\beta,\sgamma}}   /2 
\eea
 for     $\beta=\frac{1}{3},\frac12$.
Moreover,  we will prove    by an induction   that  $u_{i}$ given  in \eqref{eq-uk0} satisfy 
\be\label{eq-ind11}
 \Vert u_{i}  \Vert_{C_{\rho_1}^{2,\beta,\sgamma}} \le 2^{-{i}}  \Vert u_0\Vert_{C_{\rho_1}^{2,\beta,\sgamma}}   \qquad \forall i\geq1
\ee 
 for     $\beta=\frac{1}{3},\frac12$.
Note that this, in particular, implies 
\be
 \Vert u_{i}  \Vert_{C_{\rho_1}^{2,\beta,\sigma}} \le e^{(\sgamma -\sigma)  \rho_1}   \Vert u_{i}  \Vert_{C_{\rho_1}^{2,\beta,\sgamma}} \le 2^{-i} e^{(\sgamma -\sigma)  \rho_1}   \Vert u_{0}  \Vert_{C_{\rho_1}^{2,\beta,\sgamma}}   \leq  2^{-{i}}    c_1.  
 \ee  
  Suppose  that  \eqref{eq-ind11} is true for all $1\leq i\leq k$. 
From the assumption,  it follows that  $
  \sum_{i=0}^k \Vert u_{i}  \Vert_{C_{\rho_1}^{2,\beta,\sigma}}     \leq 2c_1\le c_0$  for     $\beta=\frac{1}{3},\frac12$.
Utilizing  Corollary \ref{cor-ex-linear}  and   Lemma \ref{lem-error-estimate}    yields  that 
\be
 \ba
 \Vert u_{k+1} \Vert_{C_{\rho_1}^{2,\beta,\sgamma}} &\le  C _0 \left\Vert  E\Big( \sum_{i=0}^{k} u_i \Big)-E\Big (\sum_{i=0}^{k-1} u_i \Big )  \right\Vert_{C_{\rho_1}^{0,\beta,\sgamma}}  \\ 
 &\le  C_0   C_1  \Big\{  e^{ (\sgamma-\sigma) {\rho_1}}  \sum_{i=0}^{k } \Vert  u_i \Vert_{C_{\rho_1}^{2,\beta,\sgamma}}  +e^{-4\alpha {\rho_1}} \Big\} \Vert u_k\Vert_{C_{\rho_1}^{2,\beta,\sgamma}}\\ 
  &\le  C_0   C_1  (   2c_1 +c_1 ) \Vert u_k\Vert_{C_{\rho_1}^{2,\beta,\sgamma}} <  \Vert u_k\Vert_{C_{\rho_1}^{2,\beta,\sgamma}}/2  
 \ea
\ee  for     $\beta=\frac{1}{3},\frac12$.
Thus, we have proved  \eqref{eq-ind11} since   \eqref{eq-ind11} is true for   $i=1$  by \eqref{eq-est-u1-121}.    Thus    $\tilde w_{\mathbf{0}}= \sum_{i=0}^\infty u_i$ satisfies 
    \be
\Vert \tilde w_{\mathbf{0}} \Vert_{C^{2,\beta,  \sgamma  }_{ R_{\rho_1,\mathbf{0}}} }\le 2 \Vert u_0\Vert_{C_{\rho_1}^{2,\beta,\sgamma}} \le 2  c_1 e^{(\sigma -  \sgamma  )  { R_{\rho_1,\mathbf{0}}}}
\ee for     $\beta=\frac{1}{3},\frac12$, and  $\hat L \tilde w_{\mathbf{0}}=E(\tilde w_{\mathbf{0}})$ on $ (R_{\rho_1,\mathbf{0}}  ,\infty)\times \mathbb{S}^1 $. 
 
 \medskip

 (ii) Next, suppose that $\tilde  w _{\mathbf{a}} \in C^{2,{ \frac12},\sgamma}_{ R_{\rho_1,\mathbf{a}}}$ is a  solution    with some  $\mathbf{a}=(\mathbf{0}_{j+1}, \mathbf{a}_{K-j-1})\in\mathbb{R}^K$ (for some $0\le j\leq K-1$) that satisfies  
\be \label{eq-barwa}
e^{(\sgamma -\sigma) R_{\rho_1,\mathbf{ a}}} \Vert \tilde w_{\mathbf{a}} \Vert _{C^{2,\beta,\sgamma}_{ R_{\rho_1,\mathbf{a}}}}   \le 2(K-j)c_1 
\ee  for     $\beta=\frac{1}{3},\frac12$.
  For a given  $\mathbf{a}'=(\mathbf{0}_{j}, a_j, \mathbf{a}_{K-j-1}) $ with $a_j\not=0$,     it suffices to construct  a solution  $\tilde w_{\mathbf{a}'} \in C^{2,{\frac12} ,\sgamma}_{ R_{\rho_1,\mathbf{a}'}}$ satisfying  the estimate  
\be \label{eq-barwbara}
e^{(\sgamma  -\sigma) R_{\rho_1,\mathbf{   a}' }} \Vert \tilde w_{\mathbf{ a}' } -\tilde w_{\mathbf{ a} }  \Vert _{C^{2,\beta,\sgamma}_{ R_{\rho_1,\mathbf{ a}' }}}   \le 2c_1
\ee 
for     $\beta=\frac{1}{3},\frac12$, and \eqref{eq-wdiff}. 
We define  $\tilde w_{\mathbf{a}'}$ by 
\bea 
\tilde w_{\mathbf{a}'}= \tilde w_{\mathbf{a}} + \sum_{i=0}^{\infty} u_i\qquad \hbox{with\,\, \,$u_0:= a_j \varphi_j e^{\beta^+_j s}$ ,}\eea 
where the  remaining $u_{i}$'s in $  C^{2,{\frac12} ,\sgamma}_{ R_{\rho_1,\mathbf{a}'}}$  are inductively defined by
\bea \label{eq-deftildeu1}
u_1=   H_{j,  R_{\rho_1,\mathbf{a}'}} \Big(E(\tilde w_{\mathbf{a}}+u_0) -E(\tilde w_{\mathbf{a}})\Big), \quad  \quad   
 u_{k+1}:= H_{j,  R_{\rho_1,\mathbf{a}'}} \Big( E(\tilde w_{\mathbf{a}}+\sum_{i=0}^k u_i)-E(\tilde w_{\mathbf{a}}+\sum_{i=0}^{k-1} u_i)  \Big)\eea 
 for $k\ge1$. 
Here,   we conveniently interpret $\tilde w_{\mathbf{a}}$ and $a_j \varphi_j e^{\beta^+_j s}$ as functions defined on $ [ R_{\rho_1,\mathbf{a}'},\infty)\times \mathbb{S}^1$ by restricting their domains since $ R_{\rho_1,\mathbf{a}'} \ge  R_{\rho_1,\mathbf{ a}}$. Then we will show   that 
\be \label{eq-barwa-20}
\Vert u_{1} \Vert_{C^{2,\beta,\beta^+_j-\delta}_{ R_{\rho_1,\mathbf{ a}' }}}   \le [2(1+C_2)]^{-1} e^{\delta R_{\rho_1,\mathbf{ a}' } }\Vert u_{0} \Vert_{C^{2,\beta,\beta^+_j}_{ R_{\rho_1,\mathbf{ a}' }}}  , 
  \ee 
for     $\beta=\frac{1}{3},\frac12$, and by an induction,  
 \be \label{eq-barwa-2'}
  \Vert u_{i+1} \Vert_{C^{2,\beta,\beta^+_j-\delta}_{  R_{\rho_1,\mathbf{ a}' }}}  \le 2^{-1}   \Vert u_{i} \Vert_{C^{2,\beta,\beta^+_j-\delta}_{ R_{\rho_1,\mathbf{ a}' }}} \qquad \forall i\ge1  . 
 \ee
for     $\beta=\frac{1}{3},\frac12$.
 
 To begin with,   by the choice  of $C_2$,  we  have  
\be  \label{eq-u0c00}
  \Vert u_0 \Vert_{C^{2,\beta, \sigma  }_{ R_{\rho_1,\mathbf{ a}' }}} \le e^{(\beta^+_j -\sigma) R_{\rho_1,\mathbf{ a}' } }  \Vert u_0 \Vert_{C^{2,\beta, \beta^+_j }_{ R_{\rho_1,\mathbf{ a}' }}}  
  %\le \frac{e^{(\beta^+_j -\sigma)  \rho_1 } } {|a_j|+1} \cdot C _2 |a_j|
  \le  e^{(\beta^+_j -\sigma)  R_{\rho_1,\mathbf{ a}' } } |a_j|  C _2 %\le C _2e^{(\beta^+_j -\sigma)   \rho_1 }\le c_1\le c_0 . 
  \ee 
for     $\beta=\frac{1}{3},\frac12$.
Since 
 $  R_{\rho_1,\mathbf{a}'} = \rho_1+ \ln (\vert \mathbf{b}(\mathbf{a}')\vert+1)$, $ |b_j|=|a_j|^{1/({\sigma-\beta^+_j})}$, and   $\beta^+_j\le \beta^+_{K-1}$, we obtain
 \be  \label{eq-u0c0}
  \Vert u_0 \Vert_{C^{2,\beta, \sigma  }_{ R_{\rho_1,\mathbf{ a}' }}} \le e^{(\beta^+_j -\sigma) R_{\rho_1,\mathbf{ a}' } }  \Vert u_0 \Vert_{C^{2,\beta, \beta^+_j }_{ R_{\rho_1,\mathbf{ a}' }}}  
  %\le \frac{e^{(\beta^+_j -\sigma)  \rho_1 } } {|a_j|+1} \cdot C _2 |a_j|
   %\le  e^{(\beta^+_j -\sigma)  R_{\rho_1,\mathbf{ a}' } } |a_j|  C _2
   \le  e^{(\beta^+_j -\sigma)   \rho_1 } C _2\le c_1\le c_0   \ee 
for     $\beta=\frac{1}{3},\frac12$.
% \be \label{eq-barwa-2}
%  \Vert u_{k} \Vert_{C^{2,\beta,\beta^+_j-\delta}_{\tilde R_{\mathbf{ a}' }}}  \le 2^{-k} e^{   {   \delta}  \tilde R_{\mathbf{   a}' }} |a_j|\qquad\hbox{ for all $k\ge1$. } 
% \ee
Once  \eqref{eq-barwa-20}-\eqref{eq-barwa-2'} are shown, we deduce that 
\bea \label{eq-conv-u_i}
\Vert \sum _{i=0}^\infty u_i \Vert _{C^{2,\beta, \sgamma  }_{ R_{\rho_1,\mathbf{ a}' }}} \le   e^{(\beta^+_j -\sgamma ) R_{\rho_1,\mathbf{ a}' } }\sum_{i=0}^\infty \Vert u_i \Vert_{C^{2,\beta, \beta^+_j  }_{ R_{\rho_1,\mathbf{ a}' }}}\le  e^{(\beta^+_j -\sgamma ) R_{\rho_1,\mathbf{ a}' } } 2 \Vert u_0\Vert _{C^{2,\beta, \beta^+_j }_{ R_{\rho_1,\mathbf{ a}' }}} \le 2c_1 e^{(\sigma  -\sgamma ) R_{\rho_1,\mathbf{ a}' } }  ,\eea 
and 
\bea
\Vert \sum _{i=1}^\infty u_i \Vert _{C^{2,\beta,\beta^+_j-\delta  }_{ R_{\rho_1,\mathbf{ a}' }}}  \le 2\Vert u_1 \Vert_{C^{2,\beta,\beta^+_j-\delta  }_{ R_{\rho_1,\mathbf{ a}' }}} \le (1+C_2)^{-1} e^{\delta R_{\rho_1,\mathbf{ a}' } }   \Vert u_0 \Vert_{C^{2,\beta,\beta^+_j }_{ R_{\rho_1,\mathbf{ a}' }}}  \le   e^{\delta R_{\rho_1,\mathbf{ a}' } } |a_j| 
\eea
for     $\beta=\frac{1}{3},\frac12$, 
which imply  
\eqref{eq-barwbara} and \eqref{eq-wdiff}.

 \medskip

So, it remains to prove \eqref{eq-barwa-20} and \eqref{eq-barwa-2'}. By  \eqref{eq-barwa} and \eqref{eq-u0c0}, we apply Corollary  \ref{cor-ex-linear}  and   Lemma \ref{lem-error-estimate}   to get 
\bea
&\Vert  u_1 \Vert_{C^{2,\beta,\beta^+_j-\delta}_{ R_{\rho_1,\mathbf{ a}' }}}
 \le C _0\Vert E(\tilde w_{\mathbf{a}}+u_0) -E(\tilde w_{\mathbf{a}}) \Vert_{C^{0,\beta,\beta^+_j-\delta}_{ R_{\rho_1,\mathbf{ a}' }}} \\
 %&\leq  C_1  \left(  e^{   (\sgamma { +\delta}-\sigma) \tilde R_{\mathbf{   a}' }}( \Vert \tilde w_{\mathbf{ a}} \Vert_{C_{\tilde R_{\mathbf{   a}' }}^{2,\beta,\sgamma}}  +\Vert u_0\Vert_{C_{\tilde R_{\mathbf{   a}' }}^{2,\beta,\sgamma}}    )+e^{-4\alpha \tilde R_{\mathbf{   a}' }} \right) { \Vert u_0\Vert_{C^{2,\beta,\beta^+_j }_{\tilde R_{\mathbf{ a}' }}}}\\
  &\leq  { C_0 }C_1  e^{   {  \delta}   R_{\rho_1,\mathbf{   a}' }}  \Big\{  e^{   (\sgamma  -\sigma)  R_{\rho_1,\mathbf{   a}' }} \Vert \tilde w_{\mathbf{ a}} \Vert_{C_{ R_{\rho_1,\mathbf{   a}' }}^{2,\beta,\sgamma}}  +  e^{   (\beta_j^+  -\sigma)  R_{\rho_1,\mathbf{   a}' }}\Vert u_0\Vert_{C_{ R_{\rho_1,\mathbf{   a}' }}^{2,\beta,\beta_j^+}}     +e^{-4\alpha  R_{\rho_1,\mathbf{ a}' }} \Big\}  { \Vert u_0\Vert_{C^{2,\beta,\beta^+_j }_{ R_{\rho_1,\mathbf{ a}' }}}}\\
    &\leq { C_0 } C_1  e^{   {  \delta}   R_{\rho_1,\mathbf{   a}' }}   \left(  2Kc_1+c_1+c_1  \right)  { \Vert u_0\Vert_{C^{2,\beta,\beta^+_j }_{ R_{\rho_1,\mathbf{ a}' }}}}
\eea
  for     $\beta=\frac{1}{3},\frac12$. Here, we  recall that  $\delta $ is given by \eqref{eq-deltagap} when applying the error estimate in   Lemma \ref{lem-error-estimate}. 
This shows \eqref{eq-barwa-20} since $C_0C_1(2K+2)c_1 \le [2(1+C_2)]^{-1}$ by the choice of $c_1$.    
Next,  suppose that  \eqref{eq-barwa-2'} holds for all $2\le i\le k$ with a  given $k\ge2$.  By \eqref{eq-barwa-20}, \eqref{eq-u0c0} and the induction hypothesis, it follows that 
 \be  
 \sum _{i=0}^{k} \Vert u_i \Vert _{C^{2,\beta, \sigma  }_{ R_{\rho_1,\mathbf{ a}' }}}\le \sum _{i=0}^k  e^{(\beta^+_j -\sigma) R_{\rho_1,\mathbf{ a}' } }  \Vert u_i \Vert_{C^{2,\beta, \beta^+_j }_{ R_{\rho_1,\mathbf{ a}' }}} \le 2 e^{(\beta^+_j -\sigma) R_{\rho_1,\mathbf{ a}' } }  \Vert u_0 \Vert_{C^{2,\beta, \beta^+_j }_{ R_{\rho_1,\mathbf{ a}' }}}  \le 2 c_1 \le c_0  
 \ee 
for     $\beta=\frac{1}{3},\frac12$. Applying  Corollary  \ref{cor-ex-linear}  and   Lemma \ref{lem-error-estimate}  again,    we  obtain 
\bea 
&\Vert u_{k+1}\Vert_{C^{2,\beta,\beta^+_j-\delta}_{ R_{\rho_1,\mathbf{ a}' }}}\le C_0 \Vert E(\tilde w_{\mathbf{a}}+\sum_{i=0}^k u_i)-E(\tilde w_{\mathbf{a}}+\sum_{i=0}^{k-1} u_i) \Vert _{C^{0,\beta,\beta^+_j-\delta}_{ R_{\rho_1,\mathbf{ a}' }}}\\
 &\le   C_0 C_1 \Big\{   e^{( \sgamma  -\sigma)  R_{\rho_1,\mathbf{ a}' }}  \Vert \tilde w_{\mathbf{a}}\Vert_{C^{2,\beta,\sgamma  }_{ R_{\rho_1,\mathbf{ a}' }} }  +\sum_{i=0}^k  e^{( \beta^+_j  -\sigma)  R_{\rho_1,\mathbf{ a}' }}\Vert u_i\Vert_{C^{2,\beta,\beta^+_j }_{ R_{\rho_1,\mathbf{ a}' }} }  + e^{-4\alpha { R_{\rho_1,\mathbf{ a}' }}}\Big\}\Vert u_k \Vert_{C^{2,\beta,\beta^+_j -\delta}_{ R_{\rho_1,\mathbf{ a}' }}} ,\\
\eea 
and hence  
\bea 
\Vert u_{k+1}\Vert_{C^{2,\beta,\beta^+_j-\delta}_{ R_{\rho_1,\mathbf{ a}' }}}  
%&\le  C \left[   2(K-j) c_0 + e^{(\beta^+_{j} -\sigma )\tilde R_{\mathbf{ a}' }} \sum_{i=0}^k \Vert u_i\Vert_{C^{2,\beta,\beta^+_{j}  }_{\tilde R_{\mathbf{ a}' }} }  + e^{-4\alpha \tilde R_{\mathbf{ a}' }}\right ]\Vert u_k \Vert_{C^{2,\beta,\beta^+_j-\delta}_{\tilde R_{\mathbf{ a}' }}}  \\
% &\le  C_1 \left[   2K c +  \frac{e^{(\beta^+_{j} -\sigma ) \rho_1}}{|a_j|+1} \sum_{i=0}^k \Vert u_i\Vert_{C^{2,\beta,\beta^+_{j}  }_{\tilde R_{\mathbf{ a}' }} }  + e^{-4\alpha \tilde R_{\mathbf{ a}' }}\right ]\Vert u_k \Vert_{C^{2,\beta,\beta^+_j-\delta}_{\tilde R_{\mathbf{ a}' }}}   \\
%&\le  C_1 \left[   2K c +   2 {e^{(\beta^+_{j} -\sigma ) \rho_1}}  + e^{-4\alpha \tilde R_{\mathbf{ a}' }}\right ]\Vert u_k \Vert_{C^{2,\beta,\beta^+_j-\delta}_{\tilde R_{\mathbf{ a}' }}} \\
&\le   C_0 C_1 ( 2K c_1 +  2c_1 + c_1)\Vert u_k \Vert_{C^{2,\beta,\beta^+_j-\delta}_{ R_{\rho_1,\mathbf{ a}' }}} \le  2^{-1} \Vert u_k \Vert_{C^{2,\beta,\beta^+_j-\delta}_{ R_{\rho_1,\mathbf{ a}' }}} 
\eea 
for     $\beta=\frac{1}{3},\frac12$.
Since  the case when  $k=1$ can be proved by the same argument,  we deduce that    \eqref{eq-barwa-2'} holds true.  
This completes the proof. 
\end{proof}

%%%%%%%%%%%%%%%%%%%%%%%%%%%%%%%%%%%%%%%%
\section{Construction of global barriers}\label{sec-global-bar}
%%%%%%%%%%%%%%%%%%%%%%%%%%%%%%%%%%%%%%%%%
The next goal is to fill in translators existing on exterior domains, given  in Theorem \ref{thm-existencerevised11}, and obtain complete translators. This goal is achieved by constructing global barriers which are asymptotic to exterior translators. 
In this section,   we  first construct a translator   on the exterior of a  large ball  %$ B_{\hat R_{\mathbf{a}}}$ 
$ B_{ R_{\rho_2,\mathbf{a}}}$  (for $  \rho_2 >      \rho_1 $),   that has a prescribed  behavior of  its    value  and radial derivative of    $S(e^s,\theta)/h(\theta)$ on the boundary $\partial B_{ R_{\rho_2,\mathbf{a}}}$ in   Proposition \ref{prop-exist-exterior'1}. 
This technical result will be used  to construct  global upper and lower  barriers in Proposition \ref{prop-global-barriers}. 

 For this task, the constants 
$c_1$ and $\rho_1$ in  \eqref{eq-c-prefix}-\eqref{eq-R0-prefix} need to satisfy an additional condition. By possibly replacing with smaller $c_1$ and larger $\rho_1$, let us further assume that
\be \label{eq-conditioneee}
{ C_0 }C_1 \left[   2  c_1 K+ 2 c_1  + e^{-4\alpha  \rho_1}\right ] <\frac{\sigma\eae }{4(  \inf_{\mathbb{S}^{1}} h+   \sup_{\mathbb{S}^{1}} h)},\quad \text{where}\quad \eae:=\frac{1}{4} \left( 2-\frac{1}{\sigma}\right) .\ee 
Here, the constants  $C_0$ and  $C_1$  are given
  by Corollary \ref{cor-ex-linear},
 and Lemma \ref{lem-error-estimate}, % {\color{red}with $\beta=\beta_i$ for both $i=1,2$},
 respectively,  as before  for both     $\beta=\frac{1}{3},\frac12$.  
 Let us also fix a  small constant  $c_2>0$  such that 
% \be \label{eq-conditionee} c_2 \le \frac {c_1} {  \inf_{\mathbb{S}^{1}  }h+ \|h\|_{L^\infty(\mathbb{S}^1)}   },  \quad  
%  (1-\sigma c_2)^{ \frac 1\sigma -1 } >1- \left( \sigma^{-1} -1  + \eae\right)\sigma c_2,\quad  (1+\sigma c_2)^{ \frac 1\sigma -1 }  < 1 +\left( \sigma^{-1} -1  + \eae\right)\sigma c_2 ,   \ee and
%\be\label{eq-conditionee2} (1-\eae)(\sigma^{-1}-1+\eae)\sigma c_2<\eae.\ee 
\be \label{eq-conditionee}
 c_2 \le \frac {c_1} {  \inf_{\mathbb{S}^{1} 
  }h+  \sup_{\mathbb{S}^{1}}h   },   
 \ee  
 \be\label{eq-conditionee2}  
  (1-\sigma c_2)^{ \frac 1\sigma -1 } >1- \left( \sigma^{-1} -1  + \eae\right)\sigma c_2, \quad (1-\eae) (\sigma^{-1}-1+\eae)\sigma c_2<\eae, 
 \ee 
  \be\label{eq-conditionee3}  
 \hbox{and}\quad(1+\sigma c_2)^{ \frac 1\sigma -1 }  < 1 +\left( \sigma^{-1} -1  + \eae\right)\sigma c_2 
  .\ee 
Here, the last three conditions will be used  when constructing global barriers in Proposition \ref{prop-global-barriers}.  
\smallskip

In the following, we obtain  exterior
translators  having a prescribed behavior of its value and radial derivative on the boundary.

 \begin{prop}[Barriers on exterior domains]\label{prop-exist-exterior'1}  
For $\mathbf{a}\in \mathbb{R}^K$, let    {$\tilde  w _{\mathbf{a}} \in C^{2,{ \frac12},\sgamma}_{ R_{\rho_1,\mathbf{a}}}$}  %   $\tilde w_{\mathbf{a}} \in  {   C^{2,\beta,\sgamma }_{ R_{\rho_1,\mathbf{a}}}  }$  
represent  the translator  
obtained in Theorem \ref{thm-existencerevised11} with   $R_{\rho_1,\mathbf{a}}  =\rho_1+ \ln (\vert \mathbf{b}(\mathbf{a})\vert+1)$. 
  Then, there is a constant $\rho_2  >  \rho_1$ (depending on  $\alpha$ and  $h$) such that  for each  $\mathbf{a}\in \mathbb{R}^K$ and for each $\rho \ge \rho_2$,  there exist    solutions  $w_+ $ and $ w_- $ in   $ C^{2,\frac12, -\frac12  }_{ R_{\rho,\mathbf{a}}} $   to the equation
 \be
\label{eq-prop27'1}
\hat L( \tilde w _{\mathbf{a}}+w_\pm)=E( \tilde w _{\mathbf{a}}+w_\pm) \qquad \text { on }\,\, ( R_{\rho,\mathbf{a}},\infty)\times \mathbb{S}^1
\ee
 with $R_{\rho,\mathbf{a}}= \rho   +   \ln (\vert \mathbf{b}(\mathbf{a})\vert+1)$, 
which satisfy the following    boundary conditions: %{\color{red} given  a sufficiently  small  $0<\eae<\frac{1-4\alpha}{\sigma}$ and $0<c_1\ll c_0< \eae$, }
\be\label{eq-exist-ex-bd1'1}
\left\{\ba
& %\frac{g+u}h \text{  is a constant on } \p B_R\,\, \text{ and }\,\,
 \frac{\tilde w _{\mathbf{a}}+w_\pm}h = \mp c_2 e^{\sigma{ R_{\rho,\mathbf{a}}} } \qquad &\text{ on }\,\,  s={ R_{\rho,\mathbf{a}}},\\
&  \pm \frac{\p }{\p s } \Big(\frac{\tilde w _{\mathbf{a}}+w_\pm}{h}\Big) >  (1-\eae){\sigma c_2}{ } e^{\sigma { R_{\rho,\mathbf{a}}}} \qquad &\text{ on } \,\, s={ R_{\rho,\mathbf{a}}} .
\ea\right.
\ee 
Moreover,  we have 
  \be\Vert w_\pm \Vert _{C^{2,  {\beta},-\frac12}_{ R_{\rho, \mathbf{a}}}}  \le 2  c_1 e^{(\sigma+\frac12){{ R_{\rho,\mathbf{a}}}}} \qquad (\beta=\tfrac13,\tfrac12).
  \ee
\end{prop}

\begin{proof}
It suffices to construct a  solution $w_+$ to \eqref{eq-prop27'1} %with  the      boundary conditions in   \eqref{eq-exist-ex-bd1'1}
 since the other is similar. For  each  $\rho > \rho_1$, 
 there exists a   solution  $\hat   g$  in  $  C^{2,\frac12,\beta^-_0}_{ R_{\rho,\mathbf{a}}} $  %with ${\hat R_{\mathbf{a}}}=\rho + \ln (\vert \mathbf{b}(\mathbf{a})\vert+1) $
 to the problem 
\be\label{eq-hat-g-def}
\left\{
\ba 
&\hat L(\hat g) =0\qquad &&\text{ on \,\, $( R_{\rho,\mathbf{a}},\infty)\times \mathbb{S}^1$} ,\\
&\hat  g = -\tilde w_{\mathbf{a}} \qquad &&\text{ on  \,\,$\{s= R_{\rho,\mathbf{a}}\}\times \mathbb{S}^1$} ,
\ea \right.
\ee
where we recall $\beta^-_0=-\sigma$. 
More specifically,  let $\tilde w_{\mathbf{a}} ({ R_{\rho, \mathbf{a}}} ,\cdot) =\displaystyle \sum_{j=0}^\infty d_j  \varphi_j(\cdot) $ with some  constants  $d_j\in\mathbb{R}$ for $j\geq0$. 
Here,   $\{\varphi_j \}_{j=0}^\infty $  is an orthonormal basis of $L^2_h(\mathbb{S}^1)$  consisting of  eigenfunctions of $L$ solving  \eqref{eq-lambdavarphi}.  Then  a solution $ \hat g$ can be explicitly written as  
 \bea \label{eq-hatgdef}
 \hat g(s,\theta) = - \sum _{j=0}^\infty d_j e^{\beta^-_j (s-{ R_{\rho,\mathbf{a}}} )} \varphi_j(\theta)\quad \text{ on    $[ R_{\rho,\mathbf{a}},\infty)\times \mathbb{S}^1$}. 
 \eea 
 Here,   it can be shown that  $\hat g$ belongs to $  C^{2,\frac12,\beta^-_0}_{ R_{\rho,\mathbf{a}}} $ by  using   (interior and boundary) $W^{2,2}$ and Schauder estimates  for elliptic equations. Indeed, it holds that  for each $\rho\geq \rho_2$,    
\be \label{est-hat-g}
 e^{    \beta^-_0   R_{\rho,\mathbf{a}}}     \Vert\hat g  \Vert _{C^{2,\beta ,\beta^-_0 }_ {{ R_{\rho,\mathbf{a}}}} } \leq Ce^{  \sgamma   R_{\rho,\mathbf{a}}}    \Vert \tilde w_{\mathbf{a}} \Vert_{C^{2,\beta,  \sgamma  }_{ R_{\rho,\mathbf{a}}} }  \leq C    e^{(\sigma -  \sgamma  )(  \rho_1-\rho)}  e^{ \sigma     R_{\rho,\mathbf{a}}}  
\ee   with  $\beta=\frac13,\frac12$, 
 since  $ \Vert \tilde w_{\mathbf{a}} \Vert_{C^{2,\beta,  \sgamma  }_{ R_{\rho_1,\mathbf{a}}} } \le e^{(\sigma -  \sgamma  ) R_{\rho_1,\mathbf{a}}}  $ by Theorem \ref{thm-existencerevised11}.  Here, a constant    $C>1$ 
 is chosen to satisfy regularity estimates   for both    $\beta=\frac13,\frac12$, and it depends on $\alpha$ and  $h$.
As in the proof of Theorem \ref{thm-existencerevised11}, we 
 may     omit to mention   that   constants appearing  uniform estimates  are selected  to satisfy those for both $\beta=\frac{1}{3},\frac12$.
 By choosing  $\rho_2> \rho_1$ sufficiently large, this implies that  for each $\rho\geq \rho_2$, 
\be\label{eq-est-hat-g}
 e^{    \beta^-_0   R_{\rho,\mathbf{a}}}     \Vert\hat g  \Vert _{C^{2,\beta ,\beta^-_0 }_ {{ R_{\rho,\mathbf{a}}}} } + e^{  \sgamma  R_{\rho,\mathbf{a}}}    \Vert \tilde w_{\mathbf{a}} \Vert_{C^{2,\beta,  \sgamma  }_{ R_{\rho,\mathbf{a}}} }   < {  \frac{ \,{ 1 }  }{2  }( \inf_{\mathbb{S}^{1}} h)  \eae  \sigma     c_2      e^{   \sigma     R_{\rho,\mathbf{a}}}} ,
\ee   with   $\beta=\frac13,\frac12$, 
 where  $\rho_2$ depends only on $\alpha$ and  $h$. 
 \medskip

 In order to construct a solution $ u$, we 
 first define
\bea\label{eq-u_0outer} 
 u_0(s,\theta) := \hat  g (s,\theta)- { c_2   e^{\sigma{ R_{\rho,\mathbf{a}}} } e^{\beta^-_0 (s- { R_{\rho,\mathbf{a}}} )}h(\theta)}\qquad\hbox{on  $[ R_{\rho,\mathbf{a}},\infty)\times \mathbb{S}^1$}.
 \eea
Using \eqref{eq-est-hat-g}  and $\beta^-_0 = -\sigma $, 
we obtain 
\be\label{eq-u_0cases1}
 \begin{cases}
  \ba &\hat L (u_0) = 0 && \text{on\,\, $(R_{\rho,\mathbf{a}},\infty)\times \mathbb{S}^1$} , \\  
  &\frac{\tilde w_{\mathbf{a}}+u_0}{h} = -c_2 e^{\sigma{ R_{\rho,\mathbf{a}}} } &&\text{on  \,\,$\{s= R_{\rho,\mathbf{a}}\}\times \mathbb{S}^1$}, \\
 &\frac{\p}{\p s} \Big(\frac{\tilde w_{\mathbf{a}}+u_0}{h} \Big){ > \Big(1-\frac{\eae }
 {2}\Big)\sigma c_2 e^{\sigma { R_{\rho,\mathbf{a}}}} } && \text{on  \,\,$\{s= R_{\rho,\mathbf{a}}\}\times \mathbb{S}^1$}.
 \ea 
 \end{cases} 
 \ee  
% {\color{red} Here, we note that 
% $$\| \hat g({\hat R_{\mathbf{a}}},\cdot)\|  +\|  \tilde w_{\mathbf{a}}({\hat R_{\mathbf{a}}},\cdot)\|   \lesssim   e^{(\sigma -  \gamma  )\tilde R_{\mathbf{a}}}  e^{  \gamma  \hat  R_{\mathbf{a}}}  \leq e^{(\sigma -  \gamma  )(  \rho_1-\rho_2)}  e^{  \sigma  \hat  R_{\mathbf{a}}}\leq  \frac{\eae }{4} c_2   \sigma e^{  \sigma  \hat  R_{\mathbf{a}}} $$
% by choosing  $\rho_2>\rho_1$.  }
   In view of Lemma \ref{lem-ex-linear} (with   $\beta_0^-=-\sigma< -\frac12< \beta_0^+=-2\alpha$),  
% we have  an     inverse operator      $ H  : C^{0,\beta,-\frac12}_R \rightarrow C_R^{2,\beta,-\frac12}  $     of $\hat L$  such that  for any $v \in   C^{0,\beta,-\frac12}_R$, the solution $H(v )$ satisfies  
%\be\label{eq-Heta}
%\begin{cases} 
%\ba
%&\hat L(H(v ))=v  \quad &&\hbox{on $B_R^c$,}\\
%& H(v ) =0 \quad &&\hbox{on $\p B_R$,}
%\ea\end{cases}
%\ee 
%and  
%\bea \label{eq-Hestimate}\Vert H(v) \Vert_{C^{2,\beta,-\frac12}_R} \le C \Vert v \Vert_{C^{0,\beta,-\frac12}_R}.\eea 
%Here and below, a constant $C>0$ is  independent of $R$, but may become larger during the proof.  
%Using the  operator $H $,   
we  inductively define  
\be \label{eq-u1} 
u_{1}:= H_{0,{ R_{\rho,\mathbf{a}}}}\Big(E({\tilde w_{\mathbf{a}}}+u_0)- E({\tilde w_{\mathbf{a}}})\Big),
\ee
and 
\be \label{eq-uk+1} 
u_{k+1}:= H_{0,{ R_{\rho,\mathbf{a}}}}\Big(E\Big({\tilde w_{\mathbf{a}}}+\sum_{i=0}^ku_i\Big)- E\Big({\tilde w_{\mathbf{a}}}+\sum_{i=0}^{k-1}u_i\Big)\Big)\quad\hbox{for all $k\geq1$}.
\ee 
(We will verify that  $E\Big({\tilde w_{\mathbf{a}}}+\sum_{i=0}^{k} u_i \Big)-E\Big ({\tilde w_{\mathbf{a}}}+\sum_{i=0}^{k-1} u_i \Big )$ belongs to {$ C^{0,\frac12,-\frac12}_{ R_{\rho,\mathbf{a}}}$} as before.) %,    which   will follow from the proof of Claim \ref{claim-uk} below.) 
\bigskip 

To  construct a desired solution $u$, we  claim that  %{\color{red}there is a sufficiently large  constant  $\rho_2>\rho_1$ such that  }
for   each  $\rho\geq \rho_2$, 
 \bea \label{eq-inductionstrong1}  
 \Vert u_{i+1}\Vert_{C^{2,\beta,-\frac12}_{ R_{\rho,\mathbf{a}}}} \le   2^{-i}   c_1  e^{(\frac12+\sigma){ R_{\rho,\mathbf{a}}}}\qquad \forall i\ge 0    \eea 	
  with    $\beta=\frac13,\frac12$.   To prove the claim,  
we  first   obtain  from \eqref{eq-est-hat-g} and the choice of $c_2$ in \eqref{eq-conditionee} that  
\bea \label{eq-est-u0-4.81}
\Vert u_0 \Vert _{C^{2,\beta ,-\frac12 }_ {R_{\rho,\mathbf{a}}} }&\leq e^{( \frac12+\beta^-_0) { R_{\rho,\mathbf{a}}}}\Vert u_0 \Vert _{C^{2,\beta ,\beta^-_0 }_ { R_{\rho,\mathbf{a}}} }% \leq {    C  \left[  e^{( \sigma-\sgamma ){\tilde  R_{\mathbf{a}}}} e^{(\sgamma +\frac12){{\hat R_{\mathbf{a}}}} } + c_2 e^{(\sigma +\frac12){{\hat R_{\mathbf{a}}}} }  \right]} \\
 \le  c_2 e^{(\frac12 +\sigma ){ R_{\rho,\mathbf{a}}}} \left[ \frac{ \,{ 1 }  }{2  }( \inf_{\mathbb{S}^{1}} h)  \eae  \sigma       +   ( \sup_{\mathbb{S}^{1}} h)   \right]   \le c_1 e^{(\frac12+\sigma){ R_{\rho,\mathbf{a}}}} 
 \eea 
  for    $\beta=\frac13,\frac12$. 
%{\color{red}by choosing   sufficiently large $\rho_2$.}   
By a similar argument as in the proof of Theorem \ref{thm-existencerevised11}, we have  
\bea \label{eq-4156} \Vert u_{k+1}\Vert_{C^{2,\beta,-\frac12 }_{{ R_{\rho,\mathbf{a}}}}}&\le C_0  \Vert E({\tilde w_{\mathbf{a}}}+\sum_{i=0}^k u_i)-E({\tilde w_{\mathbf{a}}}+\sum_{i=0}^{k-1} u_i)  \Vert _{C^{0,\beta,-\frac12 }_{ R_{\rho,\mathbf{a}}}}\\
&\le  { C_0 }C_1 \left[   e^{( \sgamma -\sigma) { R_{\rho,\mathbf{a}}}}  \Vert {\tilde w_{\mathbf{a}}}\Vert_{C^{2,\beta,\sgamma  }_{R_{\rho,\mathbf{a}}} }  +\sum_{i=0}^k \Vert u_i\Vert_{C^{2,\beta,\sigma  }_{{  R_{\rho,\mathbf{a}}}} }  + e^{-4\alpha { R_{\rho,\mathbf{a}}}}\right ]\Vert u_k \Vert_{C^{2,\beta,-\frac12 }_{{ R_{\rho,\mathbf{a}}}}} \\
%&\le  {\color{blue}C_0 }C_1 \left[  e^{( \sgamma -\sigma) {\hat R_{\mathbf{a}}}} \cdot  e^{(  \sigma-\sgamma) {\tilde  R_{\mathbf{a}}}}    + e^{(-\frac12  -\sigma ){\hat R_{\mathbf{a}}}} \sum_{i=0}^k \Vert u_i\Vert_{C^{2,\beta,-\frac12   }_{R} }  + e^{-4\alpha {\hat R_{\mathbf{a}}}}\right ]\Vert u_k \Vert_{C^{2,\beta,-\frac12}_{{\hat R_{\mathbf{a}}}}} \\
&\le  { C_0 }C_1 \left[   2  c_1 K+ 2 c_1  + e^{-4\alpha  \rho_1}\right ]\Vert u_k \Vert_{C^{2,\beta,-\frac12}_{{ R_{\rho,\mathbf{a}}}}} <2^{-1}\Vert u_k \Vert_{C^{2,\beta,-\frac12}_{{ R_{\rho,\mathbf{a}}}}}     
  \eea    for    $\beta=\frac13,\frac12$. 
 %{\color{red}for   a   sufficiently large $\rho_2$.}   
 This implies \eqref{eq-inductionstrong1}   and hence      $\{u_k\}_{k=0}^\infty$ geometrically converges  in $ C^{2,\beta, -\frac12  }_{ R_{\rho,\mathbf{a}}}     $ (for $\beta=\frac13,\frac12$), and     letting  $w_+:= \sum _{i=0}^\infty u_i$,  \eqref{eq-est-u0-4.81} yields that   for    $\beta=\frac13,\frac12$, 
\be \label{eq-ui-1/21}
\Vert w_{+} \Vert _{C^{2,\beta,-\frac12}_{ R_{\rho,\mathbf{a}}}} \le 2\Vert  u_0 \Vert _{C^{2,\beta,-\frac12}_{ R_{\rho,\mathbf{a}}}}\le 2  c_1 e^{(\sigma+\frac12){{ R_{\rho, \mathbf{a}}}}} .
 \ee

 Note that  $\tilde w _{\mathbf{a}}+w_+$ solves  the translator equation on $ ( R_{\rho,\mathbf{a}},\infty)\times \mathbb{S}^1$. It remains  to check  the boundary conditions  in \eqref{eq-u_0cases1}.  %In view of   \eqref{eq-conditioneee}  and  the zero boundary condition  for the operator $H_{0, R_{\rho,\mathbf{a}}} $,   
Since    \eqref{eq-4156}, \eqref{eq-conditioneee} and \eqref{eq-est-u0-4.81}     imply
  \bea \Vert u_{1}\Vert_{C^{2,\beta,-\frac12 }_{{ R_{\rho,\mathbf{a}}}}} 
&\le    { \frac{ \sigma\eae}{4( {\inf_{\mathbb{S}^{1}} h  } + \sup_{\mathbb{S}^{1}} h )} } \Vert u_0 \Vert_{C^{2,\beta,-\frac12 }_{{ R_{\rho,\mathbf{a}}}}} 
%&\le  C \left( e^{( \gamma -\sigma) { (\rho_1-\rho_2)}  }+  \varsigma    + e^{-4\alpha \rho_1}\right )\Vert u_0 \Vert_{C^{2,\beta,-\frac12}_{{\hat R_{\mathbf{a}}}}} \\
%&\le {\color{blue}\frac{\eae \sigma }{2({\color{green} (\inf_{\mathbb{S}^{1}} h(\cdot)) }+ \|h\|_\infty )} }    C_h c_2  e^{(\frac12 +\sigma ){\hat R_{\mathbf{a}}}}  
<   \frac{ 1}{4} \eae    \sigma c_2   e^{(\frac12 +\sigma ){ R_{\rho,\mathbf{a}}}},
\eea it follows that    
  \bea \label{eq-ui-1/211}
\Vert \sum _{i=1}^\infty u_i \Vert _{C^{2,\beta,-\frac12}_{ R_{\rho,\mathbf{a}}}}\le  { 2 }\Vert  u_1 \Vert _{C^{2,\beta,-\frac12}_{ R_{\rho,\mathbf{a}}}}%&\le    C \left(    e^{(  \sigma-\gamma) {\tilde  R_{\mathbf{a}}}}     +1\right ) e^{-\eae{\hat R_{\mathbf{a}}} }\Vert u_0 \Vert_{C^{2,\beta,-\frac12}_{{\hat R_{\mathbf{a}}}}}  \\
&\le   \frac{  1}{2} \eae    \sigma c_2  e^{(\frac12 +\sigma ){ R_{\rho,\mathbf{a}}}} 
\eea    for    $\beta=\frac13,\frac12$. 
By  the zero boundary condition  for the operator $H_{0, R_{\rho,\mathbf{a}}} $, we have 
\be \label{eq-u-u_0cases1}  \ba  & \sum_{i=1}^\infty u_i =0,  && 
 &\bigg \vert\frac{\p}{\p s}  \sum_{i=1}^\infty u_i  \bigg \vert < \frac {1} 2 \eae    \sigma c_2e^{  \sigma    R_{\rho,\mathbf{a}}} &&\text{ on $s=  R_{\rho,\mathbf{a}}$ }  .
 \ea   \ee  
 Therefore, the boundary conditions in  \eqref{eq-exist-ex-bd1'1}   follow from \eqref{eq-u_0cases1} and \eqref{eq-u-u_0cases1}. This finishes the proof.
\end{proof}

\begin{remark}
In Proposition \ref{prop-exist-exterior'1}, the weight $-\frac{1}{2}$ for the function space is chosen arbitrarily to satisfy $ \beta^-_0=2\alpha-1 < -\frac{1}{2} < \beta^+_0=-2\alpha $.
\end{remark}
% {\color{blue}
%\begin{remark}\label{rem-two-exp2}
%As in Remark \ref{rem-two-exp},  we can      obtain    uniform estimates  
%     for the exterior  barriers    with respect to  two distinct H\"older exponents.  Indeed,  for given exponents $0<\beta_1<\beta_2<1$,   the estimates  appearing in  Proposition \ref{prop-exist-exterior'1}   and its proof hold true with  uniform constants, which are  independent of $\beta=\beta_i$ ($i=1,2$).\end{remark}} 

 \medskip

Next, we construct global barriers that are asymptotic to an exterior translator $\tilde w_{\mathbf{a}}$. 

   \begin{prop}[Global barriers]\label{prop-global-barriers}

 % Then there is a constant $\rho_3 >  \rho_2$ (depending on $\alpha$, $h$, $\beta$) such that for each $R\ge \rho_2$,  there exists   a solution $w_1 \in C^{2,\beta, -\frac12  }_{\hat R_{\mathbf{a}}}$     to the equation
%$$$$
%Let $\sigma^{-1} {e^{\sigma s}} h(\theta) + g(s,\theta)$  be   a   	translator,  where   
%$g \in C_{\rho_1}^{2,\beta,\gamma_1} $ is  a solution to $\hat Lg=E(g)$ on $B_{\rho_1}^c$ for some  $\rho_1>1$, $0<\beta<1$ and    $0<\gamma_1 <\sigma$.         
%    $$$$    
    There is a   constant  $\rho_3>\rho_2$ such that
      for each   exterior translator $\tilde w_{\mathbf{a}} $  given in Theorem \ref{thm-existencerevised11}, we have  continuous global barriers $S_\pm(l,\theta)$ on  $[l_\pm,\infty)\times \mathbb{S}^1$     for some $l_\pm < \ln R_{\rho_3,\mathbf{a}}$ with   $R_{\rho_3,\mathbf{a}}=\rho_3 +   \ln (\vert \mathbf{b}(\mathbf{a})\vert+1) $, which satisfy the following properties: 
      
\begin{enumerate}[(a)] 
 \item  On  $ [{R_{\rho_3,\mathbf{a}}},\infty)\times \mathbb{S}^1$,
 \be
 S_{\pm}(e^s,\theta)=   {\sigma}^{-1}   e^{\sigma s} h(\theta)+ 
 \tilde w_{\mathbf{a}}(s,\theta)  +  w_\pm (s,\theta) 
 \ee
       with 
  $\Vert w_\pm \Vert _{C^{2,\beta,-\frac12}_{ R_{\rho_3,\mathbf{a}}}}  \le   2c_1 e^{(\sigma+\frac{1}{2}){{ R_{\rho_3,\mathbf{a}}}}} $  ($\beta=\frac13,\frac12$). Here, $w_{\pm}$ are given by Proposition \ref{prop-exist-exterior'1}.
  
  \item      $S_\pm(l,\theta)$ are the support functions \eqref{eq-supportfunction} of graphs of  continuous  entire functions $u_\pm(x)$ on $\mathbb{R}^2$.  
 \item      $u_+$ (resp. $u_-$) is   a smooth   supersolution (resp. subsolution) to the translator equation \eqref{eq-translatorgraph}  away from  the origin and the level curve $\{u_+=e^{R_{\rho_3,\mathbf{a}}}\}$ (resp. $\{u_-=e^{R_{\rho_3,\mathbf{a}}}\}$). 
 
\item $u_+$ (resp. $u_-$) is a viscosity  supersolution (resp. subsolution)  to \eqref{eq-translatorgraph} on $\mathbb{R}^2$.
   
   \item $u_-\le u_+$ on $\mathbb{R}^2$. 
   
\end{enumerate}
\end{prop}

 \begin{proof}
\noindent{\bf Step 1.} 
In order to construct     global barriers, we first consider a barrier on an inner region.  For a constant $M>0$,
let $f_M =f_M (l)$ be the radial solution to an $M$-modified translator equation  
\bea \label{eq-m-trans}
S+S_{\theta\theta}+ (M+S_{l}^{-2})^{\frac{1}{2\alpha}-2} S_l^{-4} S_{ll} =0
\qquad \hbox{ for $l>0$}
\eea  
with the conditions $f_M(0)=0$ and $f_M'(0)=\infty$. 
Note that $f_1$ is the radial translator. 
%{\color{blue} Note that $ f_{M+\delta}\geq f_M$ for $\delta>0$.} 
By direct computations,   there holds an asymptotic expansion
\be\label{eq-est-aymp-f1}
f_M(l)= {c_\alpha} \sigma^{-1} l^{\sigma} +   \frac{   c_\alpha^3 M}{2(1-4\alpha)}l^{\sigma-4\alpha}
 (1+o(1)) \qquad \hbox{as $ l\to\infty$}
\ee
with a constant $c_\alpha:=(2\alpha \sigma )^\alpha$.   
Now let us consider  a barrier of the form 
\be 
S(l,\theta)=Cf_M(l)h(\theta),
\ee where constants $C$ and $M$ will be determined later. 
By plugging this into   \eqref{eq-translatorsupport},    we have
 \bea &S+S_{\theta\theta}+ (1+S_{l}^{-2})^{\frac{1}{2\alpha}-2} S_l^{-4} S_{ll}\\
&= C  f_M \cdot 2\alpha \sigma  h^{1-\frac1\alpha} + (1+ (Cf_M'h)^{-2})^{\frac1{2\alpha}-2}(Cf_M'h)^{-4} (Cf_M''h)\\
&=  - C   (2\alpha\sigma ) h^{1-\frac1\alpha}  (M+f_M'^{-2})^{\frac1{2\alpha}-2}f_M'^{-4}f_M''  + C^{1-\frac1\alpha} h^{1-\frac1\alpha} \{(Ch)^2 +f_M'^{-2}\}^{\frac1{2\alpha}-2}f_M'^{-4}f_M''  \\
&=- Ch^{1-\frac1\alpha}f_M'^{-4}f_M'' \cdot \left[2\alpha\sigma(M+f_M'^{-2})^{\frac1{2\alpha}-2} - C^{-\frac1\alpha}\{(Ch)^2+f_M'^{-2}\}^{\frac1{2\alpha}-2}\right].
 \eea  
Letting $C=c_\alpha^{-1}$,  a barrier $ c_\alpha^{-1} f_M(l)h(\theta)$ solves
\be\left\{
\begin{aligned}
&S+S_{\theta\theta}+ (1+S_{l}^{-2})^{\frac{1}{2\alpha}-2} S_l^{-4} S_{ll} \ge 0 \qquad \text { if } \quad M\ge c_\alpha^{-2} \sup_{\mathbb{S}^1} \,  h^2 ,\\
&  S+S_{\theta\theta}+ (1+S_{l}^{-2})^{\frac{1}{2\alpha}-2} S_l^{-4} S_{ll} \le 0 \qquad \text { if } \quad M\le  c_\alpha^{-2} \inf_{\mathbb{S}^1} \,  h^2,
\end{aligned}\right.
\ee
since $f_M'>0$, $f_M''<0$, and $\frac{1}{2\alpha}-2>0$. 
\smallskip

Now we choose positive constants $M_+$ and $M_-$ such that  
%\be
%M_+=2  c_\alpha^{-2} \sup_{\mathbb{S}^1} \,  h^2 \qquad \hbox{and}\qquad M_-= 2^{-1}   c_\alpha^{-2} \inf_{\mathbb{S}^1} \,  h^2,
%\ee
\be
M_+\ge  c_\alpha^{-2} \sup_{\mathbb{S}^1} \,  h^2 \qquad \hbox{and}\qquad M_-\le   c_\alpha^{-2} \inf_{\mathbb{S}^1} \,  h^2,
\ee
and define
\be \label{eq-defUpm}
U_\pm(l,\theta)= c_\alpha^{-1} f_{M_\pm}(l)h(\theta).
\ee
Then,    $U_+$  is   a subsolution  and   $U_-$ is   a supersolution  
to  \eqref{eq-translatorsupport}  on each region where $U_i$ is smooth. Later, we will replace by a larger $M_+$ and a smaller $M_-$ so that barriers are respectively  a viscosity subsolution and supersolution of   \eqref{eq-translatorgraph}  at the origin. 
Using \eqref{eq-est-aymp-f1}, it follows that 
\be
\begin{aligned} \label{eq-est-aymp-U/h}
 \frac{U_\pm(l,\theta)}{h(\theta)}&=  \sigma^{-1}l^{\sigma}  \left\{1+  \frac{\sigma  c_\alpha^2 M_\pm   }{2(1-4\alpha)} l^{-4\alpha}   (1+o(1))\right\} \qquad \hbox{as $l\to\infty$}.
 \end{aligned}
\ee
Letting $g_{\pm}$ be the inverse of $c_\alpha^{-1} f_{M_\pm}=  {U_\pm}/{h} $,  we have 
\be\label{eq-est-aymp-g}
g_\pm(t)=\left( \sigma t\right)^{\frac{1}{ \sigma}} -\frac{  c_\alpha^2 M_\pm }{2(1-4\alpha)}    \left(\sigma  t\right)^{\frac{1 }{ \sigma} -\frac{ 4\alpha}{ \sigma}} (1+o(1))
\qquad \hbox{as $ t\to\infty$}.
 \ee There also hold corresponding asymptotic expansions for derivatives of $g_{\pm}(t)$ (in $t$) and $U_{\pm}/h$ (in $l$).  
Using this, we  obtain
\bea \label{eq-est-asymp-dg11} 
\left.\frac{\p}{\p l}\right\vert _{U_\pm/h = L}  ({U_\pm }/{h })=   \frac{1}{g_\pm ' (L) }=(  \sigma L)^{ 1-\frac{1}{\sigma}} + \frac{  c_\alpha^2 M_\pm}{2} ( \sigma L)^{1-\frac{1}{\sigma}-\frac{4\alpha}{ \sigma}} (1+o(1)) \quad\hbox{as $ L\to\infty$.}
\eea 
    In view of \eqref{eq-est-asymp-dg11},   
 there is a sufficiently large $L_0=L_0(\alpha, h , M_\pm )$ such  that  for all $L>L_0$,
\bea \label{eq-est-asymp-dg2} 
   (  \sigma L)^{ 1-\frac{1}{\sigma}} <\left.\frac{\p}{\p l}\right\vert _{U_-/h = L}  ({U_- }/{h })<\left.\frac{\p}{\p l}\right\vert _{U_+/h = L}  ({U_+ }/{h }) < (1+ \eae \sigma c_2  )   (  \sigma L)^{ 1-\frac{1}{\sigma}}   .\eea
Here, $\eae$ and $c_2$  are given in \eqref{eq-conditioneee}-\eqref{eq-conditionee3}.

     \medskip
     
  Next, we construct a barrier on an outer region. 
In light of Proposition \ref{prop-exist-exterior'1},   for     $\rho_3> \rho_2$  to be chosen later,   consider  \be \label{eq-Vpm} V_{\pm} (s,\theta)= \sigma^{-1}e^{\sigma s} h(\theta) + \tilde w_{\mathbf{a}} (s,\theta) +  w_\pm(s,\theta)\ee   which solve  \eqref{eq-translatorsupport} on 
$  ( R_{\rho_3, \mathbf{a}},\infty )\times \mathbb{S}^1  $.    
% with $  R_{\rho_3 ,\mathbf{a}}=\rho_3 +   \ln (\vert \mathbf{b}(\mathbf{a})\vert+1) $.
Then,   $V_\pm(s,\theta)$  satisfy   that 
% $ V_j(R, \theta)/h  $ ($j=1,2$) are positive  constants on $\p B_R$, 
\be\label{eq-exist-ex-bd1-glo1}
%\left\{\ba
    (V_\pm/h)    \Big\vert _{s=R_{\rho_3,\mathbf{a}}} =    (   {\sigma}^{-1}   \mp c_2)  e^{\sigma R_{\rho_3,\mathbf{a}}}, \qquad  
 \pm   \frac{\p}{\p s}\Big\vert _{s=R_{\rho_3,\mathbf{a}}} (V_\pm/h) > \pm  \left[1\pm   \left(1-  \eae\right)\sigma c_2  \right] e^{\sigma R_{\rho_3,\mathbf{a}}} .
\ee
%and 
%\be\label{eq-exist-ex-bd2-glo1}
%%\left\{\ba
%    (V_2/h)   \left. \right\vert _{s=R_{\mathbf{a}}} =    (   {\sigma}^{-1}   + c_1)  e^{\sigma R_{\mathbf{a}}}, \qquad  
%   \left.\frac{\p}{\p s}\right\vert _{s=R_{\mathbf{a}}} (V_2/h) <   \left[1-   \left(1-  \eae\right)\sigma c_1  \right] e^{\sigma R_{\mathbf{a}}} .
%\ee
Let  $ L_\pm  :=   (V_\pm/h)   \big\vert _{s=R_{\rho_3,\mathbf{a}}} $.   Then  by using \eqref{eq-exist-ex-bd1-glo1} and the choice of   $\eae$ and $c_2$  in \eqref{eq-conditioneee}-\eqref{eq-conditionee3},   we have 
 \bea \label{eq-est-asymp-dV11}
\left.\frac{\p}{\p l}\right\vert _{V_+/h = L_+} (V_+/h)=  e^{-R_{\rho_3,\mathbf{a}}}\cdot \left.\frac{\p}{\p s}\right\vert _{s=R_{\rho_3,\mathbf{a}}} (V_+/h)&>  \left[1+ (1-\eae){\sigma c_2}\right]\left(1-\sigma c_2\right)^{\frac1\sigma -1} \left ( \sigma L_+\right)^{ 1-\frac{1}{\sigma}}  \\
%&=\left(1+ (1-\eae) {\sigma c_2}{ }\right)\left(1-\sigma c_2\right)^{\frac{2\alpha}\sigma  } \left ( \sigma L_+\right)^{ 1-\frac{1}{\sigma}}  \\
%&> \left(1+ (1-2\eae  -  \frac{2\alpha}\sigma) \sigma c_2\right)  \left ( \sigma L_1\right)^{ 1-\frac{1}{\sigma}}  \\
&>\left[1+  \left(   2- {\sigma^{-1} } -3\eae\right) \sigma c_2\right]  \left ( \sigma L_+\right)^{ 1-\frac{1}{\sigma}}  \\
&{=\left(1+\eae\sigma c_2  \right) } \left ( \sigma L_+\right)^{ 1-\frac{1}{\sigma}} 
,  \eea
and similarly
\bea  \label{eq-est-asymp-dV21} 
\left. \frac{\p}{\p l}\right  \vert_{V_-/h = L_-} ({V_-}/{h}) <\left(1 -  \eae \sigma c_2\right)  \left ( \sigma L_-\right)^{ 1-\frac{1}{\sigma}} .
\eea  
  
   \medskip

 %Now  let us fix   $\gamma_2$ such that  $\max(\gamma_1, \sigma-4\alpha)< \gamma_2 <\sigma$. Note that $R(\geq \rho_2)$ has not been determined so far. 
\noindent{\bf Step 2.}   Global barriers  $S_\pm(l,\theta)$  will be constructed by    gluing  the exterior barrier  $V_\pm( \ln l, \theta)$ and a translated $U_\pm$, say $U_\pm(l-l_\pm,\theta)$ for some $l_\pm \in \mathbb{R}$, along $l = e^{R_{\rho_3,{\mathbf{a}} }}$.   Note that $\rho_3$ has not yet  been determined. 
%For    a large constant   $R ( \geq \rho_2)$   to be determined later, % and $V_1(s,\theta)$  be the     the corresponding  solution on $ B_R^c$ given by Proposition \ref{prop-exist-exterior'}.  
  % Then
 Let us   define
\bea S_\pm(l,\theta) := 
\begin{cases}
 \ba & V_\pm(\ln l,\theta) && \text{ for  } \,\, l \ge e^{R_{\rho_3,{\mathbf{a}} }}, \\
 &U_\pm(l-l_\pm,\theta) && \text{ for  }\,\, l_\pm\le l\le e^{R_{\rho_3,{\mathbf{a}} }},
 \ea
  \end{cases} 
  \eea 
  where  each $l_\pm$ is the unique number such that    $S_\pm$  is continuous along  $l= e^{R_{\rho_3,{\mathbf{a}} }}$. %(For large $R$, it is clear that such $l_1$ uniquely exists). 
%  Similarly, we define
%   \bea
%    S_1(l,\theta) := 
%    \begin{cases} 
%    \ba & V_2(\ln l,\theta) && \text{ for }\,\, l \ge e^R, \\
% &U_2(l-l_2,\theta) && \text{ for }\,\, l_2\le l\le e^R ,
% \ea
%  \end{cases}
%   \eea 
%  with     the unique number    $l_2$   to make  $S_2$        continuous along  $l= e^R$. 
Selecting   $\rho_3=\rho_3(\alpha, h , M_\pm ) >\rho_2$   large   so that  
 \be
   (\sigma^{-1} -c_2)e^{\sigma \rho_3} >L_0 
   ,
   \ee
   and utilizing      \eqref{eq-est-asymp-dg2}, \eqref{eq-exist-ex-bd1-glo1}, \eqref{eq-est-asymp-dV11} and \eqref{eq-est-asymp-dV21}, it follows that 
 \be \label{eq-glob-barrier-angle-11}
 \left\{
 \ba
 &\lim_{\eps \to 0^-} \p_l S_+(e^{R_{\rho_3,{\mathbf{a}} }}+\eps ,\theta) <\lim_{\eps \to 0^+} \p_l S_+(e^{R_{\rho_3,{\mathbf{a}} }}+\eps ,\theta)\quad \text{ for all }\theta \in \mathbb{S}^1 ,  \\
 %and   
% \bea \label{eq-glob-barrier-angle-2}
  &\lim_{\eps \to 0^-} \p_l S_-(e^{R_{\rho_3,{\mathbf{a}} }}+\eps ,\theta) > \lim_{\eps \to 0^+} \p_l S_-(e^{R_{\rho_3,{\mathbf{a}} }}+\eps,\theta)\quad \text{ for all }\theta \in \mathbb{S}^1.
   \ea\right.
     \ee
     
 \medskip

Now we will    check that $\p_{ll}S_\pm<0$ on $   \{l>l_\pm : l  \not= e^{R_{\rho_3, {\mathbf{a}} }}\}\times \mathbb{S}^1$ 
to   obtain an   entire function  $u_\pm(x)$  corresponding   to   $S_\pm$ as in \eqref{eq-supportfunction}, respectively. 
Recalling   $V_{\pm}$ in \eqref{eq-Vpm}, we see that \be
l^2\, \p_{ll } S_\pm= \p_{ss} V_\pm-\p_s V_\pm=(\sigma-1)e^{\sigma s } h(\theta)+ (\p_{ss}\tilde w_{\mathbf{a}} -\p_s \tilde w_{\mathbf{a}} ) + (\p_{ss}w_\pm-\p_s w_\pm) 
\ee
for $l >  e^{R_{\rho_3,{\mathbf{a}} }}$. 
 Since $ \Vert \tilde w_{\mathbf{a}} \Vert_{C^{2,\beta,  \sgamma  }_{ R_{\rho_1,\mathbf{a}}} } \le  2c_1K e^{(\sigma -  \sgamma  ) R_{\rho_1, \mathbf{a}}}  $,   $\Vert w_\pm \Vert _{C^{2,\beta,-\frac12}_{ R_{\rho_2, \mathbf{a}}}}  \le 2  c_1 e^{(\sigma+\frac12){{ R_{\rho_2,\mathbf{a}}}}} $, and $ R_{\rho_1,\mathbf{a}} -   R_{\rho_3,\mathbf{a}} =\rho_1-\rho_3$, it follows that 
 \be 
l^2\, \p_{ll } S_\pm= \p_{ss} V_\pm-\p_s V_\pm \leq -2\alpha e^{\sigma s } h(\theta)+2 c_1K  e^{(\sigma -  \sgamma  )( \rho_1-\rho_3)}   e^{\sigma  s } +    2  c_1 e^{(\sigma+\frac12) (\rho_1-\rho_3)}  e^{\sigma s} 
\ee
for   $s>{R_{\rho_3,{\mathbf{a}} }}$.  
%Here a constant $C>0$ may depend  on  $\Vert g \Vert _{C^{2,\beta,\gamma_1}_{\rho_1}}$, $\gamma_1$ and  $\gamma_2$. 
Choosing $\rho_3$ sufficiently large, we deduce  that   
\be
   \p_{ll } S_{\pm} <0\quad\hbox{and}\quad \p_{l } S_{\pm}> 0  \quad \hbox{on $   \{l >  e^{R_{\rho_3,{\mathbf{a}} }}\}\times \mathbb{S}^1$}.
\ee
Using properties of $f_{M_\pm}$ solving \eqref{eq-m-trans},  we deduce  that
  $\p_{ll}S_\pm<0$ and  $\p_{l } S_{\pm}> 0$  on $   \{l>l_\pm : l  \not= e^{R_{\rho_3, {\mathbf{a}} }}\}\times \mathbb{S}^1$.

   %   Similarly,    an entire function  $u_-(x)$  on $\mathbb{R}^2$ whose support function is    $S_-$,    a smooth    subsolution to \eqref{eq-translatorgraph}  away from  the origin and the level curve $\{x\in \mathbb{R}^2:\, u_-(x)=e^{R_{{\mathbf{a}} }}\}$. 
   
   \medskip

\noindent{\bf Step 3.}   In light of    Proposition \ref{prop:graphicality_S}, 
   there exists an   entire function  $u_+(x)$  on $\mathbb{R}^2$ whose support function is    $S_+$, and $u_+$  is  a smooth   supersolution  to \eqref{eq-translatorgraph}  away from  the origin and the level curve $\{x\in \mathbb{R}^2:\, u_+(x)= e^{R_{\rho_3, {\mathbf{a}} }}\}$. Similarly, we obtain an entire function $u_-$ corresponding to $S_-$. 
   \smallskip

We will show  that $u_+$     (resp.  $u_-$) is  a viscosity supersolution   (resp.  subsolution). %  and $u_-\leq u_+$  on $\mathbb{R}^2$.  
  %We first prove that  $u_+$  is  a viscosity supersolution   on $\mathbb{R}^2$. 
Observe that  $u_+$ is   a smooth   supersolution away from  the origin and the level set $\{u_+= e^{R_{\rho_3,{\mathbf{a}} }}\}$ and no smooth function can touch from above at a point on $\{u_+= e^{R_{\rho_3,{\mathbf{a}} }}\}$ by \eqref{eq-glob-barrier-angle-11}. So it remains to   prove that   $u_+$ is a viscosity supersolution to \eqref{eq-translatorgraph}  at the origin.  Let $v$ be a smooth strictly    convex function which touches $u_+ $   at $x=0$  from below.  Note that    $v(0)=l_+$ and $\nabla v(0)=0$.  
Let   $A(l)$ be the area enclosed by the level curve $\{v(x)=l\}$ for $l>l_+$.   Then we have 
\be \label{eq-est-area-A}
   A(l)= \frac{{2}\pi ( l-l_+) }{\sqrt{\det D^2v(0) +o(1)} }\qquad\hbox{as  $l\to l_+$.}
\ee  
Let $A_+(l)$ be the area enclosed by the level curve $\{u_+(x)=l\}$ for $l>l_+$, and let  $w$  be the  radial solution to $ \det D^2 w = ( M_++|Dw|^2)^{2-\frac1{2\alpha}}$  with $w(0)=l_+$, whose support function at level $l$  is $ f_{M_+}(l-l_+) $. From the definition of $S_+$ (or see \eqref{eq-defUpm}),   there is a constant $c=c({\alpha,h})>0$ such that 
\be 
A_+(l)=      \frac{   2 \pi c( l-l_+)}{\sqrt{\det D^2w(0) +o(1)} }=  \frac{ 2 \pi   c  ( l-l_+)}{\sqrt{M_+^{2-\frac{1}{2\alpha}} +o(1)} } \qquad\hbox{as  $l\to l_+$}
.\ee
%  where  $w$  is a  solution to $ \det D^2 w = ( M_1+|Dw|^2)^{2-\frac1{2\alpha}}$  with $w(0)=l_1$, and its   support function is $ f_{M_1}(l) $
 The touching implies an obvious inequality $ A_+(l) \le A(l) $, which further yields  that 
 \be\det D^2v(0) \leq C    M_+^{2-\frac{1}{2\alpha}}
 \ee
 with some  constant $C=C({\alpha,h})>0$.  
  By selecting $M_+$ large, we obtain  that  $
\det D^2v(0) <1.$ Combining above, $u_+$ is a viscosity supersolution  to \eqref{eq-translatorgraph}    on  $\mathbb{R}^2$. Similarly, by selecting $M_->0$ small, we obtain that  $u_-$ is a viscosity subsolution to \eqref{eq-translatorgraph} on $\mathbb{R}^2$.   
\medskip

  Finally, we prove  that $u_- \le u_+$ on $\mathbb{R}^2$.   Proposition \ref{prop-exist-exterior'1} implies  that for each   constant $\delta>0$,
\bea\label{eq-est-asymp-infty1}
S_+( l-\delta,\theta )&=  V_\pm( \ln(l-\delta),\theta )=  \sigma^{-1}(l-\delta )^{\sigma} h(\theta) +\tilde w _{ {{\mathbf{a}} }}(\ln(l-\delta),\theta) +O(l^{-\frac{1}{2}}),  \\
&= \sigma^{-1}(l-\delta)^{\sigma} h(\theta)   +\tilde w_{{{\mathbf{a}} }}(\ln l,\theta) +O(l^{ \bar \gamma-1}),\\
S_-( l,\theta )&= V_-( \ln l,\theta )= \sigma^{-1}l^{\sigma} h(\theta)  + \tilde w_{{{\mathbf{a}} }}(\ln l,\theta) +O(l^{-\frac{1}{2}})\qquad \hbox{as $l\to\infty$. }
\eea
Note that $   \sgamma -1, -\frac{1}{2}  <  \sigma-1=-2\alpha$.    Let $u_{+,\delta}$ be the entire graph obtained by $S_+( l-\delta,\theta ).$
For  each   $\delta>0$,  \eqref{eq-est-asymp-infty1}  implies that $ u_-\leq u_{+,\delta}$  on the exterior of a large ball, and then by the comparison principle, it follows that $ u_-\leq u_{+,\delta}$ on $\mathbb{R}^2$. Letting $\delta\to 0^+$,   we conclude that  $ u_-\leq u_+$    on $\mathbb{R}^2$. 
 \end{proof}

\medskip

{In order to prove the continuity of translators with respect to the parameter in Theorem \ref{thm-existence}, we first show   continuity of barriers on an exterior region.   }

  \begin{proposition}
 \label{thm-modulicont-barr} 
 Let  $\mathbf{a}_i\in\mathbb{R}^K$ for  $i=1,2$,  and    
   let  $w_{\mathbf{a}_i}^{\pm}$      be the  barriers given by  Proposition \ref{prop-global-barriers}  satisfying the relation \be
   S_{\mathbf{a}_i}^{\pm}(e^s,\theta)=    {\sigma}^{-1}   e^{\sigma s} h(\theta)+w_{\mathbf{a}_i}^{\pm}(s,\theta)\qquad \hbox{for  $i=1,2$}.
   \ee
 If  $\vert  \mathbf{a}_i \vert \le M$  ($i=1,2$)   for some $M>0$,  then there holds
\bea  \label{eq-modulicont-barr} 
 \Vert w_{\mathbf{a}_1}^{\pm}-w_{\mathbf{a}_2}^\pm \Vert _{C^{2,\frac13,\sgamma }_{R}}  \le  C_*  \vert \mathbf{a}_1 -\mathbf{a}_2\vert  , 
\eea 
where  $ {R}:=\displaystyle \max_{i=1,2} R_{\rho_3,  {\mathbf{a}}_i} $ and $C_*<\infty $ is a constant depending on $M$,  $\alpha$, and   $h$.
  \end{proposition}

\begin{proof}%[Proof of Proposition \ref{thm-modulicont-barr}]

\noindent{\bf Step 1.}   Recall that  for  $\mathbf{a}\in \mathbb{R}^K$,  {$\tilde w_{\mathbf{a} } \in C^{2,\frac12,\sgamma}_{R_{\rho_1,\mathbf{a}}}$} is   the exterior translator  constructed in Theorem \ref{thm-existencerevised11}. 
In this step, we show
 \be  \label{eq-wdiff1}
  \Vert \tilde w _{\mathbf{a}_1}  -  \tilde w _{\mathbf{a}_2} \Vert_{C^{2,\frac13,\sgamma }_ {R_*}} \le   C_* |\mathbf{a}_1-\mathbf{a}_2|   
 \ee 
 % for    $\mathbf{a}_i\in \mathbb{R}^K$  ($i=1,2$)  satisfying $\vert  \mathbf{a}_i \vert \le M$  with some $M>0$,  
  {with ${R}_*:= \displaystyle\max_{i=1,2} R_{\rho_1,  {\mathbf{a}}_i} $}   for some positive constant  $C_* $  depending  on $M$,  $\alpha$, and   $h$. 
Indeed, we will prove that 
 \be  \label{eq-wdiff0}
 \Vert \tilde w_{\mathbf{a}_1' }- \tilde w_{\mathbf{a}_2' }  \Vert _{C^{2,\frac13, \sgamma }_{ { \tilde {R}}}}  \le  C   \left[ \Vert \tilde w _{\mathbf{a}_1}  -  \tilde w _{\mathbf{a}_2} \Vert_{C^{2,\frac13,\sgamma }_ {\tilde R}}+{C_*}e^{(\beta^+_I -\sgamma  ) {\tilde R}}(|\mathbf{a}_1-\mathbf{a}_2|+ |d_1-d_2 | ) \right] 
 \ee  
%  \be  \label{eq-wdiff0}
% \Vert \tilde w_{\mathbf{a}_1+ d _1 \mathbf{e}_{I+1}}- \tilde w_{\mathbf{a}_2+ d_2 \mathbf{e}_{I+1}}  \Vert _{C^{2,\beta, \sgamma }_{ {R_*}}}  \le  C e^{( \sgamma- \beta^+_I+\delta  ) {R}^*} ( \Vert \tilde w _{\mathbf{a}_1}  -  \tilde w _{\mathbf{a}_2} \Vert_{C^{2,\beta,\sgamma }_ {R_*}}+Le^{\delta {R_*}}(|\mathbf{a}_1-\mathbf{a}_2|+ |d_1-d_2 |   ) , 
% \ee  
 for parameters of form  
 ${\mathbf{a}}_i':= \mathbf{a}_i+d_i\mathbf{e}_{I+1}$  ($i=1,2$), with 
 \bea\label{eq-ai-ci}
 \mathbf{a}_i = (\mathbf{0}_{I+1},\mathbf{c}_i) \in \mathbb{R}^K,\quad \mathbf{c}_i \in \mathbb{R}^{K-I-1}, \quad 0\le I \le K -1 ,  \quad d_i\in \mathbb{R} \quad  (i=1,2),  
 \eea
 {where
 %$\tilde R$  is    given  by  %chosen  to satisfy   a  quantitative bound
 \be\label{eq-R*} 
 \tilde {R}:= \max_{i=1,2} R_{\rho_1,  {\mathbf{a}}_i'}    % \rho_1 +\ln \big( \max_{i=1,2}   \vert  \mathbf{b}({\mathbf{a}_i+d_i \mathbf{e}_{I+1}}) \vert  +1\big )
 \ee 
 with $\rho_1$ appearing in Theorem \ref{thm-existencerevised11}.  
 Here and below, $C$   depends on $\alpha$ and  $h$,  and      ${C_*}$ may additionally depend on $ \max_{i=1,2}   \vert  {{\mathbf{a}}_i'} \vert $. } 
Then  \eqref{eq-wdiff1}   follows by     applying  \eqref{eq-wdiff0} at most $K$ times. 
 \medskip
 
In order to show   \eqref{eq-wdiff0}, we recall from  the proof of Theorem \ref{thm-existencerevised11} that  
\be
\tilde w_{ {\mathbf{a}}_i'} = \tilde w_{\mathbf{a}_i}+\sum_{j=0}^\infty \tilde u_{i,j} ,
\ee where 
    $\tilde u_{i,j}$$\in   C^{2,{\frac12} ,\sgamma}_{ \tilde  R_ i}$ ($i=1,2,$  $j=0,1,2,\ldots$) are   defined by \eqref{eq-deftildeu1} as follows: 
 \be\label{eq-deftildeu_uij}
  \begin{cases} \ba 
 &\tilde u_{i,0}:= d_i e^{\beta^+_Is} \varphi_I,\\  
&\tilde u_{i,1}:=  \tilde H_i\Big(E(\tilde w_{\mathbf{a}_i}+\tilde u_{i,0}) -E (\tilde w_{\mathbf{a}_i})\Big), \\
&\tilde u_{i,k+1}:=  \tilde H_i \Big(E(\tilde w_{\mathbf{a}_i}+ \sum _{j=0}^k \tilde u_{i,j})-E(\tilde w_{\mathbf{a}_i}+ \sum _{j=0}^{k-1} \tilde u_{i,j})\Big)\quad \hbox{for $k\ge1$ }.
\ea \end{cases}
\ee 
 Here and below, we   simply denote 
  \be \label{eq-def-tildeR}
  \tilde R_i: =  R_{\rho_1, {\mathbf{a}}_i' },  \quad  \quad \tilde H_i:= H_{I, \tilde R_i}\qquad  \hbox{for $i=1,2$},
  \ee
  and    
 {
 $\Vert \cdot \Vert_\gamma:= \Vert\cdot \Vert _{C^{2,\frac13,\gamma }_{ \tilde R}}$ for  $\gamma\in \mathbb{R}$. 
 }
 
% for a given $\gamma\in \mathbb{R}$,  $\Vert\cdot \Vert _{C^{2,\beta,\gamma }_{ {R}_*}} $ is      simply denoted by  $\Vert \cdot \Vert_\gamma  $.

\medskip
Firstly, we prove that % if $\rho_4$ is sufficiently large, then there holds  
 an recursive inequality: for all $k\ge0$, 
  \bea  \label{eq-225}
  \big\Vert \sum _{j=1}^{k+1} (\tilde u_{1,j}- \tilde u_{2,j}) \big\Vert_{\beta^+_I-\delta  } 
 & \le \frac12 \big\Vert \sum_{j=1}^k( \tilde u_{1,j}-  \tilde u_{2,j} )\big\Vert_{\beta^+_I-\delta}   \\
 & \quad + Ce^{\delta   \tilde R}\Big [ |d_2| e^{(\sgamma -\sigma) {\tilde R}}  \big\Vert \tilde w_{\mathbf{a}_1} - \tilde w_{\mathbf{a}_2}\big\Vert_{\sgamma}   + C_*( |\mathbf{a}_1-\mathbf{a}_2|+ |d_1-d_2 |  )\Big ], 
 \eea
where  $C$ and $C_*$ {are constants with previously stated dependencies, and  $\delta>0$ is  the constant defined in  \eqref{eq-deltagap}.} 
Here and below, %we simply denote that  for a given $\gamma\in \mathbb{R}$, 
% {\color{red}\be \Vert \cdot \Vert_\gamma:= \Vert\cdot \Vert _{C^{2,\frac13,\gamma }_{ \tilde R}},\ee} and 
 %$\Vert\cdot \Vert _{C^{2,\beta,\gamma }_{ {R}_*}} $ is    denoted simply by  $\Vert \cdot \Vert_\gamma  $,  and 
we interpret $\sum_{j=1}^0 (\tilde u_{1,j}-  \tilde u_{2,j})$ as zero.
If the estimate \eqref{eq-225} holds true for all $k\geq0$, then  using  $\bar \gamma=\sigma-\delta$  and the choice of $\tilde R$ in \eqref{eq-R*}  
implies 
\be \label{eq-476} 
\Vert  \sum_{j=1}^\infty  (\tilde u_{1,j}- \tilde  u_{2,j}) \Vert_{\sgamma  }   \le C \Big[\Vert \tilde w_{\mathbf{a}_1}- \tilde w_{\mathbf{a}_2} \Vert_{\sgamma  } +C_* e^{(\beta^+_I -\sgamma ) {\tilde R}} ( |\mathbf{a}_1-\mathbf{a}_2|+ |d_1-d_2 |)    \Big].
\ee 
Here, we used
$
|d_2| e^{(\beta^+_I -\sigma ) {\tilde R}}\le e^{(\beta^+_I -\sigma)\rho_1} \le 1 
$. This proves \eqref{eq-wdiff0}. 
 %Note that \eqref{eq-476} implies the estimate \eqref{eq-thm24} for  $\tilde w_{\mathbf{a}_i+ d_i \mathbf{e}_{I+1}}$  replacing  $  w_{\mathbf{a}_i+ d_i \mathbf{e}_{I+1}}$ . 
%{\color{red}Then we obtain 
%\be  \label{eq-wdiff0}
% \Vert \tilde w_{\mathbf{a}_1+ d _1 \mathbf{e}_{I+1}}- \tilde w_{\mathbf{a}_2+ d_2 \mathbf{e}_{I+1}}  \Vert _{C^{2,\beta, \sgamma }_{ {R_*}}}  \le  C \left( \Vert \tilde w _{\mathbf{a}_1}  -  \tilde w _{\mathbf{a}_2} \Vert_{C^{2,\beta,\sgamma }_ {R_*}}+Le^{(\beta^+_I -\sgamma  ) {R_*}}(|\mathbf{a}_1-\mathbf{a}_2|+ |d_1-d_2 |  \right),
% \ee 
% which implies 
% \eqref{eq-wdiff1} 
% by  repeating this process several times. } 
%\medskip

 \medskip
Now it suffices to prove \eqref{eq-225}. 
We remind two important facts: first,   the corresponding  linear operators $\tilde H_i$'s for $\tilde u_{i,j}$ ($j=0,1,\ldots$) are different associated with  the zero Dirichlet boundary condition   imposed on different radii $\tilde R_i$'s.  %We will simply denote the operators $\tilde H_i:= H_{I,\tilde R_{\mathbf{a}_i+ d_i \mathbf{e}_{I+1} }}$ and radii $\tilde R_i: = \tilde R_{\mathbf{a}_i+ d_i \mathbf{e}_{I+1} } $ for  $\tilde u_{i,k}$. 
 Second, the operators $\tilde H_i$ are defined on the spaces $C^{0,\beta, \beta^+_I-\delta}_{\tilde R_i}$  ($\beta=\frac13,\frac12$) for $i=1,2$.  
%{\color{red}BC: this needs a change. $C^0$에서 $C^2$로 가는 함수이고 holder exponent가 1/3인걸로 봐야할수도있어서. (물론 같은함수지만)  }
 
 \smallskip
 
 Without loss of generality, we  assume that  $|\mathbf{b}({\mathbf{a}}_2' )|\ge| \mathbf{b}({\mathbf{a}}_1'  )| $ and thus $\tilde R_2=\tilde R$. From  linearity of $\tilde H_i$,  it holds that for   $k\ge0$, 
\bea 
\Vert \sum _{j=1}^{k+1} (\tilde u_{1,j}- \tilde u_{2,j}) \Vert_{\beta^+_I -\delta}%{{\color{red}C^{2,\frac13,\beta^+_I -\delta}_{\tilde R_2} }  }
&\le   \big\Vert \tilde  H_2 ( E(\tilde  w_{\mathbf{a}_1}+ \sum_{j=0}^k \tilde u_{1,j} )-E(\tilde w_{\mathbf{a}_1}) -E( \tilde w_{\mathbf{a}_2}+ \sum_{j=0}^k \tilde u_{2,j} )+   E(\tilde w_{\mathbf{a}_2})   )\big\Vert_{\beta^+_I -\delta}%{{\color{red}C^{2,\frac13,\beta^+_I -\delta}_{\tilde R_2} }} 
\\
&\quad + \big\Vert (\tilde H_1-\tilde H_2) ( E( \tilde w_{\mathbf{a}_1}+ \sum_{j=0}^k \tilde u_{1,j} )-E(\tilde w_{\mathbf{a}_1}))\big\Vert_{\beta^+_I -\delta} %{{\color{red}C^{2,\frac13,\beta^+_I -\delta}_{\tilde R_2} }}\
\\
&=: \text{\bf I}+ \text{\bf II}.
\eea 
We first  estimate the term $\text{\bf I}$.  Let us decompose the term in the argument of $\tilde H_2$ as 
\bea\label{eq-454} & \Big [  E(\tilde  w_{\mathbf{a}_1}+ \sum_{j=0}^k (\tilde u_{1,j}-\tilde u_{2,j} )) -E(\tilde w_{\mathbf{a}_1})  \Big] \\
&+    \Big \{ [E( \tilde w_{\mathbf{a}_1}+ \sum_{j=0}^k \tilde u_{1,j} )-   E(\tilde  w_{\mathbf{a}_1}+ \sum_{j=0}^k (\tilde u_{1,j}-\tilde u_{2,j} )) ]  -[E( \tilde w_{\mathbf{a}_2}+ \sum_{j=0}^k \tilde u_{2,j} )-   E(\tilde w_{\mathbf{a}_2})]\Big \}   \eea
so to apply \eqref{eq-est-error-diff1}  and \eqref{eq-est-error-diff2} in Lemma \ref{lem-error-estimate}, respectively.  Estimating  the     terms in \eqref{eq-454},  
%\bea  &  e^{ (\sgamma -\sigma) {\tilde R_2}} \Vert\tilde w_{\mathbf{a}_1} \Vert_{\sgamma}%_{{\color{red}C^{2,\frac13,\sgamma}_{\tilde R_2} }  }
%+e^{ (\beta^+_I -\sigma) { \tilde R_2}} \sum_{i=1}^2\sum_{j=0}^\infty   \Vert\tilde u_{i,j}   \Vert_{\beta^+_I}%_{{\color{red}C^{2,\frac13,\beta^+_I}_{\tilde R_2} }   }
%+e^{-4\alpha  {\tilde R_2}}   \\
% &\leq 2c_1K + 2(|d_1|+|d_2|) { C_2}e^{ (\beta^+_I -\sigma) {\tilde R_2}} +e^{-4\alpha  {\tilde R_2}}  \\&
% \leq   2c_1K +   { \color{red}2C_2}e ^{(\beta^+_I-\sigma )\rho_1}+e^{-4\alpha  {\rho_1}},
%\eea
% where we used \eqref{eq-448bound}, $ \Vert\tilde u _{i,j}\Vert_{C^{2,\frac13,\beta^+_I}_{\tilde R_i}}\le 2^{-j}\Vert \tilde u_{i,0} \Vert_{C^{2,\frac13,\beta^+_I}_{\tilde R_i}}   $, and $\Vert \tilde u_{i,0} \Vert_{C^{2,\frac13,\beta^+_I}_{\tilde R_i}} \le C_2|d_i| $.
 note  from  
 \eqref{eq-448bound} and \eqref{eq-conv-u_i}  that for $i=1,2$, 
 \bea\label{eq-err-conv-u_i}  &  e^{ (\sgamma -\sigma) {\tilde R_i}} \Vert\tilde w_{\mathbf{a}_i} \Vert_{\sgamma}\leq 2c_1K ,\qquad e^{ (\beta^+_I -\sigma) { \tilde R_i}}  \sum_{j=0}^\infty   \Vert\tilde u_{i,j}   \Vert_{\beta^+_I}\leq e^{ (\beta^+_I -\sigma) {\tilde R_i}} \cdot 2|d_i| { C_2}\leq 2c_1e ^{(\beta^+_I-\sigma )\rho_1} .
%+e^{ (\beta^+_I -\sigma) { \tilde R_i}}  \sum_{j=0}^\infty   \Vert\tilde u_{i,j}   \Vert_{\beta^+_I}%_{{\color{red}C^{2,\frac13,\beta^+_I}_{\tilde R_2} }   }+e^{-4\alpha  {\tilde R_i}}   \\
%  &\leq 2c_1K +  e^{ (\beta^+_I -\sigma) {\tilde R_i}} \cdot 2|d_i| { C_2}+e^{-4\alpha  {\tilde R_i}}  \\
% &\leq   2c_1K +   { \color{red}2C_2}e ^{(\beta^+_I-\sigma )\rho_1}+e^{-4\alpha  {\rho_1}},
%  &\leq   2c_1K + 2c_1e ^{(\beta^+_I-\sigma )\rho_1}+e^{-4\alpha  {\rho_1}}. 
\eea
 Using \eqref{eq-c-prefix}-\eqref{eq-R0-prefix},   Corollary \ref{cor-ex-linear}    and  Lemma \ref{lem-error-estimate} imply    
\bea  
\text{\bf I}&\le C_0 C_1  \big( 2 c_1 K+ 5c_1 \big)\cdot e^{ \delta   {  \tilde R_2}}  \big\Vert \sum_{j=0}^k (\tilde u_{1,j}-  \tilde u_{2,j})  \big\Vert_{\beta^+_I}%_{\color{red}C^{2,\frac13,\beta^+_I}_{\tilde R_2} } 
\\ 
 &\quad + C_0 C_1 e^{( \delta  +\sgamma -\sigma  ) {{  \tilde R_2}}}  \big\Vert \sum_{j=0}^k \tilde u_{2,j}  \big\Vert _{\beta^+_I} %_{{\color{red}C^{2,\frac13,\beta^+_I}_{\tilde R_2} }  }
 \,\, \big\Vert \tilde w_{\mathbf{a}_1}- \tilde w_{\mathbf{a}_2} +\sum_{j=0}^k( \tilde u_{1,j}-  \tilde u_{2,j})  \big\Vert_{\sgamma}%_{{\color{red}C^{2,\frac13, \sgamma}_{\tilde R_2} } }
 \\
&\le  \frac12 \Big(e^{\delta  {{ \tilde R_2}}}|d_1-d_2|+   \big\Vert \sum_{j=1}^k( \tilde u_{1,j}-  \tilde u_{2,j})  \big\Vert_{\beta^+_I -\delta}%_{ {\color{red}C^{2,\frac13,  \beta^+_I -\delta}_{\tilde R_2} }   }
\Big) +C |d_2|e^{( \delta  +\sgamma -\sigma  ) {  \tilde R_2}}  \big\Vert \tilde w_{\mathbf{a}_1}- \tilde w_{\mathbf{a}_2}  \big\Vert_{\sgamma}.%_{{\color{red}C^{2,\frac13, \sgamma}_{\tilde R_2} }   }
 \eea
%for some    constant $C=C(\alpha,h)$. %  {\color{blue}  with some $0<\delta<\delta_0$ to be determined later. }
 Thus it holds that 
\bea\label{eq-2299'} 
\text{\bf I}&%\leq  C \Big \Vert E(\tilde  w_{\mathbf{a}_1}+ \sum_{j=0}^k \tilde u_{1,j} )-E(\tilde w_{\mathbf{a}_1}) -E( \tilde w_{\mathbf{a}_2}+ \sum_{j=0}^k \tilde u_{2,j} )+   E(\tilde w_{\mathbf{a}_2})\Big \Vert_{ \beta^+_I -\delta  }\\&      
\le\frac12  \big\Vert \sum_{j=1}^k \tilde u_{1,j}-  \tilde u_{2,j}   \big\Vert_{\beta^+_I -\delta}% _{ {\color{red}C^{2,\frac13,  \beta^+_I -\delta}_{\tilde R_2} }   }  
+ C   e^{\delta  {\tilde R_2}}\left[ |d_2| e^{(\sgamma -\sigma) {\tilde R_2} }  \Vert \tilde w_{\mathbf{a}_1}-\tilde w_{\mathbf{a}_2}\Vert_{\sgamma}% _{{\color{red}C^{2,\frac13, \sgamma}_{\tilde R_2} } }  
+|d_1-d_2|\right] .
 \eea 

%where, in the last inequality, we have chosen $\rho_3$ sufficiently large so that terms in the brakets $[ \cdots ]$ are small.  

\medskip

Next, let us estimate the second term $\text{\bf II}$ which involves $\tilde H_1 -\tilde H_2$. 
Recalling from \eqref{eq-wj}-\eqref{eq-wj2},   $w= (\tilde H_1 -\tilde H_2 )g$ is given by $w(s,\theta)=\sum_{j=0}^\infty  w_j(s) \varphi_j(\theta)$, where 
\begin{align} 
  w_j (s)&= -e^{\beta^-_j s} \int_{\tilde R_1}^{\tilde R_2} e^{(\beta^+_j-\beta^-_j)r }\int_r^\infty e^{-\beta^+_j t}g_j(t) dt dr \quad \text{ for  }j \ge I,  \label{eq-wj''}\\
  w_j (s)&= e^{\beta^-_j s} \int_{\tilde R_1}^{\tilde R_2} e^{(\beta^+_j-\beta^-_j)r }\int_{\tilde R_1} ^r e^{-\beta^+_j t}g_j(t) dt dr \label{eq-wj1''}\\ 
&\quad + e^{\beta^-_j s} \int_{\tilde R_2}^{s} e^{(\beta^+_j-\beta^-_j)r }\int_{\tilde R_1} ^{\tilde R_2} e^{-\beta^+_j t}g_j(t) dt dr \quad  \text{ for }0\le j <I \label{eq-wj2''}.
\end{align}
We aim to show, for $g=E( \tilde w_{\mathbf{a}_1}+ \sum_{j=0}^k \tilde u_{1,j} )-E(\tilde w_{\mathbf{a}_1})$,
\bea \label{eq-est-diff-II}
%\Vert (\tilde H_1 -\tilde H_2 )g \Vert_{C^{2,\beta,\beta^+_I-\delta}_{ {R_*}}} =
\mathbf{II}=\Vert  w \Vert_{C^{2,{\frac13},\beta^+_I-\delta}_{ \tilde R_2}}  &\le C  ( \tilde R_2 - \tilde R_1)^\vartheta \Vert g\Vert _{C^{0,{\frac{1}{2}},\beta^+_I-\delta }_{ \tilde R_1}}  
\eea 
{for some uniform constants $C >0$ and $0< \vartheta<1$.}  
%Here $0<\beta_0<\frac{\beta_0+1}{2}<1$ and $\theta=\theta(\beta_0)>0$.
The estimate follows by a similar argument to the proof of  Lemma \ref{lem-ex-linear}. A technical difficulty lies in the fact that one can not apply the elliptic regularity theory on $s\ge \tilde R_2$ upto the boundary $s=\tilde R_2$.   We bypass this by introducing a stronger holder exponent $1/2$ in  \eqref{eq-est-diff-II}.

Let us momentarily assume \eqref{eq-est-diff-II}. 
Observe  that 
\bea
\big\Vert E( \tilde w_{\mathbf{a}_1}+ \sum_{j=0}^k \tilde u_{1,j} )-E(\tilde w_{\mathbf{a}_1})  \big\Vert _{C^{0,{\frac1{2}},\beta^+_I-\delta  }_{  \tilde{R}_1}}  &\le   C e^{\delta   \tilde{R}_1}\big\Vert \sum_{j=0}^k \tilde u_{1,j} \big\Vert_{C^{2,{\frac{ 1}{2}},\beta^+_I}_{ \tilde{R}_1}} 
%&\le   C e^{\delta  \tilde{R}_1}\Vert \tilde u_{1,0} \Vert_{C^{2,{\frac{ 1}{2}},\beta^+_I}_{\color{red} \tilde{R}_1}} 
  \le C  e^{\delta_1  \tilde{R}_1 } |d_1|  
\eea  by   Lemma \ref{lem-error-estimate} and a similar argument as for   \eqref{eq-err-conv-u_i}. 
%and \eqref{eq-def-tildeR}, 
Then  we use     the estimate  \eqref{eq-est-diff-II}   to obtain   %\text{II}$ is bounded by 
\bea \label{eq-4781}
\text{\bf II} %& \leq  C|d_1|   e^{\delta  \tilde{R}_1} (\tilde R_2-\tilde R_1)\\
&\leq  C|d_1|   e^{\delta  \tilde{R}_1} \Big[\ln (|\mathbf{b}(\mathbf{a}_2' )|+1)-\ln ( |\mathbf{b}(\mathbf{a}'_1 )|+1) \Big]^\vartheta  \le C_* e^{\delta \tilde{R}_1 }  \left|\mathbf{a}_2' -  \mathbf{a}_1'   \right|, 
\eea %{\color{blue}where $C_*$ also depends on  $ \max_{i=1,2}   \vert  {\mathbf{a}_i+d_i \mathbf{e}_{I+1}} \vert $. }
and then the estimates   \eqref{eq-2299'}   and \eqref{eq-4781} imply \eqref{eq-225}.

 \medskip

To finish \textbf{Step 1}, it remains to prove \eqref{eq-est-diff-II}.  First, there holds the $L^2$-estimate  
 \bea\label{eq-L-2-linear-sol1}
  e^{-(\beta^+_I-\delta ) s}\Vert w(s,\cdot)\Vert_{L^2_h(\mathbb{S}^1)} %=  e^{-(\beta^+_I-\delta ) s}  \Big(\sum_{j=0}^\infty {w_j^2(s)} \Big)^{1/2}  
 \le C( \tilde R_2 - \tilde R_1) \Vert g \Vert_{C^{0,0,\beta^+_I-\delta }_{ \tilde R_1} }\qquad\hbox{ for all $s\geq  
 \tilde R_2$}, 
  \eea  
which is an analogous estimate corresponding Claim \ref{cla-L-infty-ext} in Lemma \ref{lem-ex-linear}.
Indeed, for \eqref{eq-L-2-linear-sol1}, we follow   the proof of Claim \ref{cla-L-infty-ext} after replacing  $\gamma$ and $\delta$ in the claim by $\beta_I^+-\delta$ and $\delta_1=\delta/2$, respectively. Then    we have  that for $j\geq I$, 
\be
\ba
 e^{-2(\beta^+_I-\delta) s} {w_j^2(s)}  &\le  C e^{ 2(\beta^-_j -\beta^+_I+\delta) s}  \int_{\tilde R_1} ^{\tilde R_2}  e^{2(\beta^+_I-\delta- \beta^-_j -\delta_1)  r}dr \int_{\tilde R_1}^{\tilde R_2} e^{4\delta_1 r} \int_r^\infty e^{-2(\beta^+_I-\delta+\delta_1) t}g^2_j(t) dtdr . 
%& \le  C   (\tilde R_2 -\tilde R_1) e^{ 2(\beta^-_j-\beta^+_I+\delta)( s-\tilde R_2) } e^{-2\delta_1 \tilde R_2}\int_{\tilde R_1}^{\tilde R_2}e^{4\delta_1 r} \int_r^\infty e^{-2(\beta^+_I-\delta+\delta_1) t}g^2_j(t) dtdr , 
\ea\ee
The choice $\delta_1= \delta/2$ implies $2(\beta^+_I -\delta -\beta^-_j -\delta_1)>0$. Then    for $j\geq I$, 
\be
\ba
 e^{-2(\beta^+_I-\delta) s} {w_j^2(s)}  %&\le  C e^{ 2(\beta^-_j -\beta^+_I+\delta) s}  \int_{\tilde R_1} ^{\tilde R_2}  e^{2(\beta^+_I-\delta- \beta^-_j -\delta_1)  r}dr \int_{\tilde R_1}^{\tilde R_2} e^{4\delta_1 r} \int_r^\infty e^{-2(\beta^+_I-\delta+\delta_1) t}g^2_j(t) dtdr \\
& \le  C   (\tilde R_2 -\tilde R_1) e^{ 2(\beta^-_j-\beta^+_I+\delta)( s-\tilde R_2) } e^{-2\delta_1 \tilde R_2}\int_{\tilde R_1}^{\tilde R_2}e^{4\delta_1 r} \int_r^\infty e^{-2(\beta^+_I-\delta+\delta_1) t}g^2_j(t) dtdr , 
\ea\ee
and by adding above, it follows from $\beta^-_j -\beta^+_I+\delta<0$  that for $s\ge \tilde R_2$,  
\bea \label{eq-est-L2-jgeqI} 
e^{-2(\beta^+_I-\delta) s}\sum _{j\ge I}    w_j^2(s) % \le C(\tilde R_2-\tilde R_1)^2  \,e^{2(\beta^-_j-\beta^+_I+\delta)(s-\tilde R_2) }\Vert g \Vert ^2 _{C^{0,0,\beta^+_I-\delta }_{\tilde R_1}}
\le C(\tilde R_2-\tilde R_1)^2\sup_{s\ge \tilde R_1} \Vert   e^{-(\beta^+_I-\delta )s} g(s)  \Vert_{L^2_h(\mathbb{S}^1)}^2. 
\eea 
%{\color{red}where  $C$ depend on $ \alpha, $ $h $, and $ \delta$. }  
One can similarly argue for $0\le j< I$  in  \eqref{eq-wj1''}-\eqref{eq-wj2''} and obtain  that for $s\ge \tilde R_2$,  
\be
 e^{-2(\beta^+_I-\delta) s}\sum _{0\le j< I}   w_j^2(s) \le C (\tilde R_2-\tilde R_1)^2\sup_{s\ge \tilde R_1} \Vert   e^{-(\beta^+_I-\delta )s} g(s)  \Vert_{L^2_h(\mathbb{S}^1)}^2.
\ee
Adding two estimates, \bea  \label{eq-4761} \sup_{s\ge \tilde R_2} \Vert e^{-(\beta^+_I-\delta )s} w(s)\Vert _{L^2_h(\mathbb{S}^1)}   \le C(\tilde R_2-\tilde R_1)  \sup_{s\ge \tilde R_1} \Vert   e^{-(\beta^+_I-\delta )s} g(s)  \Vert_{L^2_h(\mathbb{S}^1)},  \eea 
which  implies  \eqref{eq-L-2-linear-sol1}.

Let us proceed to \eqref{eq-est-diff-II}.  
On each subdomain $ [s-1,s+1]\times \mathbb{S}^1$,  from the Sobolev inequality and the interpolation inequalities in Sobolev spaces and H\"older spaces, there holds that 
\be \label{est-GN}
\Vert w \Vert _{C^{2,{\frac13}}_{s,1}}\le C \Vert w \Vert_{L^2_{s,1}}^{\vartheta} \Vert w \Vert ^{1-\vartheta} _{C^{2,{\frac{ 1}{2}}}_{s,1}}
\ee
{for some uniform constants $C >0$ and $0< \vartheta<1$.}     Note that, by Corollary \ref{cor-ex-linear} applied to $\tilde H_ig$  ($i=1,2$),  
\bea \label{eq-4762} \Vert w \Vert _{C^{2,\frac12,\beta^+_I-\delta }_{\tilde R_2 }} \le	 C \Vert g \Vert  _{C^{0,\frac12,\beta^+_I-\delta }_{\tilde R_1 }} \eea
  for some constant  $C=C(\alpha,h)$. 
 Thus, by \eqref{eq-L-2-linear-sol1} and \eqref{eq-4762},  we deduce that
 % with choice $\beta= \frac{\beta_0+1}{2}$, 
%\be
 %\Vert w \Vert _{C^{2,{\frac13},\beta^+_I-\delta }_{ \tilde R_2}} \le C (\tilde R_2 -\tilde R_1)^{\theta } \Vert g \Vert_{C^{0,{\frac{1}{2}},\beta^+_I -\delta}_{\tilde R_1} }, \ee
%which  proves 
\eqref{eq-est-diff-II} holds true. This finishes {\bf Step 1}.

%\smallskip Finally, to estimate II, recalling
%\bea\Vert E( \tilde w_{\mathbf{a}_1}+ \sum_{j=0}^k \tilde u_{1,j} )-E(\tilde w_{\mathbf{a}_1})  \Vert _{C^{0,{\frac1{2}},\beta^+_I-\delta  }_{\color{red} \tilde{R}_1}} \le   C e^{\delta   \tilde{R}_1}\Vert \sum_{j=0}^k \tilde u_{1,j} \Vert_{C^{2,{\frac{ 1}{2}},\beta^+_I}_{ \tilde{R}_1}} \le   C e^{\delta  \tilde{R}_1}\Vert \tilde u_{1,0} \Vert_{C^{2,{\frac{ 1}{2}},\beta^+_I}_{\color{red} \tilde{R}_1}} , % \le C  e^{\delta_1  {R}^*} |d_1| \eea   and \eqref{eq-def-tildeR}, 
%we use the estimate  \eqref{eq-est-diff-II} to conclude that %\text{II}$ is bounded by 
%\bea \label{eq-4781}
%\text{II} %& \leq  C|d_1|   e^{\delta  \tilde{R}_1} (\tilde R_2-\tilde R_1)\\
%&\leq  C|d_1|   e^{\delta  \tilde{R}_1} \Big[\ln (|\mathbf{b}(\mathbf{a}_2+ d_2 \mathbf{e}_{I+1})|+1)-\ln ( |\mathbf{b}(\mathbf{a}_1+ d_1 \mathbf{e}_{I+1})|+1) \Big]^\theta  \\  
%&\le C_* e^{\delta \tilde{R}_1 }  \left|(\mathbf{a}_2+ d_2 \mathbf{e}_{I+1})-\left(  \mathbf{a}_1+d_1 \mathbf{e}_{I+1}\right)\right|.
%\eea {\color{blue}Here $C_*$ also depends on  $ \max_{i=1,2}   \vert  {\mathbf{a}_i+d_i \mathbf{e}_{I+1}} \vert $.} 

%Adding the estimate for I (with choice $\beta =\beta_0$) in \eqref{eq-2299'}  and the estimate for II in \eqref{eq-4781}, the estimate \eqref{eq-225} and this  finishes Step 1.  

\bigskip

\noindent{\bf Step 2.}
The  estimate \eqref{eq-modulicont-barr}   for barriers $w_{\mathbf{a}_i }^{\pm}$ follows by an adopting similar method of {\bf Step 1} to the construction of exterior barriers. 
 %The idea of  the proof is   similar to Step 1.
 Let us  simply denote 
  {\be 
  R_i: =  R_{\rho_3, \mathbf{a}_i},  \quad  \quad H_i:= H_{0, R_i}\qquad  \hbox{for $i=1,2$}.
  \ee  
  and let us assume $R_2\ge R_1$.} Recall from  the proofs of   Proposition \ref{prop-global-barriers} and Proposition \ref{prop-exist-exterior'1}, the barriers $w_{\mathbf{a}_i } ^\pm$ are defined as follows:  
\be
w_{\mathbf{a}_i } ^\pm= \tilde w_{\mathbf{a}_i } + \sum_{j=0}^\infty u_{i,j}^{\pm}, %\quad \hbox{on $s\ge {R_i}$},%+{\color{red} v_i}  
\ee
where    $ u_{i,j}^{\pm}$'s in $C^{2,\frac12,-\frac12}_{R_i}$ ($i=1,2,$  $j=0,1,2,\ldots$) are defined by 
\be \begin{cases}
 \ba & u_{i,0}^{\pm}:= \hat  g_i (s,\theta) \mp c_2e^{\sigma  R_i} e^{\beta^-_0 (s-  R_i)}h(\theta)  ,\\  
& u_{i,1}^{\pm}:= H_i\Big(E(\tilde w_{\mathbf{a}_i}+u_{i,0}^{\pm}) -E (\tilde w_{\mathbf{a}_i})\Big), \\
& u_{i,k+1}^{\pm}:= H_i \Big(E(\tilde w_{\mathbf{a}_i}+ \sum _{j=0}^k u_{i,j}^{\pm})-E(\tilde w_{\mathbf{a}_i}+ \sum _{j=0}^{k-1} u_{i,j}^{\pm})\Big)\quad\hbox{for $k\ge1$ }.
\ea \end{cases}
\ee   Here, $ \hat  g_i$'s ($i=1,2$)  satisfy  \eqref{eq-hat-g-def} with $\tilde w_{\mathbf{a}_i }$    having   the expression \eqref{eq-hatgdef}. %for ${\color{red} v_i}    \in C^{2,\beta,-\frac12 }_{R_{\mathbf{a}}}$; see the proof of  Theorem \ref{thm-existencerevised11-n}.   
By arguing similarly as {\bf Step 1} and employing estimates in the proof of Proposition \ref{prop-exist-exterior'1}, we have 
 \be  \label{eq-est-uij} 
 \big\Vert  \sum_{j=1}^\infty  (  u_{1,j}^\pm-    u_{2,j}^\pm) \big\Vert_{C^{2,\frac13,\sgamma}_{R_2} } \le  C \Big[ \Vert  u_{1,0}^\pm-  u_{2,0}^\pm\Vert_{C^{2,\frac13,\sgamma}_{R_2} }  +    \big\Vert \tilde w_{\mathbf{a}_1}- \tilde w_{\mathbf{a}_2} \big\Vert_{C^{2,\frac13,\sgamma }_{R_2} } +  e^{(\sigma-\sgamma)R_1}(R_2-R_1)^\vartheta   \Big] 
\ee 
for some $0<\vartheta<1$. Here, we used \eqref{eq-inductionstrong1},    and \eqref{eq-est-diff-II} with $\beta^+_I-\delta $ replaced by $\sgamma$ for instance.   
Here and below, $C$   depends on $\alpha$ and  $h$,  and      $C_*$ may additionally depend on $ \max_{i=1,2}   \vert  {{\mathbf{a}}_i} \vert $.

Lastly, it remains to estimate $   u_{1,0}^\pm-  u_{2,0}^\pm $.   In light of \eqref{eq-hatgdef}, we have 
 \bea \label{eq-hatgdef-i}
 \tilde w_{\mathbf{a}_i} ( R_i ,\theta) =\displaystyle \sum_{j=0}^\infty d_{i,j}  \varphi_j(\theta) ,\qquad  \hat g_i(s,\theta) = - \sum _{j=0}^\infty d_{i,j} e^{\beta^-_j (s-R_i )} \varphi_j(\theta)\quad \text{ for $s\geq R_i $} 
 \eea 
with some  constants  $d_{i,j}\in\mathbb{R}$   ($i=1,2$, $j=0,1,2,\ldots$). To estimate $ \hat g_1-  \hat g_2$, consider 
\bea \label{eq-hatgdef-ex}
 \tilde w_{\mathbf{a}_1} ( R_2 ,\theta) =\displaystyle \sum_{j=0}^\infty d_{j}  \varphi_j(\theta) ,\qquad  \hat g(s,\theta) = - \sum _{j=0}^\infty d_{j} e^{\beta^-_j (s-R_2)} \varphi_j(\theta)\quad \text{ for $s\geq R_2$} 
 \eea    with some  constants  $d_j\in\mathbb{R}$ for $j\geq0$. 
Then we have 
\bea
e^{-\sgamma s}\Vert\hat g( s ,\cdot) -  \hat g_2( s ,\cdot)\Vert_{L^2_h(\mathbb{S}^1)}
%{\color{blue} \leq e^{-\sgamma s} \sum _{j=0}^\infty |d_{j} -d_{2,j}|^2}
%\leq  \Vert\tilde w_{\mathbf{a}_1} ( R_2 ,\cdot )-\tilde w_{\mathbf{a}_2} ( R_2 ,\cdot)\Vert_{L^2_h(\mathbb{S}^1)}^2
%\leq  \Vert\tilde w_{\mathbf{a}_1} ( R_2 ,\cdot )-\tilde w_{\mathbf{a}_2} ( R_2 ,\cdot)\Vert_{L^2_h(\mathbb{S}^1)}^2
\leq  C\Vert\tilde w_{\mathbf{a}_1} -\tilde w_{\mathbf{a}_2}  \Vert_{C^{0,\frac13,\sgamma}_{R_2}}
\qquad \hbox{for $s\geq R_2$}
\eea 
 and 
\bea
 \Vert\hat g_1( s ,\cdot) -  \hat g( s ,\cdot)\Vert_{L^2_h(\mathbb{S}^1)}^2
& \leq 2  \sum _{j=0}^\infty \Big[ |d_{1,j} -d_{j}|^2 + \big|d_{1,j} ( e^{\beta_j^-(R_2-R_1)}-1)\big|^2 \Big]e^{2\beta_j^-(s-R_2)}\\
 & \leq 2  \sum _{j=0}^\infty \Big[|d_{1,j} -d_{j}|^2 + \big| d_{ 1,j}  \beta_j^-   (R_2-R_1) \big|^2e^{2\beta_j^-(s-R_2)} \Big]
%\leq  \Vert\tilde w_{\mathbf{a}_1} ( R_2 ,\cdot )-\tilde w_{\mathbf{a}_2} ( R_2 ,\cdot)\Vert_{L^2_h(\mathbb{S}^1)}^2
%\leq  \Vert\tilde w_{\mathbf{a}_1} ( R_2 ,\cdot )-\tilde w_{\mathbf{a}_2} ( R_2 ,\cdot)\Vert_{L^2_h(\mathbb{S}^1)}^2\\
%\leq  \Vert\tilde w_{\mathbf{a}_1} -\tilde w_{\mathbf{a}_2}  \Vert_{C^{0,\frac13,\sgamma}_{R_2}}
\qquad \hbox{for $s\geq R_2$}.
\eea 
Since
\bea\sum_{j=0}^{\infty}|d_{1,j} -d_{j}|^2= \Vert \tilde w_{\mathbf{a}_1}(R_1,\cdot)-\tilde w_{\mathbf{a}_1}(R_2,\cdot)\Vert _{L^2 _h(\mathbb{S}^1)}^2, \quad\sum_{j=0}^\infty \big| d_{{1,}j}  \beta_j^-   \big|^2= \Vert \partial_s \tilde w _{\mathbf{a}_1}(R_1,\cdot)\Vert_{L^2_h(\mathbb{S}^1)}^2, \eea 
it follows that 
\bea
e^{-\sgamma s} \Vert\hat g_1( s ,\cdot) -  \hat g( s ,\cdot)\Vert_{L^2_h(\mathbb{S}^1)} 
 & \leq  C (R_2-R_1) \Vert \tilde w_{\mathbf{a}_1}\Vert_{C^{2,\frac13,\sgamma}_{R_1}}
 %\leq  Ce^{(\sigma-\sgamma) R_2}  |R_2-R_1|  %  +  (R_2-R_1)   e^{\beta_0^- R_2}\Vert\hat g \Vert_{C^{2,\frac13,\beta_0^-}_{R_2}}\Big)
%\leq  \Vert\tilde w_{\mathbf{a}_1} ( R_2 ,\cdot )-\tilde w_{\mathbf{a}_2} ( R_2 ,\cdot)\Vert_{L^2_h(\mathbb{S}^1)}^2
%\leq  \Vert\tilde w_{\mathbf{a}_1} ( R_2 ,\cdot )-\tilde w_{\mathbf{a}_2} ( R_2 ,\cdot)\Vert_{L^2_h(\mathbb{S}^1)}^2\\
%\leq  \Vert\tilde w_{\mathbf{a}_1} -\tilde w_{\mathbf{a}_2}  \Vert_{C^{0,\frac13,\sgamma}_{R_2}}
 \qquad \hbox{for $s\geq R_2$}.
\eea  
Using \eqref{est-GN} and \eqref{est-hat-g}, we obtain 
 \bea 
 %e^{\beta_0^- R_2}
\Vert\hat g_1-  \hat g_2\Vert_{C^{2,\frac13,\sgamma}_{R_2}}\leq      C_*\Big[ \Vert\tilde w_{\mathbf{a}_1} -\tilde w_{\mathbf{a}_2}  \Vert_{C^{2,\frac13,\sgamma}_{R_2}}^\vartheta +   (R_1-R_2)^\vartheta  \Big] ,
 \eea 
and  this together with \eqref{eq-wdiff1} implies that 
\bea  \Vert   u_{1,0}^{\pm} -  u_{2,0}^{\pm}   \Vert_{C^{2,\beta,\sgamma}_{R_2} }&\le  \Vert\hat g_1-  \hat g_2\Vert_{C^{2,\frac13,\sgamma}_{R_2}}+ C   e^{ (\sigma-\sgamma) R_2 } \big(1-e^{(\sigma-\beta_0^-) (R_1-R_2)}\big) \leq  C_*  |\mathbf{a}_1-\mathbf{a}_2|   .
 \eea
Therefore   we conclude  from \eqref{eq-est-uij} and {\bf Step 1}  that 
 \bea
\Vert w_{\mathbf{a}_1}^{\pm}-w_{\mathbf{a}_2}^\pm \Vert _{C^{2,\frac13,\sgamma }_{ {R_2}}}  
&\le     \big\Vert \tilde w_{\mathbf{a}_1}- \tilde w_{\mathbf{a}_2} \big\Vert_{C^{2,\frac13,\sgamma }_{R_2} } + \big\Vert  \sum_{j=0}^\infty  (  u_{1,j}^\pm-    u_{2,j}^\pm) \big\Vert_{C^{2,\frac13,\sgamma}_{R_2} }\le       C_*|{\mathbf{a}_1}-{\mathbf{a}_2 }|,   
\eea 
which completes the proof. 

\end{proof}

%%%%%%%%%%%%%%%%%%%%%%%%%%%%%%%%%%%%%%%%%%%%%%%%%%%%%%%%%%%%%%%%%%%
\section{Proof of Theorem \ref{thm-existence}} \label{sec-thm-existence}
%%%%%%%%%%%%%%%%%%%%%%%%%%%%%%%%%%%%%%%%%%%%%%%%%%%%%%%%%%%%%%%%%%%

\begin{proof}[Proof of Theorem  \ref{thm-existence}]

  {\bf Step 1.} 
  For a given each   exterior translator $\tilde w_{\mathbf{a}} $ for $\mathbf{a}\in \mathbb{R}^K$, let $S_+$ and $S_-$ be the global barriers obtained in Proposition \ref{prop-global-barriers}. 
Let us use their graphical representations $u_+$ and $u_-$ to construct an entire solution located between two. 
Note that $u_+$ is not globally convex due to its gluing. If we consider the convex envelop of $u_+$, say $\bar u_+$, then it is straightforward to check that $u_- \le \bar u_+ \le u_+ $ on $\mathbb{R}^2$ and $\bar u_+$ is a convex viscosity supersolution to \eqref{eq-translatorgraph}. \medskip

To construct   an entire solution to the translator equation  \eqref{eq-translatorgraph}, we shall use     its dual problem by  the  Legendre transformation.     For     any proper function $ \psi: \mathbb{R}^2\to \mathbb{R}\cup\{
+\infty\}$,    the Legendre dual  $ \psi^*: \mathbb{R}^2\to \mathbb{R}\cup\{
+\infty\}$  of $\psi$  is defined  by 
\bea   \psi^*(y) := \sup_{x\in \mathbb{R}^2} \,\left\{  \langle y,x\rangle  - \psi(x)\right\}. \eea 
Here, we refer to  \cite[Chapter 2]{MR1964483}  for preliminary properties of  the Legendre dual transformation.
\smallskip

Let  $v_+$ and $v_-$ be  the Legendre duals of $u_+$ and $u_-$, respectively.  Note that the legendre duals of $u_+$ and $\bar u_+$ are the same as $v_+$.
%Let us consider the Legendre duals of $u_-$ and $u_+$ and denote them by $v_-$ and $v_+$, respectively. 
Since   $u_\pm $  are asymptotic to the homogeneous function   by Proposition \ref{prop-global-barriers}, there are constants $C>0$ and $d>0$ such that 
 \bea\label{eq-uL1} C^{-1} |x|^{\frac{1}{ 1-2\alpha }} \le u_-(x) \le u_+(x) \le C |x|^{\frac{1}{ 1-2\alpha }} \quad \text{ on } |x|\ge d .  \eea  For the duals, we have that there are constants $C'>0$ and $d'>0$ such that 
\bea\label{eq-vL1} C'^{-1} |y|^{\frac{1}{2\alpha }} \le v_+(y) \le v_-(y) \le C' |y|^{\frac{1}{ 2\alpha }} \quad \text{ on } |y|\ge d' .  \eea 
In particular, the duals   $v_{\pm}$
 are entire functions. Moreover,  $v_-$ and $v_+$ are entire convex viscosity supersolution and subsolution to the dual equation $\det D^2 v= (1+|y|^2)^{\frac{1}{2\alpha}-2}$ in $\mathbb{R}^2$, respectively.  
 \smallskip
 %are entire functions. Moreover, note that the Legendre duals of $u_+$ and $\bar u_+$ are the same as $v_+$, and that $v_-$ and $v_+$ are entire convex viscosity supersolution and subsolution to the dual equation $\det D^2 v= (1+|y|^2)^{\frac{1}{2\alpha}-2}$, respectively. 

We first aim to construct an entire solution to the dual equation.    For each sufficiently large $L<\infty$,  we denote    $\Omega_L:=\{y\in \mathbb{R}^2\, : \, v_+(y) <L\}$,   the sub-level set of $v_+$, which is     is pre-compact and convex in $\mathbb{R}^2$. 
Then, consider the following  Dirichlet problem  
\be\label{eq-translatorgraph-mod-v}
\left\{\ba \det D^2 v &=  (1+|y|^2)^{\frac{1}{2\alpha}-2}\qquad &\hbox{in $\Omega_L$},\\
v&= L \qquad  & \hbox{on $\partial\Omega_L$}.
\ea\right.
\ee
This problem admits a unique convex solution $v_{L}$ which is smooth strictly convex on $\Omega_L$;  we may refer to {\cite[Sections 2 and 4]{MA-figalli-book}} for instance.  By   the comparison principle, it holds   that  
\bea \label{eq-uL2}
 v_+  \le v_L  \le v_- \quad  \hbox{on $\Omega_L$.}
  \eea 
Using the  interior regularity estimates  for the  Monge--Amp\`ere equation (see   \cite[Section 4]{MA-figalli-book}   for instance), we may take a subsequence $v_{L_i}$ so that $v_{L_i}$  converges on each compact subset of $\mathbb{R}^2$ as $L_i\to\infty$. Then we deduce that the limit $v_\infty$ is a smooth strictly convex solution to the dual equation  on $\mathbb{R}^2$, which  satisfies $ v_+\le v_{\infty} \le v_-$ on $\mathbb{R}^2$. 
\medskip

Now let    $u_\infty$  be  the Legendre dual of $v_\infty$. Then,   $u_{\infty}$ is an entire smooth strictly convex solution to the translator equation \eqref{eq-translatorgraph} with the following bounds:
%If we denote the Legendre dual of $v_\infty$ by $u_\infty$, then $u_{\infty}$ is an entire smooth strictly convex solution to the translator equation \eqref{eq-translatorgraph} and it has bounds
\be\label{eq-est-bd-u1_u2} 
 u_-  \le u_\infty  \le u_+ \quad  \hbox{on}\,\, \mathbb{R}^2. 
\ee  
Let us define the surface $\Sigma_{\mathbf{a}}$ as the graph of $u_\infty$. The bound \eqref{eq-est-bd-u1_u2} implies that $w_{\mathbf{a}}$ belongs to $ C^{0,0,\bar \gamma }_{ R_{\rho_3,\mathbf{a}}}$. 
 Therefore,  we have proved the existence of   a $K$-parameter family of  (complete) translators $\Sigma_{\mathbf{a}}$  in Theorem   \ref{thm-existence}. In view of the construction of translators (Proposition \ref{prop-global-barriers}) and the Schauder theory (Lemma \ref{lem-B1}), there holds the estimate \be \label{eq-barww0}
\Vert w_{\mathbf{a} }- \tilde w _{\mathbf{a}}\Vert _{C^{2,\beta,-1/2}_{ R_{\rho_3,\mathbf{a}}}}\le {\frac{1}{2}  e^{(\sigma +1/2)R_{\rho_3,\mathbf{a}}}}\qquad (\beta=\tfrac12,\tfrac13)
.\ee 
 Here, note that the constant $c_1$ may have been replaced by a smaller constant due to the constant $C$ in \eqref{eq-B3}.
The first two properties (i) and (ii) are immediate  from Theorem  \ref{thm-existencerevised11} and Proposition  \ref{prop-global-barriers}.

\bigskip

{\bf Step 2.}  
To prove (iii), we first  recall that 
$ e^{\beta^+_0s} \varphi_0(\theta)= m_0  e^{-2\alpha s} h(\theta)$, $ e^{\beta^+_1 s }\varphi_1( \theta)= m_1 \cos \theta$, and $ e^{\beta^+_2s}\varphi_2( \theta)= m_2 \sin  \theta$; see \eqref{eq-c0c1c2}.  These Jacobi fields account for  the translations along the   $x_3$, $x_1$, $x_2$-axes.

\smallskip

Let    $\mathbf{a}=(\mathbf{0}_{3},   \mathbf{a}_{K-3}) $ and  $\mathbf{a}'=(\mathbf{0}_{2}, a_2, \mathbf{a}_{K-3}) $ with $a_2\in\mathbb{R}$.  When constructing the exterior barrier $\tilde w_{\mathbf{a}'}$ from $\tilde w_{\mathbf{a}}$ in the proof of Theorem \ref{thm-existencerevised11}, 
we observe that 
\be
E(\tilde w_{\mathbf{a}}+u_0) = E(\tilde w_{\mathbf{a}})\qquad \hbox{with \,\,$u_0= a_2e^{\beta^+_2s}\varphi_2(\theta)= a_2 m_2\sin\theta$}.
\ee 
This implies that $u_1=   H_{2,  R_{\rho_1,\mathbf{a}'}} (E(\tilde w_{\mathbf{a}}+u_0) -E(\tilde w_{\mathbf{a}}))=0$, and moreover we have 
\be
\tilde w_{\mathbf{a}'}=\tilde w_{\mathbf{a}} +a_2 m_2\sin\theta .
\ee
This combined with \eqref{eq-barww0} yields that
\be
  S_{\mathbf{a}'}(l,\theta)- \left[ S_{\mathbf{a}} (l,\theta)+a_2 m_2\sin\theta \right] = o(l^{-2\alpha})
\ee
since $-1/2< -2\alpha=\beta_0^+$.  By Lemma \ref{lem-uniqueness}  below, it follows that
\be
  S_{(\mathbf{0}_{2}, a_2, \mathbf{a}_{K-3}) }=   S_{(\mathbf{0}_{3},  \mathbf{a}_{K-3}) } +a_2 m_2\sin\theta  .
\ee
Arguing similarly with   $ e^{\beta^+_1 s }\varphi_1( \theta)= m_1 \cos \theta$, we obtain 
\be
  S_{(0,a_1, a_2, \mathbf{a}_{K-3}) }=   S_{(\mathbf{0}_{3},  \mathbf{a}_{K-3}) } +a_1 m_1\cos\theta  +a_2 m_2\sin\theta  .
\ee
\smallskip

For the first parameter, let $S:=S_{( 0 , \mathbf{a}_{K-1})}$ with $\mathbf{a}_{K-1}\in \mathbb{R}^{K-1}$ and let $a_0\in \mathbb{R} $.
Then we have
\be
S(l+a_0m_0, \theta)=S(l,\theta)+a_0m_0 l^{-2\alpha}h(\theta) + o(l^{-2\alpha})
\ee
in light of (i) since $S(l,\theta)= \sigma^{-1}l^\sigma h(\theta)+ o(l^\sigma)$ with $\sigma=1-2\alpha$. 
Moreover, it holds from  (ii) that
\be
  S_{(a_0 ,  \mathbf{a}_{K-1}) }=  S_{(0,   \mathbf{a}_{K-1}) } +  a_0m_0 l^{-2\alpha}h(\theta) + 
 o(l^{-2\alpha}).
\ee
Then by  using again Lemma \ref{lem-uniqueness}  below, we deduce that
\be
  S_{(a_0 ,  \mathbf{a}_{K-1}) } (l,\theta)=  S_{( 0 ,  \mathbf{a}_{K-1}) }(l+a_0m_0, \theta) .
\ee
Therefore, we conclude that the first three parameters correspond to translations  along the   $x_3$, $x_1$, $x_2$-axes:
\be
  \Sigma_{(a_0,a_1, a_2, \mathbf{a}_{K-3}) }=    \Sigma_{(\mathbf{0}_{3},  \mathbf{a}_{K-3}) }  +a_1 m_1 \mathbf{e}_{1}  +a_2 m_2\mathbf{e}_{2} -a_0m_0\mathbf{e}_{3}.
\ee

\bigskip

Now it remains to show (iv). Let $\{\mathbf{a}_{i} \} $  be a sequence   converging to $\mathbf{a}$ in $ \mathbb{R}^K$. Let us denote   by $S_{\mathbf{a}_i }^{\pm} (l,\theta)$ the support functions of the barriers  at the level $x_3=l$ in Proposition \ref{prop-global-barriers}. By Proposition \ref{thm-modulicont-barr}, there are constants  $C>1$ and $ R>1 $ such that 
 \begin{equation}\label{eq-above}
 \|S_{\mathbf{a}_i }^{\pm} (l,\cdot)-S_{\mathbf{a} }^{\pm}(l,\cdot)\|_{L^\infty(\mathbb{S}^1)}\le C |\mathbf{a}_i-\mathbf{a}|  \, l^{\sgamma}\qquad \forall l\geq e^R
\end{equation}
for any $i\geq1$. 
 In view of  the convergence of barriers on the exterior domain in  \eqref{eq-above} and the comparison principle applied on its complement, a compact region,  we have that for every subsequence $\{i_n\}$, 
 there is a further subsequence $\{i_m\}$ such that $\Sigma_{\mathbf{a}_{i_m}}$ converges  to a translator $\Sigma'$ in the $C^0_{\text{loc}}$-convergence topology. 
We also notice that the support functions of $ \Sigma'$ and $\Sigma_{\mathbf{a}}$ are bounded above and below by $S_{\mathbf{a} }^-(l,\theta)$ and $S_{\mathbf{a} }^+(l,\theta)$, respectively. 
Since $\|S_{\mathbf{a}}^+(l,\cdot)-S_{\mathbf{a}}^-(l,\cdot)\|_{L^\infty(\mathbb{S}^1)}=o(l^{-2\alpha})$ by Proposition \ref{prop-global-barriers}, it follows from Lemma \ref{lem-uniqueness} that $\Sigma'=\Sigma_{\mathbf{a}}$. Therefore,  the original sequence $\Sigma_{\mathbf{a}_i}$ converges to $\Sigma_{\mathbf{a}}$ in the $C^0_{\text{loc}}$-convergence topology. This completes the proof of Theorem \ref{thm-existence}. 
\end{proof}

\bigskip

 Theorem \ref{thm-existencerevised11} and \eqref{eq-barww0} allow us to decompose each $w_{\mathbf{a}}$ in terms of Jacobi fields $a_j \varphi_j e^{\beta^+_js}$ and corrector terms. The following corollary is a quantitative statement of $(i)$ and $(ii)$ in   Theorem \ref{thm-existence}, where $f_{\mathbf{a}}:=w_{\mathbf{a}}-\tilde w_{\mathbf{a}}$  denotes the difference between   the entire translator and the  exterior translator.  

\begin{corollary}\label{cor-wexpression} There are  constants $\rho <\infty $ and $\varepsilon>0$ (depending  on $\alpha $ and  $h $) such that %on $s\ge R_{\mathbf{a}} :=\rho +\ln (|\mathbf{b}(\mathbf{a})|+1)$ 
the translator $\Sigma_{\mathbf{a}}$ satisfies the estimate  \be 
\Vert  w_{\mathbf{a}} \Vert_{C^{2,\beta,  \sigma -\eps  }_{R_{\rho,\mathbf{a}}} } \le       e^{\eps  R_{\rho, \mathbf{a}}} .
\ee
Moreover, there are families of functions $\{ g_{\mathbf{a}} \}_{\mathbf{a}\in\mathbb{R}^K}$ and $\{f_{\mathbf{a}}\}_{\mathbf{a} \in \mathbb{R}^K}$ {defined on on $s\ge R_{\rho,\mathbf{a}}$} such that %  $w_{\mathbf{a}}$  can be decomposed into  
 \begin{equation} \label{eq-wexpression-0} 
 w_{\mathbf{a}}= g_{\mathbf{0}}+\sum_{j=0}^{K-1}  a_j\big(  \varphi_j e^{\beta^+_j s} + g_{\mathbf{a}_j} \big) +f_\mathbf{a}\qquad \hbox{on $s\ge R_{\rho,\mathbf{a}}$}.
 \end{equation}
Here, $\mathbf{a}_j$ denotes $(  \mathbf{0}_{j}, a_j,\ldots, a_{K-1})$, the projection of $\mathbf{a}\in\mathbb{R}^K$ onto $\{\mathbf{0}_j\}\times \mathbb{R}^{K-j}$.
The families $\{g_{\mathbf{a}}\}$ and $\{f_{\mathbf{a}}\}$ have the following estimates: 
\begin{enumerate}[(i)] 
% \item   $g_{\mathbf{a}}$ and $f_{\mathbf{a}}$ are defined on $s\ge R_{\mathbf{a}}$, 
 
\item $\Vert f_\mathbf{a} \Vert _{C^{2,\beta,-\frac{1}{2} }_{R_{\rho,\mathbf{a}}}} \le {\frac{1}{2}  e^{(\sigma +{1}/{2} )R_{\rho,\mathbf{a}}}}$ ,

\item  $\| g_{\mathbf{0}} \|_{C^{2,\beta,\sigma -\eps}_{R_{\rho,\mathbf{a}}}}\ \le    e^{\eps  R_{\rho,\mathbf{a}}}$, and $\| g_{\mathbf{a}} \|_{C^{2,\beta,\beta^+_{j}-\eps}_{R_{\rho,\mathbf{a}}}}\ \le    e^{\eps  R_{\mathbf{a}}}$ for every  $\mathbf{a}=(\mathbf{0}_{j},a_j,\ldots,a_{K-1})$ with $a_j \neq 0$.

\end{enumerate} 

\end{corollary}

\medskip

\begin{lemma} \label{lem-uniqueness} 
For each $i=1,2$,  let  $\Sigma_i\subset\mathbb{R}^3$ be a convex graph $x_{n+1}=u_i(x)$, where $u_i(x)$   is     an Alexandrov solution  to the equation 
 \be\label{15606}
 \det D^2 u = (\eta+|Du|^2)^{2-\frac1{2\alpha}}\qquad\hbox{on\,\, $\mathbb{R}^2$}  
 \ee 
 with a   fixed constant $\eta \in \{0,1\}$. 
Let  $S_i(l,\theta)$  be the representation of $\Sigma_i$  given by \eqref{eq-supportfunction}   such  that  
  \bea \label{eq-S-asymp}
 S_i (l,\theta)=  \sigma^{-1} l^{\sigma } h(\theta)  + o(l^{\sigma})   
 \eea   
 and  
\bea \label{eq-uniq-asymp}
   S_1(l,\theta)-S_2(l,\theta) = o(l^{-2\alpha}) 
\eea 
  uniformly for  $ \theta\in \mathbb{S}^1$ as $l\to\infty $. Then,  $\Sigma_1\equiv\Sigma_2$. 
 
 \begin{proof} 
  Observe that for each fixed $\theta\in\mathbb{S}^1$,  $S(\cdot,\theta)$ is a concave function of $l$. %{\color{blue} Using its   concavity and \eqref{eq-S-asymp},  we have } 
Using the monotonicity of $l\mapsto \partial _l S(l,\theta)$, 
 we have the asymptotics of gradient: 
  \be 
 \partial_l S_i(l,\theta) - l ^{\sigma-1} h(\theta ) =o(l^{\sigma-1})
 \ee   uniformly for  $ \theta\in \mathbb{S}^1$ as $l\to\infty $. 
 %$\Vert \partial_l S_i(l,\cdot) - l ^{\sigma-1} h(\cdot)\Vert_{L^\infty(\mathbb{S}^1)} =o(l^{\sigma-1})$ as $l\to \infty$.  
 Then for  each $\delta>0$, we obtain 
 \bea 
 S_1(l-\delta ,\theta) = S_1(l,\theta) -\delta  l^{\sigma-1 }h(\theta)+o(l^{\sigma-1 })
 \eea
  uniformly for  $ \theta\in \mathbb{S}^1$ as $l\to\infty $, 
 and hence it follows from \eqref{eq-uniq-asymp} that 
 \be
 S_1(l-\delta ,\theta)-S_2(l,\theta)=-\delta  l^{\sigma-1 } h(\theta) + o(l^{\sigma-1 }) 
 \ee
  uniformly for  $ \theta\in \mathbb{S}^1$ as $l\to\infty $. 
This observation shows that   each  level curve (the set of intersection with $\{x_3=l\}$) of $\Sigma_2 $ encloses that of $\Sigma_1+\delta \mathbf{ e}_3$ for large $l$. By viewing $\Sigma_1$ and $\Sigma_2$ as     graphs of    convex  Alexandrov solutions to \eqref{15606}
 with a  common $\eta \in\{0,1\}$, the comparison principle implies that  $\Sigma_1+\delta  \mathbf{e}_3$ is contained in the convex hull of $\Sigma_2$.  Similarly, if $ \delta<0$, then we conclude the convex hull of $\Sigma_1+\delta \mathbf{e}_3$ contains $\Sigma_2$. By letting $\delta \to 0$, we conclude $\Sigma_1\equiv\Sigma_2$. 
 \end{proof}

\end{lemma}

\medskip

\begin{remark}[Entire solutions to blow-down equation]\label{remarkno1}

Following the same construction used for translating surfaces $\det D^2 u =(1+|Du|^2)^{2- \frac{1}{2\alpha}}$ with few changes, there holds a corresponding existence theorem for the blow-down equation  $\det D^2 u = |Du|^{4-\frac1\alpha}$. The barrier construction is basically the same except we work with the main equation \eqref{eq-w29} with $\eta=0$, making the problem even easier. Here the construction of reference solution corresponding to the zero parameter in the proof of Theorem \ref{thm-existencerevised11} Step (i) is not needed: we simply choose it as the homogeneous solution, corresponding to  $\tilde w_{\mathbf{0}}(s,\theta)\equiv 0$. In the construction of global barrier in Proposition \ref{prop-global-barriers}, we do not need to work with super and subsolutions of form \eqref{eq-defUpm} in the inner region: the choice of homogeneous solution $U_+(l,\theta)=U_-(l,\theta)= \sigma^{-1} l ^\sigma h(\theta)$ simply works for the blow-down equation. Unlike {\bf Step 3} in the original proof of Proposition \ref{prop-global-barriers}, we do not need to check the graphical representation of global barriers, say $u_{+}$ and $u_{-}$, are viscosity super and subsolution at the origin, respectively. Finally, it should be noted that the resulting solution $u$ made by following the proof of Theorem \ref{thm-existence} is not smooth everywhere. Working in the dual equation, we obtain a global solution $v_{\infty}$ which is a locally uniform limit of $v_{L_i}$, viscosity solutions to $\det D^2 v = |y|^{\frac1\alpha -4}$. The limit $v_\infty$ is smooth away from the origin due to the degeneracy of equation on the right hand side. Indeed, \cite[Theorem 1.1]{DSavin} showed a sharp result that the solution is $C^{2,\delta}$ around the origin. For this reason. The dual of $v_{\infty}$, say $u_{\infty}$, is smooth away from a unique point $x_0\in\mathbb{R}^2$ such that $Du_{\infty}(x_0)=0$. 
\end{remark}

\appendix

\section{Support functions of level curves} \label{section-appendix-1}

  \begin{proof}[Proof of Proposition \ref{prop-supportfunction}]
 
 Let $x=x(l,\theta)$ be the point  which gives $S(l,\theta)$, namely $S(l,\theta)=\la x(l,\theta) , e^{i{\theta}} \ra $. In what follows, we view $x$ as a function of $(l,\theta)$. We first  have  $  \la x_\theta  , e^{i\theta} \ra =0 $ since $x(l,\theta)$ gives the maximum on the right-hand side of \eqref{eq-supportfunction}. Moreover, there holds 
 \bea S_\theta = \la x_\theta , e^{i\theta}\ra + \la x, ie^{i\theta}\ra = \la x,ie^{i\theta} \ra,  \eea 
\bea S_{\theta\theta} = \la x_\theta, ie^{i\theta} \ra + \la x, -e^{i\theta} \ra = \la x_\theta, ie^{i\theta} \ra -S  , \eea  
 \bea S_l = \la x_l, e^{i\theta} \ra, \eea
 and
 \bea S_{l\theta} = \la x_{l\theta} , e^{i\theta} \ra + \la x_l, ie^{i\theta} \ra =\la x_{\theta l} , e^{i\theta} \ra + \la x_l, ie^{i\theta} \ra = \la x_l,ie^{i\theta} \ra .   \eea 
 Using the computations above, we obtain 
 \bea\label{eq-ja1} x_\theta = ( S_{\theta\theta}+S) ie^{i\theta}\quad \text{ and } \quad x_l = S_l e^{i\theta}+ S_{l\theta} ie^{i\theta}. \eea

On the other hand, using the fact that $u(x(l,\theta)) = l$, 
it holds that 
\bea \label{eq-gradf2} \la  D u  , x_l \ra =1  \quad \text{ and }\quad \la Du , x_\theta \ra =0, \eea
which implies at each $x=x(l,\theta)$,
\bea \label{eq-gradf}
D u = \frac1{S_l} e^{i\theta}.
\eea  
Note that  $ S_{\theta\theta}+S>0$ by the strict convexity of $u$ and $S_l>0$ since $Du\neq0$. Hence $(l,\theta)$ provides a local chart of $\{x\in\mathbb{R}^2 \,:\, |Du (x)| \neq 0  \}$. 

Utilizing  \eqref{eq-gradf}, we obtain that 
 \bea \label{eq-hessf}
   x_l \cdot D^2u  = \frac{-S_{ll}}{S_l^2 } e^{i\theta} \quad \text{ and } \quad     x_\theta \cdot D^2u = - \frac{S_{l\theta}}{S_l^2 } e^{i\theta} +\frac{1}{S_l}ie^{i\theta}.
 \eea  
 Here $\cdot$ is a matrix multiplication by viewing $D^2u$ as a square matrix and  $x_l$ and $x_\theta$ as  row matrices. 
 
For the purpose of the tensor calculus, we consider the local coordinate system $(y_1,y_2)=(l,\theta)$ so that we can denote 
\begin{align}
&x_1=\tfrac{\partial}{\partial y_1}   x =\tfrac{\partial}{\partial l}   x  =x_l, && x_2=\tfrac{\partial}{\partial y_2}   x =\tfrac{\partial}{\partial \theta}   x  =x_\theta.
\end{align}
Then, we can deduce the metric tensor
 \bea   g_{ij}:= \la x_i,x_j\ra = \begin{bmatrix} S_l^2 +S_{l\theta}^2  & S_{l\theta}(  S_{\theta\theta}+S) \\ S_{l\theta}(  S_{\theta\theta}+S) & (  S_{\theta\theta}+S)^2 \end{bmatrix},\eea 
and the  Christoffel symbols
 \begin{equation}
 \Gamma^k_{ij}:=\tfrac12 g^{kl}\left(\partial_i g_{jl}+\partial_j g_{il}-\partial_l g_{ij} \right)=g^{kl}\langle x_{ij},x_l\rangle.
 \end{equation}
 Hence, on this coordinate system we can obtain
\begin{equation}
\overline{\nabla}_i\overline{\nabla}_j u = \frac{\partial^2u}{\partial y_i\partial y_j}-\Gamma_{ij}^k\frac{\partial u}{\partial y_k} =D^2u(x_i,x_j)+Du\cdot x_{ij}-g^{kl}\langle Du,x_k\rangle \langle x_{ij},x_l\rangle=\langle x_i, D^2u\cdot x_j\rangle ,
\end{equation}
where $\overline{\nabla}$ denotes the Levi-Civita connection. Therefore, combining \eqref{eq-hessf} and \eqref{eq-ja1} yields 
 \bea \label{eq:Hessian_connection}
 \overline{\nabla}_i\overline{\nabla}_j u = \begin{bmatrix}  -\frac{S_{ll}}{S_l} & 0 \\ 0 & \frac{  S_{\theta\theta}+S}{S_l}\end{bmatrix}.\eea 
Now, we recall that
  \bea  
 \det D^2u = \det\left[g^{ij} \overline{\nabla}_j\overline{\nabla}_ku \right ] =[\det \overline{\nabla}_i\overline{\nabla}_j u] [\det g_{kl}]^{-1}.  
 \eea   Hence,  it follows that 
 \bea \label{eq-dethessf}
 \det D^2u = \frac{-S_{ll}(  S_{\theta\theta}+S)}{S_l^2} \frac{1}{S_l^2 (  S_{\theta\theta}+S)^2} =\frac{-S_{ll}}{S_l^4 (  S_{\theta\theta}+S)}. \eea 
Therefore,  in light of   \eqref{eq-gradf} and \eqref{eq-dethessf}, the translator equation \eqref{eq-translatorgraph} becomes 
 \bea 
 \frac{-S_{ll}}{S_l^4(  S_{\theta\theta}+S)} = (1+S_l^{-2})^{2-\frac1{2\alpha}} .
 \eea Similarly, we obtain the blow-down equation for $S(l,\theta)$ and this finishes the proof. 
 \end{proof}

 For later purpose, let us record derivatives of $u(x)$ in terms of the derivatives of $S$. Here, $Du$ is represented as a row matrix.  
\begin{corollary} At a given point $x\in \mathbb{R}^2$, suppose $Du(x)\neq 0$ and $Du/|Du|=(\cos \theta,\sin \theta)$. If the rotation matrix is denoted by $$R_{\theta} = \begin{bmatrix}  +\cos \theta   & -\sin\theta  \\ +\sin{\theta}   & +\cos {\theta}  \end{bmatrix}, $$ then the first and second order partial derivatives of $u:\mathbb{R}^2\to \mathbb{R}$ at $x$   are 
\[D u = \begin{bmatrix}  S_l^{-1} &   0 \end{bmatrix} R_{\theta}^{\,T}\]
and 
\[D^2 u =  \frac{1}{S_l^3 (S+S_{\theta\theta})}R_{\theta}\begin{bmatrix}  S_{l\theta}^2 -S_{ll}(S+S_{\theta\theta}) & - S_l S_{l\theta} \\ - S_l S_{l\theta} & S_{l}^2 \end{bmatrix}R_{\theta}^{\, T} .\]
\begin{proof} The expression for $Du$ is immediate from \eqref{eq-gradf}, and    $D^2u$ follows from \eqref{eq-hessf} and \eqref{eq-ja1}.

\end{proof}

\end{corollary}

 \begin {lemma} \label{lem-w-eq}
 Let $S(l,\theta)$ be a solution to \eqref{Se} with $\eta=0$ or $1$  and let 
  $w$ be a function given by  \eqref{eq-w}. 
Then  $w$   solves 
 \bea \label{eq-w'} 
 w_{ss}+w_s+{h}^{\frac1\alpha}( w_{\theta\theta}+w)   =E(w)=-E_1(w) -\eta E_2(w), 
 \eea 
 where 
  \bea\label{eq-Ew} 
 E_1(w):=  \frac1\alpha h^{\frac1\alpha-1} {e^{- \sigma s}}  {w_s}( w_{\theta\theta}+w) + ( 2\alpha  e^{\sigma s}  h^{1-\frac{1}{\alpha}}  +w_{\theta\theta}+   w )N_1(w). \eea 
 with
 \begin{equation}
   N_1(w):=  \left (h+{ e^{-\sigma s}}{w_s} \right)^{\frac1\alpha} -h^{\frac1\alpha} -\frac1{\alpha} h^{\frac1\alpha-1}e^{-\sigma s}{w_s},
 \end{equation}
 and
 \bea \label{eq-Ew2}
&E_2(w):=\left( h+e^{-\sigma s}w_s \right)^{\frac1\alpha} ( 2\alpha  e^{\sigma s}  h^{1-\frac{1}{\alpha}}  +w_{\theta\theta}+   w )N_2(w).
\eea
with
\begin{equation}
     N_2(w):= \left (1+ e^{-4\alpha s} \left(  h+e^{-\sigma s}w_s \right)^2\right)^{2-\frac{1}{2\alpha}}-1.
\end{equation}
 
% \bea\label{eq-Ew} 
% E(w):=&-  \frac1\alpha h^{\frac1\alpha-1} {e^{- \sigma s}}  {w_s}( w_{\theta\theta}+w)  - \left[\left (h+{ e^{-\sigma s}}{w_s} \right)^{\frac1\alpha} -h^{\frac1\alpha} -\frac1{\alpha} h^{\frac1\alpha-1}e^{-\sigma s}{w_s} \right] (S_{\theta\theta}+S)\\
%&-\left(e^{-\sigma s} {S_s} \right)^{\frac1\alpha} (S_{\theta\theta}+S) \left [ \left (1+ e^{-4\alpha s} \left(e^{-\sigma s}{S_s} \right)^2\right)^{2-\frac{1}{2\alpha}}-1  \right ].
%\eea
 \begin{proof} %\marginpar{{\color{red} We need to compute and check again the correctness of following computation.}}
 From \eqref{Se}, remembering $\eta \in \{0,1\}$ it follows that 
\bea \label{eq-A28}  -S_{ss} &= -S_s +\left(e^{-\sigma s}{S_s} \right)^{\frac1\alpha} (S_{\theta\theta}+S)  +\eta N_2\left(e^{-\sigma s}{S_s} \right)^{\frac1\alpha} (S_{\theta\theta}+S). \eea
Since $ {\sigma}^{-1} e^{\sigma s} h(\theta)$  solves the blow-down equation (i.e. \eqref{eq-A28} with $\eta=0$), we obtain that $w(s,\theta)= S(e^s,\theta)- {\sigma}^{-1} e^{\sigma s} h(\theta)$ solves
 \bea 
 -w_{ss} =  {h}^{\frac1\alpha}(w_{\theta\theta}+w) +w_s +  \frac1\alpha  h^{\frac1\alpha-1}e^{-\sigma s}{w_s} (w_{\theta\theta}+w)   +    N_1 (S_{\theta\theta}+S)+\eta N_2\left(e^{-\sigma s}{S_s} \right)^{\frac1\alpha} (S_{\theta\theta}+S).\eea 
This completes the proof.
 \end{proof}

 \end{lemma}

 %In the next proposition, we show that as if a strictly convex function $u$ defines the support function $S(l,\theta)$, $S(l,\theta)$ can conversely define a strictly convex function $u$ whose support function at level $l$ is equal to $S(l,\theta)$. This will be needed as we construct our translators by solving \eqref{eq-translatorsupport}.  

 \begin{proposition}\label{prop:graphicality_S}
Let $S:(l_1,l_2) \times \mathbb{S}^1\to \mathbb{R}$ be a $C^2$-function   that satisfies  
 \begin{align}
& S+S_{\theta\theta}>0, && S_l>0, && S_{ll}<0.
 \end{align}
 Then, there exist nested convex sets $\Omega \subset \subset \Omega '$ in $\mathbb{R}^2$ and a strictly convex function $u: \mathrm{int}\, \Omega'\setminus \mathrm{cl}\, \Omega\to \mathbb{R}$ of class $C^2$ such that $S$ is the level set support function of graph of $u$.
 \begin{proof}  For each fixed $l\in(l_1,l_2)$, $S(l,\cdot)$ can be viewed as the support function of a closed curve $\Gamma_l$. $S+S_{\theta\theta}>0 $ implies the curve $\Gamma_l$ is strictly convex and $\theta \mapsto X(l,\theta)= S(l,\theta) e^{i\theta} + S_\theta (l,\theta) ie^{i\theta} $ is an embedding of $\mathbb{S}^1$ into $\mathbb{R}^2$ with the image  $X(\{l\}\times \mathbb{S}^1)=\Gamma_l$. Moreover, $S_l>0$ implies that $\Gamma_l$ is strictly increasing. i.e. for $l>l'$, $\Gamma_{l'}$ is included in the interior of the convex hull of $\Gamma_{l}$. As a consequence, there are two closed convex sets $\Omega \subset \subset \Omega'$ in $\mathbb{R}^2$ such that
 $\cup_{l \in (l_1,l_2)} \Gamma_l =  \mathrm{int}\, \Omega'\setminus \mathrm{cl}\, \Omega $  and $X:(l_1,l_2) \times \mathbb{S}^1 \to  \mathrm{int}\, \Omega'\setminus \mathrm{cl}\, \Omega$ is a bijection.  For each $x \in \mathrm{int}\, \Omega'\setminus \mathrm{cl}\, \Omega $ there is a unique value $u(x)\ge l_0$ such that $x\in \Gamma_{u(x)}$. Next, we prove $X$ is a $C^1$-diffeomorphism and $u$ is $C^1$-function. Observe $X_l = S_l e^{i\theta}+S_{\theta l }i e^{i\theta}$ and $X_\theta = (S+S_{\theta\theta})ie^{i\theta}$. From this, we infer that the Jacobian determinant of $X=(x_1,x_2)$ is
 \[\frac {\partial {(x_1,x_2)}}{\partial (l,\theta) }= S_l (S+S_{\theta\theta}) ,\]
which is non-zero  everywhere by the assumption. By the inverse function theorem,  $X$ is a local $C^1$-diffeomorphism and this upgrades to a global one as $X$ is a bijection. Note $u$ is also $C^1$ as $u(x_1,x_2)= l(x_1,x_2)$, the first component of $X^{-1}$.

  	By the construction, $u$ is a $C^1$-function on $\mathrm{int}\, \Omega'\setminus \mathrm{cl}\, \Omega$ and the support function at the level $\{u=l\}$,  is equal to the given function $\theta \mapsto S(l,\theta)$. Under this assumption (although a priori $u$ is not smooth and strictly convex), we may exactly follow the argument in the proof Proposition \ref{prop-supportfunction} upto \eqref{eq-gradf} and conclude $Du = S_l^{-1} e^{i\theta}$. Since $S_l(l,\theta)$ and $X^{-1}(x_1,x_2)=(l,\theta)$ are both $C^1$, $Du$ has continuous partials and thus $u$ is of class $C^2$. Then we  follow the rest of argument from \eqref{eq-hessf} to \eqref{eq:Hessian_connection} and conclude that $u$ is strictly convex. Note $S_{ll}<0$ is used to conclude \eqref{eq:Hessian_connection} is positive definite. 
 \end{proof}

 \end{proposition}

\section{Interior regularity estimate}

In this section, we establish interior estimates. For readers' convenience, we recall the abbreviated notation of norms in Definition \ref{def:norm.abbr} that
\begin{align}
&\|f\|_{C^{k,\beta}_{s,a}}:=\|f\|_{C^{k,\beta}([s-a,s+a]\times\mathbb{S}^1)}, && \|f\|_{L^{2}_{s,a}}:=\|f\|_{L^{2}([s-a,s+a]\times\mathbb{S}^1)}.
\end{align}
Also, we use the abbreviations that $S_{\lambda ,l}=\frac{\partial}{\partial l} S_\lambda , S_{\lambda ,ll}=\frac{\partial^2}{\partial l^2} S_\lambda  ,S_{\lambda ,\theta\theta}:=\frac{\partial^2}{\partial \theta^2} S_\lambda $ for  functions $S_\lambda(l,\theta)$ parametrized by $\lambda \in [0,1]$. 

\begin{lemma}\label{lem-B1} Given a solution $h\in C^\infty(\mathbb{S}^1)$ to \eqref{eq-shrinker}, $\alpha\in(0,1/4)$, $\eta \in [0,1]$, $k\geq 0$, and $\beta\in (0,1)$ , there are large $s_0$, $C$ and small $\varepsilon_0>0$ with the following significance.  For each $i \in \{0,1\}$, suppose $S_i(l,\theta) := \frac{l^{1-2\alpha}}{1-2\alpha} h(\theta) + w_i(s,\theta)$ with $s=\ln l$ and $w_i(s,\theta)\in C^\infty([s_0-2,+\infty)\times \mathbb{S}^1)$ satisfies $\| w_i \|_{C^{k,\beta}_{s,2}}\leq \varepsilon_0 e^{(1-2\alpha)s}$ for $s \geq s_0$ and
\begin{align}\label{eq-B1}
&S_i+S_{i,\theta\theta}+(\eta +S_{i,l}^{-2})^{\frac{1}{2\alpha}-2}S_{i,l}^{-4}S_{i,ll} =0, && S_{i,l}>0,
\end{align}
for $
\ln l \geq s_0-2$. Then the following estimates hold: for $s\geq s_0$, 
\begin{align}
\|w_1\|_{C^{k,\beta}_{s,1}} &\leq C (\| w_1\|_{L^2_{s,2}} + \eta e^{(1-2\alpha)s }e^{-4\alpha s}), \label{eq-B2}\\ 
\| w_1-w_0\|_{C^{k,\beta }_{s,1}} &\leq C \|w_1-w_0\|_{L^2_{s,2}}.\label{eq-B3} 
\end{align}
\end{lemma}

\begin{proof}
First, we prove \eqref{eq-B3}. The difference $W(l,\theta):=S_1(l,\theta)-S_0(l,\theta)=w_1(\ln l,\theta)-w_0(\ln l,\theta)$ solves the linear equation $ W +W_{\theta\theta}+A_\eta W_{ll} +B_\eta W_l =0$, where the coefficients $A_\eta(l,\theta)$ and $B_\eta(l,\theta)$ are given by
	\begin{align}
	A_\eta&:=   (\eta +  S_{1,l}^{-2})^{\frac1{2\alpha}-2} S_{1,l}^{-4} ,\\
	 B_\eta&:= S_{0,ll}\int_0^1 \left[(4-\alpha^{-1} ) (\eta +  S_{\lambda,l}^{-2} )^{\frac1{2\alpha}-3} S_{\lambda,l}^{-7} -4(\eta +  S_{\lambda,l}^{-2} )^{\frac1{2\alpha}-2} S_{\lambda,l} ^{-5}\right ] d\lambda,
\end{align}
with  $S_{\lambda }(l,\theta):=\lambda  S_1(l,\theta)+  (1-\lambda)S_0(l,\theta)$ for $\lambda \in [0,1]$. Namely, $W(l,\theta)=W(e^s,\theta)$ satisfies 
\begin{equation}
 \label{eq-b5} W +W_{\theta\theta}+\hat A_\eta W_{ss} +\hat B_\eta W_s =0,
\end{equation}	 
where $\hat A_\eta := e^{-2s}A_\eta $ and $\hat B_\eta:= e^{-s}B_\eta -e^{-2s}A_\eta$. Since
\begin{align}
S_{\lambda,l}&=l^{-2\alpha}h+l^{-1}(\lambda \partial_s w_1+(1-\lambda)\partial_s w_0),\\
S_{\lambda,ll}&=-2\alpha l^{-1-2\alpha}h-l^{-2}(\lambda \partial_s w_1+(1-\lambda)\partial_s w_0)+l^{-2}(\lambda \partial_s^2 w_1+(1-\lambda)\partial_s^2 w_0),
\end{align}
assuming $\varepsilon_0\ll h$ we have
\begin{align}\label{eq:S_derivative_est}
&\|e^{-2\alpha s} S_{\lambda,1}^{-1}- h^{-1}\|_{C^{k-1,\beta}_{s,2}}\leq C\varepsilon_0  , &&\|e^{(2\alpha+1)s}S_{\lambda,ll}+ h\|_{C^{k-2,\beta}_{s,2}}\leq C\varepsilon_0,
\end{align}
for some constant $C=C(\alpha,h,k,\beta)$. Hence, for any small $\delta>0$, there are large $s_0$ and small $\varepsilon_0$ such that
\begin{align}
 \|\hat  A_\eta -h^{-\frac1\alpha}\|_{C^{k-1,\beta}_{s,2}} + \|\hat  B_\eta -  h^{-\frac1\alpha}\|_{C^{k-2,\beta}_{s,2}} \leq \delta
\end{align}
hold for $s\geq s_0$. Now, we apply the interior $W^{2,p}$ estimates with $p=2$ and the interior Schauder estimates to the linear elliptic equation \eqref{eq-b5}. See Theorem 9.11 and  Theorem 6.2 in \cite{gilbarg1977elliptic}. Then, \eqref{eq-B3} follows from applying the interior Schauder estimates for the derivatives of $W$.
  	  
\bigskip

To prove \eqref{eq-B2}, we may fix $S_0=(1-2\alpha)^{-1}l^{1-2\alpha}h(\theta)$ to have $W=w_1$. Since $S_0$ satisfies \eqref{eq-B1} with $\eta=0$,
\begin{align}
&w_1+w_{1,\theta\theta}+\hat A_0w_{1,ss}+\hat B_0w_{1,s}=f, && \text{where}\;f:= S_{1,l}^{-\frac{1}{\alpha} }S_{1,ll}\left[1-(\eta S_{1,l}^{2}+1)^{\frac{1}{2\alpha}-2}\right].
\end{align}
Since \eqref{eq:S_derivative_est} implies  $S_{1,l}^{-1} \approx h^{-1}e^{2\alpha s} \gg 1 \geq \eta$, observing $\|e^{s}S_{\lambda,l}- e^{(1-2\alpha)s} h\|_{C^{k-1,\beta}_{s,2}}\leq \varepsilon_0e^{(1-2\alpha)s}$, we have
\begin{equation}
\|f\|_{C^{k-2,\beta}_{s,2}}\leq C\eta e^{ (1-6\alpha)s}.
\end{equation}
Thus, we can apply the interior estimates for $w_1$ and its derivatives as above so that we obtain \eqref{eq-B2}.   
\end{proof}

\subsection*{Acknowledgments}
BC thanks KIAS for the visiting support and also thanks University of Toronto where the research was initiated. BC and KC were supported by  the National Research Foundation(NRF) grant funded by the Korea government(MSIT) (RS-2023-00219980). BC has been supported by NRF of Korea grant No. 2022R1C1C1013511, POSTECH Basic Science Research Institute grant No. 2021R1A6A1A10042944, and a POSCO Science Fellowship. KC has been supported by the KIAS Individual Grant MG078902, a POSCO Science Fellowship, and an Asian Young Scientist Fellowship. SK has been partially supported by NRF of Korea grant No. 2021R1C1C1008184. 

We are grateful to Liming Sun for fruitful discussion. We would like to thank Simon Brendle, Panagiota Daskalopoulos, Pei-Ken Hung, Ki-Ahm Lee, Christos Mantoulidis, Connor Mooney, Ovidiu Savin for their interest and insightful comments.

 \bibliography{GCF-ref.bib}
\bibliographystyle{abbrv}

\end{document}